\documentclass[english]{article} 
\usepackage[absolute,overlay]{textpos} 
\usepackage{amsthm}
\usepackage{amsmath}
\usepackage{amssymb}
\usepackage[utf8]{inputenc}
\usepackage[english]{babel}
\usepackage{graphicx}
\usepackage{geometry}
\usepackage{caption}
\usepackage{subcaption}
\usepackage{pdfpages}
\usepackage{comment}
\usepackage{bm}
\usepackage{etoolbox}
\usepackage{mathtools}
\usepackage{tikz}

\usepackage{algpseudocode}
\usepackage{algorithm}
\usepackage{mathrsfs} 
\usepackage{multirow}

\usepackage{pgfplots}
\usepackage{pgfplotstable}
\usepackage{stmaryrd} 

\usepackage{hyperref}

\usepackage[noabbrev,capitalise]{cleveref}
\crefformat{appendix}{#2#1#3}

\usetikzlibrary{fillbetween}
\usepgfplotslibrary{fillbetween}

\usetikzlibrary{shapes.misc}
\usetikzlibrary{math} 

\usetikzlibrary{arrows}

\definecolor{darkspringgreen}{rgb}{0., 0.55, 0.3}
\definecolor{dartmouthgreen}{rgb}{0.05, 0.5, 0.06}
\definecolor{etonblue}{rgb}{0.59, 0.78, 0.64}
\definecolor{airforceblue}{rgb}{0., 0.4, 0.66}
\definecolor{arylideyellow}{rgb}{0.91, 0.84, 0.42}
\definecolor{emerald}{rgb}{0.31, 0.78, 0.47}
\definecolor{uclagold}{rgb}{1.0, 0.7, 0.0}
\definecolor{cadmiumorange}{rgb}{0.93, 0.53, 0.18}

\theoremstyle{thmstyleone}
\newtheorem{theorem}{Theorem}
\newtheorem{proposition}[theorem]{Proposition}
\newtheorem{lemma}[theorem]{Lemma}
\newtheorem{corollary}[theorem]{Corollary}

\theoremstyle{thmstyletwo}

\newtheorem{remark}{Remark}

\theoremstyle{thmstylethree}

\newcommand{\lopd}[0]{\mathcal{L}_\Delta}
\newcommand{\lopdt}[0]{\mathcal{L}_{\Delta}}

\newcommand{\usol}[0]{\underline{\uvec{u}}_\Delta}

\newcommand{\tess}[0]{\mathcal{T}_h}

\newcommand{\uvec}[2][3]{\boldsymbol{#2\mkern-#1mu}\mkern#1mu}

\newcommand\norm[1]{\left\lVert#1\right\rVert}
\newcommand\abs[1]{\left\lvert#1\right\rvert}

\newcommand{\R}{\mathbb{R}}
\newcommand{\dt}{\Delta t}

\newcommand{\CFL}{\text{CFL}}

\newcommand{\undu}{\underline{\uvec{u}}}

\newcommand{\undspacetilde}{\widetilde{\underline{\uvec{\phi}}}}
\newcommand{\spacestuff}{\uvec{\phi}}
\newcommand{\undv}{\underline{\uvec{v}}}
\newcommand{\undw}{\underline{\uvec{w}}}
\newcommand{\undr}{\underline{\uvec{r}}}
\newcommand{\undz}{\underline{\uvec{z}}}

\newcommand{\usolp}[0]{\underline{\uvec{u}}_{\Delta}^{(p)}}
\newcommand{\Xp}[0]{X^{(p)}}
\newcommand{\Yp}[0]{Y^{(p)}}
\newcommand{\uex}[0]{\underline{\uvec{u}}_{ex}}
\newcommand{\uexp}[0]{\underline{\uvec{u}}_{ex}^{(p)}}

\newcommand{\embep}[0]{\mathcal{E}^{(p)}}
\newcommand{\projp}[0]{\Pi^{(p)}}
\newcommand{\trialspace}[1]{\varphi_{#1}} %phi space _i
\newcommand{\trialspaceiter}[2]{\varphi_{#1}^{#2}} %phi space _i
\newcommand{\trialtime}[1]{\psi^{#1}} %psi time ^m
\newcommand{\trialtimeiter}[2]{\psi^{#1,#2}} %psi time ^m
\newcommand{\trialspacetime}[1]{\vartheta^{#1}} %theta space time 
\newcommand{\trialspacetimeiter}[2]{\vartheta^{#1,#2}}

\newcommand{\trialspaceFVc}[1]{\lambda_{#1}} %corrector 
\newcommand{\APNPM}{{ADER-$\mathbb P_N \mathbb P_M$ }}

\newcommand{\Temp}{T}

\newcommand{\E}{\boldsymbol{E}}

\newcommand{\RIcolor}[1]{{\leavevmode\color{black} #1}}
\newcommand{\RIIcolor}[1]{{\leavevmode\color{black} #1}}

\begin{document}
\author{L. Micalizzi\footnote{Affiliation: Institute of Mathematics, University of Zurich, Winterthurerstrasse 190, Zurich, 8057, Switzerland. Email: lorenzo.micalizzi@math.uzh.ch.}, D. Torlo\footnote{Affiliation: SISSA mathLab, SISSA, via Bonomea 265, Trieste, 34136, Italy. Email: davide.torlo@sissa.it.} and W.Boscheri\footnote{Affiliation: Dipartimento di Matematica e Informatica, University of Ferrara, via Machiavelli 30, Ferrara, 44121, Italy. Email: walter.boscheri@unife.it.}}
\title{Efficient iterative arbitrary high order methods:\\ an adaptive bridge between low and high order}

\maketitle
%\tableofcontents

\abstract{We propose a new paradigm for designing efficient p-adaptive arbitrary high order methods. 
We consider arbitrary high order iterative schemes that gain one order of accuracy at each iteration and we modify them in order to match the accuracy achieved in a specific iteration with the discretization accuracy of the same iteration.
%The novel strategy is based on matching at each iteration the reconstruction polynomial degree with the accuracy reached in the iterative process, saving the computational costs that high order reconstructions carry along the low order iterations.
Apart from the computational advantage, the new modified methods allow to naturally perform p-adaptivity, stopping the iterations when appropriate conditions are met.
Moreover, the modification is very easy to be included in an existing implementation of an arbitrary high order iterative scheme and it does not ruin the possibility of parallelization, if this was achievable by the original method. 
%Finally, in the p-adaptivity process it is possible to provably preserve various structures with a very efficient \textit{a posteriori} limiter.

An application to the \RIcolor{Arbitrary DERivative (ADER)} method for hyperbolic Partial Differential Equations (PDEs) is presented here.
We explain how such framework can be interpreted as an arbitrary high order iterative scheme, by recasting it as a Deferred Correction (DeC) method,
and how to easily modify it to obtain a more efficient formulation, 
%which allows for adaptivity and structure preserving through a novel \textit{a posteriori} limiter.
in which a local \textit{a posteriori} limiter can be naturally integrated
leading to p-adaptivity and structure preserving properties. 
Finally, the novel approach is extensively tested against classical benchmarks for compressible gas dynamics to show the robustness and the computational efficiency.
}

\section{Introduction}
In recent years, the need for a very accurate description of physical phenomena in the context of advanced technological applications has determined an increasing interest towards large scale simulations. In order to reduce their enormous computational cost and to make them more accessible, several strategies have been proposed, among which: 
\begin{itemize}
\item parallelization, leading to a reduction of the computational time proportional to the number of employed processors with excellent scaling properties \cite{krais2021flexi,rabenseifner2009hybrid,rio2021massively,tsoutsanis2018improvement,tsoutsanis2016addressing,huang2020modeling,FARHAT199361};
\item structure preserving schemes, to preserve physical properties at the discrete level without excessive mesh refinements, e.g. positivity preserving schemes \cite{ciallella2022arbitrary,huang2019positivity,meister2014unconditionally,perthame1996positivity,offner2020arbitrary}, well-balanced schemes \cite{ciallella2022arbitrary,ciallella2022global,noelle2007high,cheng2019new,berberich2021high,mantri2021well,Kur,xing2006high,CaPa,gomez2023implicit}, TVD or maximum principle preserving schemes \cite{gottlieb1998total,zhang2011maximum,hajduk2021monolithic,balsara2000monotonicity}, entropy conservative/dissipating schemes \cite{gaburro2022high,abgrall2021relaxation,chen2017entropy,chenreview,ranocha2020relaxation,abgrall2022reinterpretation,friedrich2018entropyHP,kuzmin2020entropycg,kuzmin2020entropy,abgrall2019analysis,gassner2016well,wintermeyer2015entropy,winters2015comparison,offner2018stability,lukavcova2023convergence,abgrall2022convergence}; 
\item high order methods, which guarantee higher accuracy for coarser meshes and shorter computational times, on smooth problems, as they are able to catch complicated physical structures that low order methods struggle to obtain, e.g. finite element based methods \cite{abgrall2006residual,abgrall2019high,abgrall2020high,ricchiuto2010explicit,ricchiuto2015explicit,jund2007arbitrary,kuzmin2020entropycg}, finite volume methods \cite{xing2006high,balsara2000monotonicity,shu1998essentially,berberich2021high,noelle2007high,CaPa}, discontinuous Galerkin methods \cite{glaubitz2020stable,chen2017entropy,ALEMOOD1,ADERNSE,Busto2020,Maria,gaburro2022high,gaburro2021unified,gassner2016well}.
\end{itemize}
However, high order methods are, for the moment, mostly relegated to academic contexts. The main reason is given by the fact that concrete applications are characterized by shocks, which are well-known to reduce the accuracy to first order, disregarding for the formal order of accuracy. Further, in the presence of shocks, high order schemes are more subjected to instabilities.
Users are therefore comprehensibly unwilling to pay extra costs in terms of complexity of the numerical method and its implementation, if the effort is not rewarded with the initially expected advantages.
Indeed, one can observe that in the case of non-regular solutions, usually the shocks do not cover the whole computational domain but rather some lower dimensional manifolds. Therefore, a possibility to use high order methods at their best is the adoption of extra procedures to be implemented, e.g., limiters, \textit{a posteriori} correction techniques, blenders with low order schemes or adaptive strategies relying on shock detectors. %\davide{Non mi piace troppo \textit{extra tools} è un po' generico} %, introducing the need for, e.g. limiters
However, such procedures require a relevant interference with the basic implementation as they are not naturally embedded in the original method at the theoretical level and, if their introduction is not performed in a careful way, the additional cost associated with them may be comparable to the computational gain given by the high order feature.

Here, generalizing the idea introduced in \cite{loredavide}, we propose a new arbitrary high order formulation naturally allowing for order adaptivity, namely the so-called p-adaptivity. The formulation relies on an underlying arbitrary high order iterative scheme, which is easily modified in a suitable efficient way.
Arbitrary high order iterative methods are characterized by iterative procedures involving an approximated solution to a certain problem, whose order of accuracy increases by one at each iteration, converging towards the solution of a background high order scheme.
The idea is to modify the generic iteration in such a way that the order of accuracy of the discretization proper to the iteration itself matches the order of accuracy achieved in that specific iteration, hence reducing the computational cost. %needed at each time step. NB: We did not introduce time-stepping, we are in the abstract framework
%\lore{the reconstruction accuracy} matches the order of accuracy achieved in that specific iteration, hence reducing the computational cost needed at each time step.
%\lore{(NB: Per se, the novelty is not in the reduction of the iterations (this was known), but on their modification to match the order of accuracy.)}
%, hence optimizing the total number of iterations needed to achieve the formal order of accuracy.
%increase the order of the reconstruction of the approximated solution at each iteration in order to match the order of accuracy achieved in a specific iteration.
The number of iterations is chosen equal to the aimed order of accuracy, as already done in \cite{Decoriginal,minion2003semi,abgrall2020high,han2021dec,loredavide}. 
On the contrary, in other works, the iterations are stopped when a prescribed tolerance is reached \cite{titarevtoro,BTVC16,dumbser2008unified,Lagrange2D,ArepoTN}. 
This is most of the times unnecessary for explicit methods, as the accuracy of the underlying discretization is not as accurate as that tolerance.

The modification we propose in this work results in several advantages: for a fixed final order, we get a substantial drop in the computational cost with respect to the traditional approach as the low order iterations are performed with low order structures which are computationally cheaper. 
%Moreover, our framework has a natural adaptive character which can be exploited in order not to fix \textit{a priori} the final order of accuracy but increasing the number of iterations until a certain tolerance is matched. % Sta frase non ha senso, stai parlando di un progetto che non è in questo paper. Tra l'altro sta storia della tolleranza confonde con la tolleranza che già si usa per ADER.
Moreover, in this framework it is straightforward to limit the achieved order on the fly, stopping the computation at a certain iteration if specified criteria are not met. %\lore{WE SHOULD MAKE A REFERENCE TO ADAPTIVITY EVEN THOUGH WE DID NOT INVETIGATED IT (WHICH IS NOT COMPLETELY TRUE BTW). IT IS NATURAL AND WE WILL INVESTIGATE IT.}
This last aspect allows to overcome the typical drawback of \textit{a posteriori} \RIcolor{Multi-dimensional Optimal Order Detection (MOOD)} techniques \cite{CDL1,CDL2,CDL3,ALEMOOD1,ALEMOOD2,bacigaluppi2019posteriori}, in which, if the high order scheme produces a solution which is not valid according to some physical or numerical criteria, low order solutions must be recomputed with their associated computational cost after the high order solution has been already computed.

Apart from the advantages and the many possible applications, a remarkable aspect which is worth underlying is given by the fact that, if one already has an implementation of an arbitrary high order iterative scheme, then the introduction of the proposed modification is straightforward. 
Furthermore, the modification does not prevent the possible parallelization of the original code.

In this paper, we discuss the application to an ADER framework for hyperbolic PDEs, proposed originally by Toro et al.~\cite{mill,toro3}. The approach is validated on challenging benchmarks, showing the arbitrary high order character and the optimal performance in the context of adaptivity. 
Furthermore, the efficiently designed \textit{a posteriori} limiter, which drives the p-adaptivity, allows to provably preserve physical properties. In this work, we will use it to preserve the positivity of some quantities associated to the numerical solution. 
The resulting numerical schemes are of high order of accuracy with one-step time discretization and making use of general polygonal cells in space.

\RIcolor{
Summarizing, the main contributions of this work are the following: the generalization of the idea introduced in \cite{loredavide} into an abstract and rigorous framework, its application to the ADER context for hyperbolic PDEs with increasing reconstruction degree along the iterations of the space-time predictor solution, and the design of an efficient structure preserving \textit{a posteriori} limiter.
}

The work is organized as follows. 
In Section \ref{chap:Efficient}, we describe the general idea in the specific framework of DeC methods. In Section \ref{chap:ADER}, we present the ADER method for hyperbolic PDEs with Discontinuous Galerkin (DG) spatial discretization and we explain how it can be interpreted as an iterative arbitrary high order method. Section \ref{chap:NEWADER} is devoted to the description of the proposed modifications for the ADER framework to obtain new efficient p-adaptive schemes. The new methods are validated against several challenging benchmarks in Section \ref{chap:Num}, demonstrating the accuracy and the robustness of the novel approach. Moreover, the computational advantages with respect to the original formulation are experimentally shown. Finally, conclusions and future perspectives are reported in Section \ref{chap:Conc}.

\section{New efficient iterative arbitrary high order methods}\label{chap:Efficient}
%General idea playing around with DeC, ADER, Predictor corrector
Iterative arbitrary high order methods are numerical methods characterized by an iteration process converging to the solution of an underlying arbitrary high order scheme. %\lore{(, which will be referred to as limit solution) I WOULD REMOVE IT, WE JUST USE LIMIT SOLUTION IN THE NEXT SENTENCE AND NEVER AGAIN}. 
Here, we focus on particular iterative arbitrary high order methods, for which the order of accuracy with respect to the limit solution increases by one at each iteration. Examples of such methods are the DeC \cite{minion2003semi,Decoriginal,Decremi,han2021dec,abgrall2020high} and the ADER schemes \cite{Maria,gaburro2021unified,boscheri2022continuous,dumbser2005ader,dumbser2008unified}, which have been broadly used in the context of the numerical solution of hyperbolic systems of PDEs.
The current use of such methods consists in fixing an underlying high order scheme and performing the iteration process until convergence or, more efficiently, until the desired accuracy is reached, i.e., with a number of iterations exactly equal to the order of the method.

We propose here a new simple modification of the aforementioned framework which allows the computational cost of the original methods to be reduced, designing novel schemes with a natural adaptive character. 
In particular, let us consider a general iterative arbitrary high order method of order $P$, whose generic iteration is denoted by $\mathcal{M}_P$. Then, a simple sketch of the method is given by
\begin{align}\label{eq:process}
\undu^{(0)} \xrightarrow{\mathcal{M}_P} \undu^{(1)} \xrightarrow{\mathcal{M}_P} \undu^{(2)} \xrightarrow{\mathcal{M}_P} \dots \xrightarrow{\mathcal{M}_P} \undu^{(P-1)} \xrightarrow{\mathcal{M}_P} \undu^{(P)},
\end{align}
where $\undu^{(p)}$ is the result of the $p$-th iteration and $P$ iterations have been considered to achieve the optimal accuracy with the minimal number of iterations.
The proposed modification consists in replacing the generic $p$-th iteration of the method of order $P$ with the iteration $\mathcal{M}_p$, that is the iteration associated to the same method but with order $p$ which is in general cheaper but, nevertheless, accurate enough to get order $p$. The sketch of the modified method reads
\begin{align}\label{eq:process_new}
\undu^{(0)} \xrightarrow{\mathcal{M}_1} \undu^{(1)} \xrightarrow{\mathcal{M}_2} \undu^{(2)} \xrightarrow{\mathcal{M}_{3}} \dots \xrightarrow{\mathcal{M}_{P-1}} \undu^{(P-1)} \xrightarrow{\mathcal{M}_P} \undu^{(P)}.
\end{align}
The formal order of accuracy is not spoiled as $\mathcal{M}_p$, the new $p$-th iteration, is still sufficiently accurate to provide $\undu^{(p)}$ of order $p$ starting by $\undu^{(p-1)}$ of order $p-1$.
The technical details of changing the iteration structures at each iteration depend on the underlying iterative arbitrary high order method under consideration. For example, an intermediate embedding process like an interpolation may be needed to project $\undu^{(p-1)}$ onto the same space of $\undu^{(p)}$ in order to perform the $p$-th iteration.
The modification to pass from the original formulation \eqref{eq:process} to the new one \eqref{eq:process_new} is minimal, and its inclusion in an existing implementation of an iterative arbitrary high order method is straightforward.
The modified methods are in general cheaper than the original ones, as the computational cost related to $\mathcal{M}_p$ for $p<P$ is smaller than the one related to $\mathcal{M}_P$. 
Moreover, differently from what happens in the original framework, increasing the number of iterations always determines an increase in the order of accuracy without any saturation. Therefore, in principle, it is possible not to fix \textit{a priori} the final order, but instead to continue the iterations until a certain tolerance is matched. This may provide a valid strategy for engineering applications.
%\davide{We still don't know if we want to put this part.}<-However it is possible, we did it with DeC

We focus now on a particular family of iterative arbitrary high order methods, and we discuss their modification to comply with the p-adaptivity setting proposed in this work.

%DeC with proof
\subsection{DeC methods}\label{chap:DeC}
The DeC is an abstract procedure that can be exploited to design arbitrary high order iterative methods for differential problems. 
In particular, the formulation presented in \cite{Decremi} relies on the definition of two operators $\lopdt^1,\lopdt^2: X \rightarrow Y$ associated to a given problem, dependent on a parameter $\Delta$, acting between two normed vector spaces $(X,\norm{\cdot}_X)$ and $(Y,\norm{\cdot}_Y)$. 
The operator $\lopdt^2$ is a high order nonlinear implicit operator that we would like to solve, i.e., to find $\undu_\Delta \in X$ such that $\lopdt^2(\undu_\Delta)=\uvec{0}_Y$, to get a high order approximation of the solution to the original problem. 
Nevertheless, due to its implicit nature, the operator $\lopdt^2$ is difficult to be solved.
On the other hand, the operator $\lopdt^1$ is a low order explicit operator, for which it is easy to find $\widetilde{\undu}\in X$ such that $\lopdt^1(\widetilde{\undu})=\undz$ for $\undz \in Y.$ 
Due to its simplicity, it would be desirable to solve $\lopdt^1$, rather than $\lopdt^2$. However, the resulting solution would not be accurate enough for our purposes.

%The following theorem permits to approximate $\undu_\Delta$ arbitrarily well by a simple explicit iterative procedure.
The following theorem allows to approximate $\undu_\Delta$ arbitrarily well by a simple explicit iterative procedure, 
\RIIcolor{
which is much cheaper than the direct solution of the operator $\lopdt^2$.
}

\begin{theorem}[DeC]\label{th:DeC}
Let the operators $\lopdt^1$ and $\lopdt^2$ fulfill the following properties.
\begin{enumerate}
\item \textbf{Existence of a unique solution to $\lopd^2$} \\
$\exists ! \,\usol \in X$ solution of $\lopd^2$ such that $\lopd^2(\usol)=\uvec{0}_Y$;
\item \textbf{Coercivity-like property of $\lopd^1$} \\
$\exists \,\alpha_1 \geq 0$ independent of $\Delta$ s.t. 
\begin{equation}
	\norm{\lopd^1(\underline{\uvec{v}})-\lopd^1(\underline{\uvec{w}})}_Y\geq \alpha_1\norm{\underline{\uvec{v}}-\underline{\uvec{w}}}_X, ~ \forall \underline{\uvec{v}},\underline{\uvec{w}}\in X;
	\label{eq:DeC_coercivity}
\end{equation}
\item \textbf{Lipschitz-continuity-like property of $\lopd^1-\lopd^2$} \\
$\exists\, \alpha_2 \geq 0$ independent of $\Delta$ s.t. 
\begin{equation}
	\norm{\left[\lopd^1(\underline{\uvec{v}})\!-\!\lopd^2(\underline{\uvec{v}})\right]\!-\!\left[\lopd^1(\underline{\uvec{w}})\!-\!\lopd^2(\underline{\uvec{w}})\right]}_Y\!\leq \!\alpha_2 \Delta \!\norm{\underline{\uvec{v}}-\underline{\uvec{w}}}_X ,~\forall \underline{\uvec{v}},\underline{\uvec{w}}\in X.
	\label{eq:DeC_lipschitz}
\end{equation}
\end{enumerate}
Given a $\underline{\uvec{u}}^{(0)}\in X$, define recursively the sequence of vectors $\underline{\uvec{u}}^{(p)}$ as the solution of
\begin{equation}
\label{eq:DeC_iteration}
\lopd^1(\underline{\uvec{u}}^{(p)}):=\lopd^1(\underline{\uvec{u}}^{(p-1)})-\lopd^2(\underline{\uvec{u}}^{(p-1)}), \quad p\geq 1.
\end{equation}
Then, the following error estimate holds:
\begin{equation}
\label{eq:DeC_accuracy}
\norm{\underline{\uvec{u}}^{(p)}-\usol}_X \leq \left( \Delta \frac{\alpha_2}{\alpha_1} \right)^p\norm{\underline{\uvec{u}}^{(0)}-\usol}_X \quad \forall p\in \mathbb{N}. %OKKKKK
\end{equation}
\end{theorem}
\begin{proof}
Using the coercivity-like property of $\lopd^1$ \eqref{eq:DeC_coercivity}, the definition of the DeC iteration \eqref{eq:DeC_iteration}, the fact that $\lopd^2(\usol)=\uvec{0}_Y$ and the Lipschitz-continuity-like property of $\lopd^1-\lopd^2$ \eqref{eq:DeC_lipschitz}, we get
\begin{align}
\begin{split}
\norm{\underline{\uvec{u}}^{(p)}-\usol}_X & \leq \frac{1}{\alpha_1}  \norm{\lopd^1(\underline{\uvec{u}}^{(p)})-\lopd^1(\usol)}_Y\\
&= \frac{1}{\alpha_1}  \norm{\lopd^1(\underline{\uvec{u}}^{(p-1)})-\lopd^2(\underline{\uvec{u}}^{(p-1)})-\lopd^1(\usol)}_Y\\
&= \frac{1}{\alpha_1}  \norm{ \left[ \lopd^1(\underline{\uvec{u}}^{(p-1)})-\lopd^2(\underline{\uvec{u}}^{(p-1)})\right]-\left[\lopd^1(\usol)-\lopd^2(\usol)\right]}_Y\\
&\leq \Delta \frac{\alpha_2}{\alpha_1}  \norm{\underline{\uvec{u}}^{(p-1)}-\usol}_X.
\end{split}
\end{align}
Applying recursively the previous inequality, we obtain the \RIIcolor{desired result}. %thesis.
\end{proof}
For $\Delta$ small enough, the sequence of vectors $\underline{\uvec{u}}^{(p)}$ converges to $\usol$ independently of the initial vector $\underline{\uvec{u}}^{(0)}.$
At each iteration of the DeC procedure \eqref{eq:DeC_iteration}, the computation of $\underline{\uvec{u}}^{(p)}$ is straightforward by our assumptions on the operator $\lopdt^1$, since $\underline{\uvec{u}}^{(p-1)}$ is known and it is possible to explicitly compute the right hand side. Furthermore, thanks to the accuracy estimate \eqref{eq:DeC_accuracy}, at each iteration one order of accuracy is gained with respect to $\usol$.
\begin{remark}{(On ``over-resolving'' the operator $\lopdt^2$ and on the number of iterations $P$)}\label{rmk:P}
Usually, we are not strictly interested in the solution $\usol$ of the operator $\lopdt^2$, but rather on the analytical solution $\undu_{ex}$ of the underlying problem.
If $S$ is the order of accuracy of $\usol$, in general, it suffices to approximate $\usol$ with $S$-th order accuracy.
This consideration allows to bound the number of iterations, saving computational time, without necessarily getting convergence towards the solution to the operator $\lopdt^2$.
In particular, thanks to Theorem \ref{th:DeC}, if $\underline{\uvec{u}}^{(0)}$ is an $O(\Delta)$-approximation of $\undu_{ex}$, then $\norm{\undu^{(P)}-\undu_{ex}} = O(\Delta^{1+\min{(P,S)}})$ leading to an order of accuracy equal to $\min{(P,S)}$.
Hence, the optimal choice is $P=S$. Any further iteration will not increase the order of accuracy of the method with respect to $\uex$ but only with respect to $\usol$.
%LM: Should I add the triangular inequality here?  DT: No
\end{remark}
\begin{remark}{(On explicit and implicit DeC methods)}
\RIIcolor{The DeC philosophy is based on having a simple iterative procedure allowing to obtain a high order approximation of the solution of a given problem, which would have been difficult to compute directly.
In the applications, we will focus on an explicit setting by considering explicit operators $\lopdt^1$, leading to an update formula \eqref{eq:DeC_iteration} with an explicit character.
However, within the described framework, one could easily switch to an implicit setting by selecting an implicit low order operator $\lopdt^1$.
In that case, one obtains an iterative procedure that is not explicit, yet, much simpler than the direct solution of $\lopdt^2$.
In applications to ODEs and PDEs, implicit formulations \cite{abgrall2020high,busto2021staggered,ALEMOOD2} allow to achieve better stability properties and less time step restrictions. 
The theoretical framework presented in this work applies both to explicit and implicit settings.}
\end{remark}

\subsection{New efficient DeC methods}
%New DeC with proof
Here, we will discuss, at the theoretical level, an efficient modification for DeC methods. 
It is based on the replacement of the operators $\lopdt^{1}$ and $\lopdt^{2}$ by iteration-specific operators $\lopdt^{1,(p)}$ and $\lopdt^{2,(p)}$ in order to strictly obtain the order $p$ at the $p$-th iteration. %, guaranteeing maximal order of accuracy $p$ at $\undu^{(p)}$.
In particular, we prove the following result.

\begin{theorem}\label{th:NEWDEC}
%SPACES AND OPERATORS ARE DEFINED FROM p=1 ON
%PROJECTORS FROM p=0 ON
Consider a problem with exact solution $\uex \in Z$. Then, take some normed spaces $(\Xp,\norm{\cdot}_{\Xp})$ for $p\in \mathbb{N}$ and $(\Yp,\norm{\cdot}_{\Yp})$ for $p\geq 1$. %OKKKKKKK
For every $p \geq 1 $, consider also two operators $\lopdt^{1,(p)},\lopdt^{2,(p)}:\Xp \rightarrow \Yp$ dependent on the same parameter $\Delta$ and fulfilling the properties of 
Theorem \ref{th:DeC} for some $\alpha_{1}^{(p)}, \alpha_{2}^{(p)}>0$ and $\usolp\in \Xp$. 
Furthermore, let us assume that $ \forall p\in \mathbb{N}$ there exist an embedding operator $\embep:\Xp\rightarrow X^{(p+1)}$, associating to each $\undu^{(p)} \in \Xp$ an approximation $\undu^{*(p)}:=\embep(\undu^{(p)}) \in X^{(p+1)}$, and some projection $\projp:Z\rightarrow \Xp$, associating to $\uex$ an approximation $\uexp:=\projp(\uex) \in \Xp.$

Given an $\underline{\uvec{u}}^{(0)}\in X^{(0)}$, we consider the new DeC method whose general $p$-th iteration is given by
\begin{equation}
\begin{cases}
\undu^{*(p-1)}:=\mathcal{E}^{(p-1)}(\undu^{(p-1)}),\\
\lopd^{1,(p)}(\underline{\uvec{u}}^{(p)}):=\lopd^{1,(p)}(\underline{\uvec{u}}^{*(p-1)})-\lopd^{2,(p)}(\underline{\uvec{u}}^{*(p-1)}).
\end{cases}
\label{eq:NEWDEC_p_iteration}
\end{equation}
Suppose that the following properties hold:
\begin{enumerate}
\item \textbf{Accuracy of $\usolp$ with respect to $\uexp$} \\
\begin{equation}
\norm{\usolp-\uexp}_{\Xp}=O(\Delta^{p+1}), \quad p \geq 1;
\label{eq:usolp_uexp}
\end{equation}
\item \textbf{Accuracy of the embedding $\embep$} \\
\begin{equation}
\norm{\undu^{*(p)}-\undu_{ex}^{(p+1)}}_{X^{(p+1)}}\leq C\norm{\undu^{(p)}-\uexp}_{\Xp}, \quad \forall p\in \mathbb{N},
\label{eq:embE}
\end{equation}
for some constant $C$ independent on $\Delta$;
\item \textbf{Accuracy of $\underline{\uvec{u}}^{(0)}$} \\
\begin{equation}
\norm{\undu^{(0)}-\undu_{ex}^{(0)}}_{X^{(0)}}=O(\Delta).
\label{eq:accuracyu0}
\end{equation}
\end{enumerate}
Then, it follows that
\begin{equation}
\norm{\undu^{(p)}-\uexp}_{\Xp}=O(\Delta^{p+1}), \quad \forall p\in \mathbb{N}.
\label{eq:NEWDEC_accuracy}
\end{equation}
\end{theorem}
\begin{proof}
The proof is based on the induction. The base case for $p=0$ is trivially given by assumption \eqref{eq:accuracyu0}. Let us now focus on the induction step. We assume that \eqref{eq:NEWDEC_accuracy} holds for a specific $p$ and we will prove it for $p+1$.
By the triangular inequality, we have 
\begin{equation}
\norm{\undu^{(p+1)}-\undu_{ex}^{(p+1)}}_{X^{(p+1)}}\leq \norm{\undu^{(p+1)}-\undu_{\Delta}^{(p+1)}}_{X^{(p+1)}}+\norm{\undu_{\Delta}^{(p+1)}-\undu_{ex}^{(p+1)}}_{X^{(p+1)}}.
\end{equation}
The second term at the right hand side is an $O(\Delta^{p+2})$ for \eqref{eq:usolp_uexp}, hence, let us focus on the first term.
By the proof of Theorem \ref{th:DeC} concerning the original methods, we have that
\begin{equation}
\norm{\undu^{(p+1)}-\undu_{\Delta}^{(p+1)}}_{X^{(p+1)}}\leq \Delta \frac{\alpha_{2}^{(p+1)}}{\alpha_{1}^{(p+1)}}\norm{\undu^{*(p)}-\undu_{\Delta}^{(p+1)}}_{X^{(p+1)}},
\end{equation}
which, applying the triangular inequality, gives
\begin{equation}
	\begin{split}
		&\norm{\undu^{(p+1)}-\undu_{\Delta}^{(p+1)}}_{X^{(p+1)}} \\
		\leq&\Delta \frac{\alpha_{2}^{(p+1)}}{\alpha_{1}^{(p+1)}}\left( \norm{\undu^{*(p)}-\undu_{ex}^{(p+1)} }_{X^{(p+1)}}+\norm{   \undu_{ex}^{(p+1)}-\undu_{\Delta}^{(p+1)}}_{X^{(p+1)}} \right).
	\end{split}
\end{equation}
Again, due to \eqref{eq:usolp_uexp}, the second term in parenthesis at the right hand side is $O(\Delta^{p+2})$, therefore, we focus on the first term.
Due to the assumption on the accuracy of the embedding \eqref{eq:embE} and to the induction hypothesis, we have
\begin{equation}
\norm{\undu^{*(p)}-\undu_{ex}^{(p+1)} }_{X^{(p+1)}}\leq C\norm{\undu^{(p)}-\undu_{ex}^{(p)}}_{\Xp} = O(\Delta^{p+1}).
\end{equation}
Hence,
\begin{equation}
\Delta \frac{\alpha_{2}^{(p+1)}}{\alpha_{1}^{(p+1)}}\norm{\undu^{*(p)}-\undu_{ex}^{(p+1)} }_{X^{(p+1)}}=O(\Delta^{p+2}),
\end{equation} 
which completes the proof.
\end{proof}
Let us notice that, in the previous theorem, the accuracy estimate is always referred to a projection of the exact solution and not to a fixed high order approximation. Hence, the order of accuracy is formally not bounded and we can approximate $\uex$ arbitrarily well. 
In particular, if $\uexp$ yields an approximation of $\uex$ which is $O(\Delta^{p+1})$ accurate, thanks to Theorem \ref{th:NEWDEC}, also the approximation associated to $\undu^{(p)}$ will have the same accuracy with respect to $\uex$.

\section{ADER-Discontinuous Galerkin scheme}\label{chap:ADER}
The ADER methods are various techniques to obtain arbitrary high order methods for differential problems. Even though the first ADER \cite{titarev2002ader} was based on the Cauchy-Kovalevskaya theorem, nowadays, it is mainly known as a technique that exploits the weak formulation of the original problem in order to obtain high order discretization forms that are solved iteratively \cite{dumbser2008unified,boscheri2022continuous,han2021dec,gaburro2021unified}. In this section, we will present a formulation for hyperbolic PDEs in combination with a discontinuous Galerkin (DG) space discretization, and we will show how it can be interpreted as an arbitrary high order iterative method in the previously presented DeC framework. We will also describe in a final subsection the \APNPM variant of the method, still recastable as DeC scheme, which allows for applications to finite volume (FV) formulations as well. 

\subsection{Numerical method}
We want to approximate the analytical solution $\uvec{u}: \overline{\Omega} \times \mathbb{R}^+_0 \rightarrow \mathbb{R}^Q$ of the following $Q$-dimensional hyperbolic PDE
\begin{equation}
\label{eq:sys}
\frac{\partial}{\partial t}\uvec{u}(\uvec{x},t)+\mathrm{div}_{\uvec{x}} \boldsymbol{F}(\uvec{u}(\uvec{x},t))=\boldsymbol{S}(\uvec{x},\uvec{u}(\uvec{x},t)), \quad (\uvec{x},t) \in \Omega\times \mathbb{R}^+_0,
\end{equation}
supplemented with suitable initial and boundary conditions, where $\Omega\subseteq \mathbb{R}^D$ is a bounded $D$-dimensional space domain, $\boldsymbol{F}:\R^Q \to \R^{Q\times D}$ is the flux tensor and $\boldsymbol{S}:\overline{\Omega} \times \R^{Q} \to \R^Q$ is the source function. 
To shorten the notation, let us define $\E(\uvec{u},\uvec{x}):=\mathrm{div}_{\uvec{x}} \boldsymbol{F}(\uvec{u})-\boldsymbol{S}(\uvec{x},\uvec{u})$, the time evolution operator of the PDE up to the minus sign, so that \eqref{eq:sys} becomes
\begin{equation}\label{eq:sys_ODE}
	\frac{\partial}{\partial t}\uvec{u}(\uvec{x},t)+\E(\uvec{u}(\uvec{x},t),\uvec{x})=\uvec{0}, \quad (\uvec{x},t) \in \Omega\times \mathbb{R}^+_0.
\end{equation}
Let us focus on a generic time step $[t_n,t_{n+1}]$, with $\Delta t:=t_{n+1}-t_n$. The goal is to find an approximation $\uvec{u}_{n+1}(\uvec{x})\approx \uvec{u}(\uvec{x},t_{n+1})$ of the analytical solution in $\overline{\Omega}$ at time $t_{n+1}$ by knowing an approximation $\uvec{u}_{n}(\uvec{x})\approx \uvec{u}(\uvec{x},t_{n})$ at time $t_n.$ %As usual in the context of consistency analyses, we assume here to start from the exact solution $\uvec{u}_{n}(\uvec{x})=\uvec{u}(\uvec{x},t_{n}).$

In particular, for any $n$, we adopt a classical DG space discretization for $\uvec{u}_{n}(\uvec{x})\approx \uvec{u}(\uvec{x},t_{n})$: we consider a tessellation $\tess$ of $\overline{\Omega}$ made of non-overlapping convex polytopals $K$ with mesh parameter $h$, and we consider $\uvec{u}_n(\uvec{x})$ in a space of discontinuous piecewise polynomial functions of degree $M$, i.e., $(V_M)^Q$ with $V_M:=\left\lbrace g \in L^2(\Omega) ~\text{s.t.}~ g\vert_K \in \mathbb{P}_M(K)\right\rbrace$, yielding an $(M+1)$-th order of accuracy approximation space. Therefore, locally in each element $K$, we can consider the following representation of the approximated solution with a local basis $\left\lbrace \trialspace{i}(\uvec{x}) \right\rbrace_{i=1,\dots,I} $ of $\mathbb{P}_M(K)$
\begin{equation}
\uvec{u}_n(\uvec{x}):=\sum_{i=1}^I \uvec{c}_{i}^n \trialspace{i}(\uvec{x}),\quad \forall \uvec{x}\in K,
\label{eq:reconstruction_at_tn}
\end{equation}
where 
\RIIcolor{$I$ is the number of local basis functions and}
the label $K$ on the coefficients $\uvec{c}_{i}^n$ and on the basis functions $\trialspace{i}$ is omitted to lighten the notation as we will consider computations in a single generic element, in the sequel.

The ADER method, applied to this context, is based on the weak formulation of the governing equations \RIIcolor{\eqref{eq:sys_ODE}} in space-time and it is characterized by two steps: an iterative local space-time predictor and a final corrector step, described hereafter.
\RIcolor{
Before entering the details, it is useful to briefly describe the role of such steps for ADER methods. The predictor, based on a local explicit iterative procedure, is used to compute local high order polynomial approximations of the solution in space-time control volumes without considering any communication between different cells. The solution obtained in this step is high order accurate but not stable, as no upwinding has been taken into account in its computation. This prediction is later used in the corrector step to provide a global explicit update of the numerical solution, allowing communication between neighboring cells through numerical fluxes, which introduce the necessary upwinding and numerical dissipation to achieve stability.
	A stability study of the method can be found in \cite{dumbser2008unified}.
}
\subsubsection{Local space-time predictor}
The purpose of this step is to find a high order approximation of the solution in each space-time control volume $C_K=K\times [t_n,t_{n+1}]$. In this step, no communication  between the cells happens.
We consider the weak formulation of \eqref{eq:sys} over $C_K$, obtained through the multiplication by a smooth test function $\vartheta: C_K \to \mathbb{R}$, the integration over $C_K$ and subsequent integration by parts in time:
\begin{align}
\label{eq:sys_weak}
\begin{split}
&\int_K\!\left[ \uvec{u}(\uvec{x},t_{n+1})\vartheta(\uvec{x},t_{n+1}) \!-\!\uvec{u}(\uvec{x},t_n)\vartheta(\uvec{x},t_{n})\right] d\uvec{x} -\!\! \int_{C_K}\!\!\!\!\! \uvec{u}(\uvec{x},t) \frac{\partial}{\partial t} \vartheta(\uvec{x},t) d\uvec{x}dt\\
& + \int_{C_K}  \E(\uvec{u}(\uvec{x},t),\uvec{x})\vartheta(\uvec{x},t)\, d\uvec{x}dt=\uvec{0}.
\end{split}
\end{align}
Next, we project it onto a finite dimensional space spanned by the tensor product of the previously introduced local spatial basis $\left\lbrace \trialspace{i}(\uvec{x}) \right\rbrace_{i=1,\dots,I}$ and a temporal basis $\left\lbrace \trialtime{m}(t) \right\rbrace_{m=0,\dots,M}$ over $[t_{n},t_{n+1}]$ guaranteeing $(M+1)$-th order of accuracy. 
As an example for the latter, one can think to a Lagrangian basis of degree $M$ or a truncated Taylor series up to the $M$-th degree term to obtain an approximation of order $M+1$. Here, we will use modal basis functions both for space and time, although such choice is not mandatory. The basis functions are explicitly described in Appendix \ref{app:basis_function}. 
As usual in the literature, we assume the basis functions to be normalized in such a way that their maximum absolute value over $C_K$ is an $O(1)$. 
%Just to fix the ideas, we will consider the Lagrange polynomials of degree $M$ associated to $M+1$ equispaced points in the interval $[t_{n},t_{n+1}]$, $\left\lbrace \trialtime{m} \right\rbrace_{m=0,\dots,M}$, but other more efficient choices are indeed possible.

In particular, we consider the local discretization of $\uvec{u}_h$ in $C_K$
\begin{equation}
\label{eq:reconstruction_predictor}
\uvec{u}_h(\uvec{x},t):=\sum_{i=1}^I\sum_{m=0}^M \uvec{u}_i^m\trialspace{i}(\uvec{x})\trialtime{m}(t)=\sum_{\ell=1}^L \uvec{u}^{\ell}\trialspacetime{\ell}(\uvec{x},t), \quad \forall(\uvec{x},t)\in C_K,
\end{equation}
in which, in order to shorten the notation, we have denoted by $\lbrace\trialspacetime{\ell}(\uvec{x},t)\rbrace_{\ell=1,\dots,L}$ a permutation of the basis functions $\lbrace \trialspace{i}(\uvec{x})\trialtime{m}(t) \rbrace_{\substack{i=1,\dots,I\\m=0,\dots,M}}$ and by $\uvec{u}^{\ell}$ the corresponding coefficients $\uvec{u}_i^m$, where implicitly we defined a bijection that gives $\ell = \ell(i,m)$.
Finally, we consider the projection of \eqref{eq:sys_weak} on the space-time DG functional space generated by $\lbrace \trialspacetime{\ell}(\uvec{x},t) \rbrace_{\ell=1,\dots,L}$, that is
\begin{align}
\label{eq:sys_weak_discrete}
\begin{split}
\sum_{\ell=1}^L &\left[ \int_K \trialspacetime{\ell}(\uvec{x},t_{n+1})\trialspacetime{j}(\uvec{x},t_{n+1}) d\uvec{x} - \int_{C_{K}} \trialspacetime{\ell}(\uvec{x},t) \frac{\partial}{\partial t} \trialspacetime{j}(\uvec{x},t) d\uvec{x}dt \right]\uvec{u}^{\ell} \\
-&\int_K \uvec{u}_n(\uvec{x})\trialspacetime{j}(\uvec{x},t_{n}) d\uvec{x} + \int_{C_{K}}\E(\uvec{u}_h(\uvec{x},t),\uvec{x})\trialspacetime{j}(\uvec{x},t) d\uvec{x}dt=\uvec{0}
\end{split}
\end{align}
for any $j=1,\dots,L$. This is a nonlinear system in the unknowns $\uvec{u}^{\ell}$, whose solution yields the $(M+1)$-th order accurate approximation \eqref{eq:reconstruction_predictor} of the analytical solution.
Let us notice that $\uvec{u}_n(\uvec{x})$ in \eqref{eq:sys_weak_discrete} is known by assumption and the related integral involving such function can thus be computed.
%\begin{remark}[On the handling of the differential operators]
Notice that, in order to obtain a fully local formulation, the divergence theorem in space has not been applied.
On the other hand, the integration by parts in time has been performed to introduce a causality effect and a dependency on the initial information at time $t_{n}$.
%\end{remark}

Now, it is possible to recast each local system \eqref{eq:sys_weak_discrete} in a matrix-vector formulation writing 
\begin{equation}
\label{eq:ADER_predictor_matricial}
B\undu-\undr+\undspacetilde(\undu)=\uvec{0},
\end{equation}
where the matrix $B$ and the vectors $\undu$, $\undr$ and $\undspacetilde$ are given by
{\small
\begin{align}
\label{eq:ADER_predictor_structures}
\begin{split}
B_{j,\ell}&:=\int_K \trialspacetime{\ell}(\uvec{x},t_{n+1})\trialspacetime{j}(\uvec{x},t_{n+1}) d\uvec{x} - \int_{C_K} \trialspacetime{\ell}(\uvec{x},t) \frac{\partial}{\partial t} \trialspacetime{j}(\uvec{x},t) d\uvec{x}dt,\\
\undu &:=\begin{pmatrix}
\uvec{u}^1\\
\vdots\\
\uvec{u}^L
\end{pmatrix}, \quad \undr :=\begin{pmatrix}
\int_K \uvec{u}_n(\uvec{x})\trialspacetime{1}(\uvec{x},t_n)d\uvec{x}\\
\vdots\\
\int_K \uvec{u}_n(\uvec{x})\trialspacetime{L}(\uvec{x},t_n)d\uvec{x}
\end{pmatrix}, \quad \undspacetilde(\undu):= \begin{pmatrix}
\int_{t_n}^{t_{n+1}} \spacestuff_1(\undu,t) dt\\
\vdots\\
\int_{t_n}^{t_{n+1}} \spacestuff_L(\undu,t) dt
\end{pmatrix}
\end{split}
\end{align}}
with $\spacestuff_j(\undu,t):=\int_K\E(\uvec{u}_h(\uvec{x},t),\uvec{x})\trialspacetime{j}(\uvec{x},t) d\uvec{x}$.
Let us observe that $\undr$ is constant and explicitly computable as $\uvec{u}_n(\uvec{x})$ is known.
\begin{remark}{(On the matrix $B$)}\label{rmk:B}
The definition of the matrix $B$ is referred to a scalar PDE, it must be block-expanded for a vectorial problem.
Let us notice that the elements of the matrix $B$ are $O(h^D)$ due to the integral over $K$ and to the normalization assumed on the basis functions. The integral in time on the second term of $B_{j, \ell}$ is balanced by the derivative in time on $\trialspacetime{j}.$
\end{remark}
Concerning the well-posedness and the solution of the nonlinear system \eqref{eq:ADER_predictor_matricial}, we can prove that for $\Delta t$ small enough it admits a unique solution which can be obtained through the iterative procedure
\begin{equation}\label{eq:fixed_point}
\undu^{(p)}=B^{-1}\left[ \undr-\undspacetilde \left(  \undu^{(p-1)} \right) \right],
\end{equation}
which converges unconditionally to the solution of the system, for any initial vector $\undu^{(0)}$.
In order to do that, let us first prove the following useful lemma.
\begin{lemma}[Lipschitz-continuity-like property of $\undspacetilde$]\label{lem:lipschitz_res}
Under smoothness assumptions, the function $\undspacetilde(\cdot)$ is such that 
\begin{equation}
\norm{ \undspacetilde(\undv)-\undspacetilde(\undw) }_\infty \leq \Delta t \abs{ K } C_{Lip} \norm{ \undv-\undw }_\infty
\label{eq:importantpropertyprelimiary}
\end{equation}
where $\norm{\cdot}_{\infty}$ is the infinity norm over $\mathbb{R}^{L\times Q}$ and $C_{Lip}$ is a constant independent of $\Delta t$ and of the element $K.$
\end{lemma}
\begin{proof}
By a direct computation of the generic $j$-th component of the left-hand side of \eqref{eq:importantpropertyprelimiary}, recalling the definition of the functions $\spacestuff_j$, through basic analysis, we get
\begin{align}
\begin{split}
&\Bigg\Vert \int_{t_n}^{t_{n+1}} \spacestuff_j(\undv,t)dt -\!\!\int_{t_n}^{t_{n+1}} \spacestuff_j(\undw,t) dt \Bigg\Vert_{\infty,Q}  \!\!\!\!\leq \int_{t_n}^{t_{n+1}} \Big\Vert  \spacestuff_j(\undv,t) - \spacestuff_j(\undw,t)  \Big\Vert_{\infty,Q} \!\!\! dt\\
&\leq\int_{C_K} \Big\Vert  \E(\uvec{v}_h(\uvec{x},t),\uvec{x}) -\E(\uvec{w}_h(\uvec{x},t),\uvec{x})  \Big\Vert_{\infty,Q} \Big\vert  \trialspacetime{j}(\uvec{x},t) \Big\vert d\uvec{x} dt,
\end{split}
\label{eq:randomname}
\end{align}
where $\norm{\cdot}_{\infty,Q}$ is the infinity norm over $\mathbb{R}^Q$, $\uvec{v}_h(\uvec{x},t):=\sum_{\ell=1}^L \uvec{v}^{\ell}\trialspacetime{\ell}(\uvec{x},t)$ and $\uvec{w}_h(\uvec{x},t):=\sum_{\ell=1}^L \uvec{w}^{\ell}\trialspacetime{\ell}(\uvec{x},t) \quad \forall(\uvec{x},t)\in C_K.$
For regular data, we can assume that the following Lipschitz-continuity property holds
\begin{align}
\Big\Vert   \E(\uvec{v}_h(\uvec{x},t),\uvec{x}) -\E(\uvec{w}_h(\uvec{x},t),\uvec{x}) \Big\Vert_{\infty,Q} \leq C_{0} \norm{ \undv-\undw }_\infty,
\end{align}
where $C_0$ is a constant independent of $\Delta t$ and of the element $K$, leading to
\begin{align}
\begin{split}
&\Bigg\Vert \int_{t_n}^{t_{n+1}}\!\! \spacestuff_j(\undv,t)dt -\!\!\int_{t_n}^{t_{n+1}}\!\! \spacestuff_j(\undw,t) dt \Bigg\Vert_{\infty,Q} \!\!\!\!\leq C_{0} \norm{ \undv-\undw }_\infty \int_{C_K} \Big\vert  \trialspacetime{j}(\uvec{x},t) \Big\vert d\uvec{x} dt.
\end{split}
\end{align}
The space-time basis functions $\trialspacetime{j}$ are bounded in absolute value by a constant $C_\vartheta$ independent of $\Delta t$ and $K$, yielding 
\begin{align}
\begin{split}
&\Bigg\Vert \int_{t_n}^{t_{n+1}} \spacestuff_j(\undv,t)dt -\int_{t_n}^{t_{n+1}} \spacestuff_j(\undw,t) dt \Bigg\Vert_{\infty,Q}\leq C_{0}C_\vartheta \norm{ \undv-\undw }_\infty \Delta t \abs{K},
\end{split}
\end{align}
which, setting $C_{Lip}:=C_{0}C_\vartheta$ and taking the maximum over $j=1,\dots,L$ at the left-hand side, is the thesis.
\end{proof}
A straightforward consequence of the previous result is the following corollary.
\begin{corollary}[Lipschitz-continuity-like property of $B^{-1}\undspacetilde$]\label{cor:lipschitz}
Under the assumptions of the previous lemma, it holds
\begin{equation}
\norm{ B^{-1}\left[\undspacetilde(\undv)-\undspacetilde(\undw)\right] }_\infty \leq \Delta t \widetilde{C}_{Lip} \norm{ \undv-\undw }_\infty
\label{eq:importantproperty}
\end{equation}
where $\widetilde{C}_{Lip}$ is a constant independent of $\Delta t$ and of the element $K.$
\end{corollary}
\begin{proof}
By basic linear algebra we have 
\begin{align}
\norm{ B^{-1}\left[\undspacetilde(\undv)-\undspacetilde(\undw)\right] }_\infty \leq \norm{ B^{-1}}_\infty \norm{ \undspacetilde(\undv)-\undspacetilde(\undw) }_\infty
\label{eq:intermediateneeded}
\end{align}
where the infinity norm applied to $B^{-1}$ is the matrix norm induced by the related vector norm. 
As observed in Remark \ref{rmk:B}, $B$ is an $O(h^D)$ and, hence, its inverse is an $O(h^{-D})$, leading to $\norm{ B^{-1}}_\infty\leq C_B h^{-D}$ for some constant $C_B$ independent of the specific element $K$ and of $\Delta t$. %, but depending on the regularity of the mesh.
Using this fact, in combination with the result of Lemma \ref{lem:lipschitz_res}, we obtain
\begin{align}
\norm{ B^{-1}\left[\undspacetilde(\undv)-\undspacetilde(\undw)\right] }_\infty \leq C_B h^{-D} \Delta t \abs{ K } C_{Lip} \norm{ \undv-\undw }_\infty. 
\end{align}
By observing that for a regular mesh $\abs{K}\leq C_\tau h^D$ for some constant $C_\tau$ independent of $K$, we get the thesis 
\begin{align}
\norm{ B^{-1}\left[\undspacetilde(\undv)-\undspacetilde(\undw)\right] }_\infty %&\leq C_B h^{-D} \Delta t \abs{ K } C_{Lip} \norm{ \undv-\undw }_\infty\\
%&\leq C_B h^{-D} \Delta t C_\tau h^D C_{Lip} \norm{ \undv-\undw }_\infty  \\
&\leq  \Delta t C_B  C_\tau C_{Lip} \norm{ \undv-\undw }_\infty  
\end{align}
for $\widetilde{C}_{Lip}:=C_B  C_\tau C_{Lip}.$
\end{proof}

This allows us to prove the existence and uniqueness of the solution of \eqref{eq:ADER_predictor_matricial}.
\begin{proposition}[Well-posedness and solution of the nonlinear system]\label{prop:solutiontononlinearsystem}
For $\Delta t$ small enough, the nonlinear system \eqref{eq:ADER_predictor_matricial} has a unique solution, which is the limit of \eqref{eq:fixed_point} for $p\rightarrow+\infty$.
%\begin{equation}
%\undu^{(p)}=B^{-1}\left[ \undr-\undspacetilde (  \undu^{(p-1)} ) \right].
%\label{eq:fixed_point}
%\end{equation}
\end{proposition}
\begin{proof}
We define the map $\mathcal{J}:\mathbb{R}^{L\times Q}\rightarrow\mathbb{R}^{L\times Q}$ as 
$ \mathcal{J}(\undu):=B^{-1}\left[ \undr-\undspacetilde (  \undu ) \right].$
It is immediate to verify that a fixed point of $\mathcal{J}$ (if any) is also a solution of \eqref{eq:ADER_predictor_matricial} and viceversa.
Due to the fact that $\mathbb{R}^{L\times Q}$ is finite dimensional, if we are able to prove that $\mathcal{J}$ is a contraction, by the Banach fixed-point theorem, we know that there exists a unique fixed point and that this can be obtained as the limit of the iterative procedure $\undu^{(p)}:=\mathcal{J}(\undu^{(p-1)})$, which is equivalent to \eqref{eq:fixed_point}. We will now show that, for $\Delta t$ small enough, $\mathcal{J}$ is indeed a contraction. 
In fact, by a direct computation, it holds
\begin{align}
\norm{\mathcal{J}(\undv)-\mathcal{J}(\undw)}_\infty%=\norm{B^{-1}\left[ \undr-\undspacetilde \left(  \undv \right) \right]-B^{-1}\left[ \undr-\undspacetilde \left(  \undw \right) \right] }_\infty 
=\norm{B^{-1}\left[ \undspacetilde \left(  \undv \right) -\undspacetilde \left(  \undw \right) \right] }_\infty
\end{align}
and, applying Corollary \ref{cor:lipschitz} on the Lipschitz-continuity-like property of $B^{-1}\undspacetilde$, we retrieve the thesis 
\begin{align}
\begin{split}
\norm{\mathcal{J}(\undv)-\mathcal{J}(\undw)}_\infty = \norm{B^{-1}\left[ \undspacetilde \left(  \undv \right) -\undspacetilde \left(  \undw \right) \right] }_\infty\leq \Delta t \widetilde{C}_{Lip} \norm{ \undv-\undw }_\infty
\end{split}
\end{align}
for $\Delta t<\frac{1}{\widetilde{C}_{Lip}}.$
\end{proof}

All the local approximations, obtained by solving the nonlinear system in each control volume $C_K$, constitute a global $(M+1)$-th order accurate approximation of the analytical solution. It is piecewise polynomial in each $C_K$ and discontinuous across the faces of $C_K$ shared with other control volumes and we denote it, by an abuse of notation, as $\uvec{u}_h$. %both in time and space.

\begin{remark}{(On the computational efficiency)}
In several works, the nonlinear system \eqref{eq:ADER_predictor_matricial} is solved by carrying the iterative process \eqref{eq:fixed_point} until a convergence criterion is met up to a certain tolerance \cite{dumbser2008unified,boscheri2022continuous,gaburro2021unified}. 
This leads to a waste of resources as the underlying discretization error of the system \eqref{eq:ADER_predictor_matricial} with respect to the analytical solution of the PDE is of order $M+1$, hence, smaller tolerances are in general unnecessary. 
%, in practice, the goal of the prediction step is to provide a local $(M+1)$-th order accurate approximation of the analytical solution to our PDE. 
%To this end, it would be sufficient to determine an $(M+1)$-th order accurate approximation of the solution to the nonlinear system. 
In this context, it is possible to obtain an $(M+1)$-th order accurate approximation of the solution of \eqref{eq:ADER_predictor_matricial} by performing exactly $M+1$ iterations. More details on this will be explained in Section \ref{chap:ADERasDeC}.
%To distinguish these two approaches, we will denote by ADER-DG-$\varepsilon$($M+1$,$\texttt{tol}$) the method with polynomials of degree $M$ that proceed until a convergence error of $\texttt{tol}$ is achieved, while we will denote with ADER-DG($M+1$) the method with polynomials of degree $M$ that optimally uses $M+1$ iterations. Both methods have order $M+1$.
%\lore{(NB: We are not reporting the results of ADER-DG($M+1$), I do not know if it is the case to define it here. We could give the name of ADER-DG-$\varepsilon$($M+1$,$\texttt{tol}$) also before describing the numerical results... I'm questioning from "To distinguish these two approaches", not what has been written before.)}
%NB: WE DO NOT DEFINE ADER-DG($M+1$) and ADER-DG-$\varepsilon$($M+1$,$\texttt{tol}$). The results only refer to ADER-DG-$\varepsilon$($M+1$,$\texttt{tol}$) which we can directly call ADER-DG to light the notation
\end{remark}

\subsubsection{Final corrector step}
From the predictor step, we have in each control volume $C_K$ a local $(M+1)$-th order accurate approximation $\uvec{u}_h$ of the analytical solution in the form \eqref{eq:reconstruction_predictor}, which has been computed without considering any sort of communication between the neighboring elements. 
In the final corrector step, we exploit such approximation to finally get $\uvec{u}_{n+1}(\uvec{x}),$ taking into account the coupling between the elements.
In particular, we consider again a weak formulation of \eqref{eq:sys} in $C_K$, but this time we use a spatial-only test function $\varphi(\uvec{x})$ and we apply the divergence theorem in space thus getting
\begin{equation}
\label{eq:sys_weak_corrector}
\begin{split}
&\int_K \uvec{u}(\uvec{x},t_{n+1})\varphi(\uvec{x}) d\uvec{x} -\int_K \uvec{u}(\uvec{x},t_n)\varphi(\uvec{x}) d\uvec{x} \\
+& \int_{t_n}^{t_{n+1}}\!\! \int_{\partial K} \varphi(\uvec{x}) \uvec{F}(\uvec{u}(\uvec{x},t))   \cdot \uvec{\nu}(\uvec{x}) d \uvec{\sigma} dt  \\
- & \int_{C_K}  \uvec{F}(\uvec{u}(\uvec{x},t)) \cdot \nabla_{\uvec{x}} \varphi(\uvec{x})  d\uvec{x}dt-\int_{C_K} \boldsymbol{S}(\uvec{x},\uvec{u}(\uvec{x},t)) \varphi(\uvec{x}) d\uvec{x}dt=\uvec{0},
\end{split}
\end{equation}
where $\uvec{\nu}(\uvec{x})$ is the outward pointing normal to $\partial K$.
The divergence theorem in space provides the desired coupling between the neighboring cells because the solution $\uvec{u}_h$, computed locally in each control volume $C_K$ through the predictor, is discontinuous across the boundaries $\partial K$ and, thus, a numerical flux $\widehat{\uvec{F}}$ is needed to compute the flux at the cell interfaces $\partial K$. We can use either a simple and robust local Lax Friedrichs scheme \cite{Rusanov:1961a}, or a less dissipative Osher numerical flux function \cite{OsherNC}.
At the discrete level, recalling the adopted discretization \eqref{eq:reconstruction_at_tn} for $\uvec{u}_n(\uvec{x})$ $\forall n$, we get for each control volume $C_K$
\begin{align}
\label{eq:sys_weak_corrector_discrete}
\begin{split}
&\sum_{i=1}^I \int_K \trialspace{i}(\uvec{x})\trialspace{j}(\uvec{x}) d\uvec{x} (\uvec{c}_i^{n+1}-\uvec{c}_i^{n}) \\
+& \int_{t_n}^{t_{n+1}} \int_{\partial K} \trialspace{j}(\uvec{x}) \widehat{\uvec{F}}(\uvec{u}_h\vert_K(\uvec{x},t),\uvec{u}_h\vert_{K^+}(\uvec{x},t)) \cdot \uvec{\nu}(\uvec{x}) d \uvec{\sigma} dt\\
  - &\int_{C_K}  \uvec{F}(\uvec{u}_h(\uvec{x},t)) \cdot \nabla_{\uvec{x}} \trialspace{j}(\uvec{x})  d\uvec{x}dt
 - \int_{C_K} \boldsymbol{S}(\uvec{x},\uvec{u}_h(\uvec{x},t)) \trialspace{j}(\uvec{x}) d\uvec{x}dt=\uvec{0}
\end{split}
\end{align}
for every $j=1,\dots,I$, with $K^+$ being the neighboring cell of $K$ sharing $\partial K$ at a certain point $\uvec{x}$. Again, we remark that this step has a global character due to the computation of the numerical fluxes, but it is explicit as $\uvec{u}_h$ has been obtained in the predictor step. Let us notice that the linear systems involved in the corrector are local and even smaller than the predictor ones, thus readily invertible.

By solving the linear system \eqref{eq:sys_weak_corrector_discrete} with respect to the coefficients $\uvec{c}_i^{n+1}$ in each element $K$, we get the final solution $\uvec{u}_{n+1}(\uvec{x})=\sum_{i=1}^I \uvec{c}_i^{n+1} \trialspace{i}(\uvec{x})$ $\forall \uvec{x}\in K$ which is an $(M+1)$-th order accurate approximation of $\uvec{u}(\uvec{x},t_{n+1}).$

\subsection{ADER-DG as DeC}\label{chap:ADERasDeC}
It is possible to interpret the ADER-DG predictor step as a DeC procedure. We set $\Delta:=\Delta t$ and, from the local space-time nonlinear system \eqref{eq:ADER_predictor_matricial}, we define the high order nonlinear operator $\lopdt^2:\mathbb{R}^{L \times Q} \rightarrow \mathbb{R}^{L \times Q}$ as
\begin{align}
\label{eq:ADER_l2}
\lopdt^2(\undu):=\undu-B^{-1}\left[\undr-\undspacetilde(\undu)\right].
\end{align}
Since solving the operator $\lopdt^2$ is equivalent to solve the system \eqref{eq:ADER_predictor_matricial}, we have already discussed the $(M+1)$-th order of accuracy of its solution.

The low order operator $\lopdt^1:\mathbb{R}^{L \times Q} \rightarrow \mathbb{R}^{L \times Q}$ is, instead, defined as
\begin{align}
\label{eq:ADER_l1}
\lopdt^1(\undu):=\undu-B^{-1}\left[\undr-\undspacetilde(\undu_0)\right],
\end{align}
with $\undu_0$ being a vector of local space-time representation coefficients with respect to the basis $\lbrace\trialspacetime{\ell}(\uvec{x},t)\rbrace_{\ell=1,\dots,L}$, yielding an $O(\Delta t)$-approximation of the analytical solution in $C_K$. As an example, 
%\begin{remark}[On the vector $\undu_0$]
the vector $\undu_0$ can be chosen such that $\sum_{\ell=1}^L \uvec{u}_0^\ell \vartheta^\ell(\uvec{x},t) = \uvec{u}_{n}(\uvec{x})$ for all $t \in [t_n,t_{n+1}]$.
In practice, this definition is merely formal as the related terms will cancel out in the iteration and in all the needed proofs. 
%\end{remark}
It can be shown that the local reconstruction of the PDE solution induced by the coefficients obtained by solving $\lopdt^1$ is first order accurate with respect to the analytical solution. Furthermore, let us observe how the problem $\lopdt^1(\widetilde{\undu})=\undz$ for some given $\undz \in \mathbb{R}^{L \times Q}$ can be easily solved by explicitly isolating $\widetilde{\undu}$.

In the following, we will prove that the operators that we have defined respect the three properties needed to apply Theorem \ref{th:DeC}, but first let us characterize the related DeC iterative procedure.
\subsubsection{Iterative ADER-DG-DeC procedure}
If we characterize the iterative procedure \eqref{eq:DeC_iteration} in the ADER context with the operators \eqref{eq:ADER_l1} and \eqref{eq:ADER_l2}, by a direct computation, we get  
\begin{equation}
	\begin{split}
\undu^{(p)}-B^{-1}\left[\undr-\undspacetilde(\undu_0)\right]&=\undu^{(p-1)}-B^{-1}\left[\undr-\undspacetilde(\undu_0)\right]\\
&-\undu^{(p-1)}+B^{-1}\left[\undr-\undspacetilde (\undu^{(p-1)})\right],
	\end{split}
\end{equation}
which reduces to 
\begin{equation}
\undu^{(p)}=B^{-1}\left[\undr-\undspacetilde(\undu^{(p-1)})\right].
\label{eq:ADERDEC_iteration}
\end{equation}
This is nothing but the fixed point iteration \eqref{eq:fixed_point}. The advantage of having put it into a DeC formulation is that, in this context, we have at our disposal an estimate for the accuracy of $\undu^{(p)}$ obtained at the generic iteration $p$ given by \eqref{eq:DeC_accuracy}. In particular, according to Remark \ref{rmk:P}, if $\undu^{(0)}$ yields an $O(\Delta t)$-approximation of the analytical solution, we have that the optimal number of iterations to achieve the formal accuracy is given by $P=M+1.$ A natural choice of the initial vector is thus $\undu^{(0)}:=\undu_0.$

\subsubsection{Proof of the properties of the operators $\lopdt^1,\lopdt^2$}
We have that the operators $\lopdt^1,\lopdt^2$ fulfill the hypotheses of Theorem \ref{th:DeC} as stated in the next theorem.
\begin{theorem}[ADER-DG is DeC]
The operators $\lopdt^1,\lopdt^2:\mathbb{R}^{L \times Q} \rightarrow \mathbb{R}^{L \times Q}$, defined respectively in \eqref{eq:ADER_l1} and \eqref{eq:ADER_l2}, fulfill the three hypotheses of Theorem \ref{th:DeC}.
\end{theorem}
\begin{proof}
\begin{enumerate}
\item \textbf{Existence of a unique solution to $\lopd^2$} \\
This property has been already proved in Proposition \ref{prop:solutiontononlinearsystem}, since solving the operator $\lopdt^2$ is equivalent to solve the nonlinear system \eqref{eq:ADER_predictor_matricial}.
%The proof is based on the construction of a map $\mathcal{J}:\mathbb{R}^{L\times Q}\rightarrow\mathbb{R}^{L\times Q}$ whose fixed points are solutions to the nonlinear system \eqref{eq:ADER_predictor_matricial}.
%We showed that for $\Delta t$ small enough this map is a contraction over a finite dimensional space and, therefore, there exists a unique fixed point. Since solving the nonlinear system \eqref{eq:ADER_predictor_matricial} is equivalent to solve the operator $\lopdt^2$, we deduce that there exists a unique solution to it. 
\item \textbf{Coercivity-like property of $\lopd^1$} \\
We consider the infinity norm over $\mathbb{R}^{L \times Q}$ and two general vectors $\undv,\undw \in \mathbb{R}^{L \times Q}$.
The proof of this property is immediate because, by a direct computation, we have
\begin{align}
	\begin{split}
		\norm{\lopd^1(\underline{\uvec{v}})-\lopd^1(\underline{\uvec{w}})}_\infty %&=\norm{\undv-B^{-1}\left[\undr-\undspacetilde(\undu_0)\right] - \undw+B^{-1}\left[\undr-\undspacetilde(\undu_0)\right]}_\infty\\
		&=\norm{\undv- \undw}_\infty
	\end{split}
\end{align}
and, thus, \eqref{eq:DeC_coercivity} holds with $\alpha_1=1.$
\item \textbf{Lipschitz-continuity-like property of $\lopd^1-\lopd^2$} \\
The proof of this property is based on Corollary \ref{cor:lipschitz}. A direct computation leads to the thesis:
\begin{align}
	&\norm{\left[\lopd^1(\underline{\uvec{v}})\!-\!\lopd^2(\underline{\uvec{v}})\right]\!-\!\left[\lopd^1(\underline{\uvec{w}})\!-\!\lopd^2(\underline{\uvec{w}})\right]}_\infty \\%\!\!\! = \norm{B^{-1}\left[ -\undspacetilde \left(  \undv \right) +\undspacetilde \left(  \undw \right) \right] }_\infty\\
	&=\norm{B^{-1}\left[ \undspacetilde \left(  \undv \right) -\undspacetilde \left(  \undw \right) \right] }_\infty\leq \Delta t \widetilde{C}_{Lip} \norm{ \undv-\undw }_\infty, \label{eq:proof_lip_ADERDEC}
\end{align}
where in \eqref{eq:proof_lip_ADERDEC} we applied Corollary \ref{cor:lipschitz}. 
\end{enumerate}
\end{proof}

\subsection{\APNPM and ADER-FV}
%Ma quindi il predictor è lo stesso per ADER-FV e ADER-DG se ho capito bene. 
%Quello che cambia è invece il fatto che non si usa, per la soluzione al tempo t_n, la stessa base spaziale del predictor (nel quale si continua a usare un polinomio HO). 
%La soluzione al tempo t_n viene invece definita costante a tratti e anche il corrector è formalmente identico up to il fatto che la base spaziale non è quella del predictor.
%E quando nel predictor serve la soluzione al tempo t_n per integrarla si fa una ricostruzione HO per non averla costante a tratti.

Other formulations of ADER are available in literature, in particular \APNPM \cite{gaburro2021posteriori,dumbser2009very,dumbser2010arbitrary,boscheri2019high} is a generalization of the ADER-DG formulation.
The \APNPM method is based on adopting, for the discretization $\uvec{u}_n(\uvec{x})$ of the solution at the time $t_n$, different local basis functions $\left\lbrace \trialspaceFVc{r} \right\rbrace_{r=1,\dots,R}$ spanning a space of discontinuous piecewise polynomial functions of degree $N\leq M$, i.e., $V_N$ with $V_N:=\left\lbrace g \in L^2(\Omega) ~\text{s.t.}~ g\vert_K \in \mathbb{P}_N(K)\right\rbrace$, yielding the reconstruction
\begin{equation}
\uvec{u}_n(\uvec{x}):=\sum_{r=1}^R \uvec{c}_{r}^n \trialspaceFVc{r}(\uvec{x}),\quad \forall \uvec{x}\in K.
\label{eq:reconstruction_at_tn_bis}
\end{equation}
Then, the scheme is formally identical to the one described before. 
In the predictor \eqref{eq:sys_weak_discrete}, the same $M$-th degree local spatial bases $\left\lbrace \trialspace{i}(\uvec{x}) \right\rbrace_{i=1,\dots,I}$ and temporal bases $\left\lbrace \trialtime{m}(t) \right\rbrace_{m=0,\dots,M}$ are considered, yielding a local reconstruction $\uvec{u}_h(\uvec{x},t)$ in each $C_K$ guaranteeing $(M+1)$-th order of accuracy. 
The corrector is also identical to the one previously described \eqref{eq:sys_weak_corrector_discrete}, up to the replacement of the basis functions $\trialspace{i}$ with the basis functions $\trialspaceFVc{r}$.%, allowing to obtain $\uvec{u}_{n+1}(\uvec{x})$ characterized by $(M+1)$-th order of accuracy.

%The only difference, with respect to the original formulation, is given by the fact that $\uvec{u}_{n}(\uvec{x})$, which must be integrated over $K$ in the predictor \eqref{eq:sys_weak_discrete}, has to be suitably reconstructed as a polynomial of degree $M$ to guarantee $(M+1)$-th order of accuracy if $N<M$. Usually, this is achieved via a WENO or CWENO reconstruction \cite{boscheri2019high,gaburro2021posteriori,dumbser2009very}.
The only difference with respect to the original formulation is given by the fact that, if $N<M$, a suitable $M$-th degree polynomial reconstruction $\widetilde{\uvec{u}}_{n}(\uvec{x})$ has to be considered in place of $\uvec{u}_{n}(\uvec{x})$ for the computation of the related integral over $K$ in the predictor \eqref{eq:sys_weak_discrete} in order to guarantee $(M+1)$-th order of accuracy. Usually, $\widetilde{\uvec{u}}_{n}(\uvec{x})$ is retrieved via a WENO or CWENO reconstruction \cite{boscheri2019high,gaburro2021posteriori,dumbser2009very}.

Let us observe that if $N=M$ and the basis $\left\lbrace \trialspaceFVc{r}(\uvec{x}) \right\rbrace_{r=1,\dots,R}$ coincides with $\left\lbrace \trialspace{i}(\uvec{x}) \right\rbrace_{i=1,\dots,I}$, then the \APNPM scheme reduces exactly to the ADER-DG previously introduced. On the other hand, the scheme obtained for $N=0$, i.e., with a piecewise constant approximation of $\uvec{u}_{n}(\uvec{x})$ over $\overline{\Omega}$, is the ADER-FV scheme. One can observe that the corrector, in such case, corresponds to an explicit $(M+1)$-th order accurate FV step.
All the schemes obtained for $0<N<M$ are alternatives which vary between these two schemes.
% The ``fancy'' schemes obtained for $0<N<M$ are some "reconstructed" DG schemes.

Finally, let us notice that the predictor of such methods, being formally unchanged with respect to the original formulation, can also be seen as a DeC method in which one order is achieved at each iteration until $M+1$. %In the simulations, we will present examples of both ADER-DG and ADER-FV. %NOT NOW, WHY NOW, HERE YOU ARE DESCRIBING THE METHODS

%\begin{remark}[On the choice of the notation]
%Concerning the name ADER-PNPM, in \cite{dumbser2008unified}, where such methods have been introduced, $N$ represents the polynomial degree of the spatial basis functions $\left\lbrace \trialspaceFVc{r} \right\rbrace_{r=1,\dots,R}$ used for the representation of $\uvec{u}_n(\uvec{x})$ and, hence, in the corrector, while, $M$ represents the polynomial degree of the spatial basis functions $\left\lbrace \trialspace{i} \right\rbrace_{i=1,\dots,I}$ and of the temporal basis functions $\left\lbrace \trialtime{m} \right\rbrace_{m=0,\dots,M}$ used in the predictor. 
%%In the context of this work, in order to avoid confusion with the time step index $n$, we have used $S$ in place of $N$.
%Therefore, \APNPM with the same spatial basis for the predictor and the corrector is nothing but the previously introduced ADER-DG, while, ADER-$\pnpm{0}{M}$ is actually ADER-FV.
%\end{remark}
\section{New efficient ADER schemes}\label{chap:NEWADER}
In this section, we will explain how to apply the novel modification to the described ADER framework.
We will first introduce the efficient ADER-DG-u, obtained by simply matching the order of the space-time reconstruction in each predictor iteration with the order of accuracy achieved in the same iteration, without spoiling the original order of accuracy. 
Afterwards, we will explain how such p-adaptivity strategy can be exploited to prescribe structure preservation by introducing the DOOM approach. We will focus on the ADER-DG scheme, bearing in mind that the same modifications hold true for ADER-FV and \APNPM schemes as well.
%adaptivity can be naturally embedded in this context getting efficient arbitrary high order adaptive schemes.

\subsection{Modification of ADER-DG (ADER-DG-u)}\label{chap:MOD}
We propose to change the predictor of ADER-DG by increasing the polynomial degree of the reconstruction of the numerical solution at each iteration $p$ according to the order of accuracy achieved in that specific iteration.
In particular, we define for any $p$ the general local basis $\lbrace\trialspacetimeiter{\ell}{(p)}(\uvec{x},t)\rbrace_{\ell=1,\dots,L^{(p)}}$ given by the tensor product of space basis functions $\trialspaceiter{i}{(p)}(\uvec{x})$ and time basis functions $\trialtimeiter{m}{(p)}(t)$ of degree $p$. We also define the functional spaces generated by these bases as $X^{(p)}:=\left( \mathrm{span}\lbrace \vartheta^{\ell,(p)}(\uvec{x},t) \rbrace_{\ell=1,\dots,L^{(p)}} \right)^Q$. %To be more precise we specify the dimension Q
\begin{remark}[On the spaces $X^{(p)}$]
According to our definitions of the operators \eqref{eq:ADER_l1} and \eqref{eq:ADER_l2}, formally, the spaces $X^{(p)}$ in the ADER context should be spaces of coefficients of the discrete solution. However, by definition, such spaces are in bijection with the functional spaces of $Q$-dimensional polynomials whose scalar components are spanned by the bases $\lbrace\trialspacetimeiter{\ell}{(p)}(\uvec{x},t)\rbrace_{\ell=1,\dots,L^{(p)}}$. Since in this context referring to the polynomial degree of the numerical solution in each step of the process provides a clearer overview of the method, as an abuse of notation we denote directly $X^{(p)}$ as the functional space associated to the corresponding coefficients, bearing in mind the aforementioned bijection.
\end{remark}

Then, the main procedure at the iteration $p$ passes from the space-time representation coefficients $\undu^{(p-1)}$ with respect to the basis of $X^{(p-1)}$, $(p-1)$-th order accurate with respect to the analytical solution, to $\undu^{(p)}$ in $\Xp$ with accuracy $p$.
To perform this step, we first use an embedding $\mathcal{E}^{(p-1)}:X^{(p-1)}\to X^{(p)}$, for example an interpolation or an $L^2$-projection, to pass to $\undu^{*(p-1)} = \mathcal{E}^{(p-1)}(\undu^{(p-1)}) \in X^{(p)}$. This embedding should not spoil the accuracy of the reconstructed solution. 

At this point, a simple iteration of the standard method \eqref{eq:ADERDEC_iteration}, with structures \eqref{eq:ADER_predictor_structures} associated to the basis $\lbrace\trialspacetimeiter{\ell}{(p)}(\uvec{x},t)\rbrace_{\ell=1,\dots,L^{(p)}}$, results in $\undu^{(p)}$ and the related $p$-th order accurate reconstruction in $C_K$.
These structures read
\begin{align}
\label{eq:ADER_li_bis}
\begin{split}
\lopdt^{2,(p)}(\undu)&:=\undu-\left( B^{(p)} \right)^{-1}\left[\undr^{(p)}-\undspacetilde^{(p)}(\undu)\right],\\
\lopdt^{1,(p)}(\undu)&:=\undu-\left( B^{(p)} \right)^{-1}\left[\undr^{(p)}-\undspacetilde^{(p)}(\undu_0^{(p)})\right],
\end{split}
\end{align}
with
\begin{align}
\label{eq:ADER_predictor_structures_bis}
\begin{split}
B^{(p)}_{j,\ell}:=\int_K \trialspacetimeiter{\ell}{(p)} (\uvec{x},t_{n+1})\trialspacetimeiter{j}{(p)}&(\uvec{x},t_{n+1}) d\uvec{x} - \int_{C_K}\!\!\! \trialspacetimeiter{\ell}{(p)}(\uvec{x},t) \frac{\partial}{\partial t} \trialspacetimeiter{j}{(p)}(\uvec{x},t) d\uvec{x}dt,\\
\undu :=\begin{pmatrix}
\uvec{u}^1\\
\vdots\\
\uvec{u}^{L^{(p)}}
\end{pmatrix}, \quad &\undr^{(p)} :=\begin{pmatrix}
\int_K \uvec{u}_n(\uvec{x})\trialspacetimeiter{1}{(p)}(\uvec{x},t_n)d\uvec{x}\\
\vdots\\
\int_K \uvec{u}_n(\uvec{x})\trialspacetimeiter{L^{(p)}}{(p)}(\uvec{x},t_n)d\uvec{x}
\end{pmatrix},\\ 
\undspacetilde^{(p)}(\undu):= &\begin{pmatrix}
\int_{t_n}^{t_{n+1}} \spacestuff^{(p)}_1(\undu,t) dt\\
\vdots\\
\int_{t_n}^{t_{n+1}} \spacestuff^{(p)}_{L^{(p)}}(\undu,t) dt
\end{pmatrix},
\end{split}
\end{align}
where $\spacestuff^{(p)}_j(\undu,t):=\int_K \E(\uvec{u}_h(\uvec{x},t),\uvec{x})\trialspacetimeiter{j}{(p)}(\uvec{x},t) d\uvec{x}$, with $\uvec{u}_h(\uvec{x},t)=\sum_{\ell=1}^{L^{(p)}}\uvec{u}^\ell \trialspacetimeiter{\ell}{(p)}(\uvec{x},t)$ $\forall (\uvec{x},t)\in C_K$ and $\undu^{(p)}_0$ some local coefficients extrapolated from the initial datum $\uvec{u}_n(\uvec{x})$ in $K$, yielding an $O(\Delta t)$-approximation of the analytical solution to our PDE in $C_K$. Again,  the definition of $\undu^{(p)}_0$ is merely formal as it cancels in the iterations.

The resulting modified ADER-DG-u iterative procedure reads
\begin{equation}
\label{eq:ADER_iteration_bis}
	\begin{split}
&\undu^{(0)} = \undu_0^{(0)},\\
&\begin{cases}
\undu^{*(p-1)} = \mathcal{E}^{(p-1)}\left(\undu^{(p-1)}\right),\\
\undu^{(p)}=\left( B^{(p)} \right)^{-1}\left[\undr^{(p)}-\undspacetilde^{(p)}(\undu^{*(p-1)})\right],
\end{cases} \text{for }p\geq 1.
	\end{split}
\end{equation}
%Again, we have that the definition of $\undu^{(p)}_0$ is merely formal and not needed in the context of the iteration. Further, 
We remark that the superscript $(p)$ in the definition of the structures in \eqref{eq:ADER_li_bis} and \eqref{eq:ADER_predictor_structures_bis} is simply referred to the iteration, the modified method is still fully explicit and all the terms at the right hand side of the iteration formula \eqref{eq:ADER_iteration_bis} can be explicitly computed.

%Concerning the procedure $\mathcal{E}^{(p)}$, a natural way to pass from $\undu^{(p)}$ to $\undu^{*(p)}$ when adopting nodal bases is through interpolation. In particular, we can proceed in the following way.  % THIS HAS ALREADY BEEN SAID ABOVE AND WE DO NOT USE INTERPOLATION!!
The accuracy evolves as follows throughout the procedure.
We start with $\undu^{(0)}$ associated to a piecewise constant $O(\Delta t)$-approximation of the solution to the PDE in $C_K$ and we perform the embedding in $X^{(1)}$ to get $\undu^{*(0)}$, still $O(\Delta t)$-accurate. Performing the first iteration via $\lopdt^{1,(1)},\lopdt^{2,(1)}$ we get $\undu^{(1)}$ yielding an $O(\Delta t^2)$-approximation of the solution in $C_K$. 
%We continue with another interpolation getting $\undu^{*(1)}$, still guaranteeing the same accuracy, and we perform another iteration through $\lopdt^{1,(2)},\lopdt^{2,(2)}$ getting $\undu^{(2)}$ with $O(\Delta t^3)$-accuracy. 
We continue iteratively with $\undu^{(p-1)}$ associated to a $(p-1)$-th order accurate approximation of the solution in $X^{(p-1)}$ spanned by polynomial bases of degree $p-1$, that is embedded in $\Xp$ obtaining $\undu^{*(p-1)}$ with the same accuracy $p-1$. This allows to compute $\undu^{(p)}$ via a DeC iteration with $\lopdt^{1,(p)},\lopdt^{2,(p)}$ achieving $p$-th order of accuracy.
\begin{remark}{(On the accuracy of the interpolation)}
%One must notice that the discretization through the tensor product of polynomials of degree $p$ in space and in time allows, in general, an order of accuracy $p+1$ with respect to the analytical solution to the PDE. Nevertheless, the interpolation accuracy is only of order $p$ and, hence, the interpolation must be performed before saturating the accuracy associated to the current polynomial basis to avoid the consequent degradation of the order.
\RIcolor{
One must notice that the discretization in the space $X^{(p)}$, corresponding to the tensor product of polynomials of degree $p$ in space and in time, allows in general a maximal order of accuracy $p+1$ with respect to the analytical solution of the PDE, corresponding to an error $O(\Delta t^{p+2})$. 
On the other hand, the embedding $\mathcal{E}^{(p)}:X^{(p)}\to X^{(p+1)}$ can be simply realized by interpolation, i.e. by evaluating the reconstruction associated to $\undu^{(p)}\in X^{(p)}$ in the nodal values defining the $X^{(p+1)}$ basis functions, when nodal bases are employed, %or by adding $\uvec{0}$ coefficients for the new modes of higher degree, when modal bases are considered.
or by an $L^2$-projection, when more general bases are considered.
In both cases, this operation can preserve at most the accuracy of order $p$, hence introducing an error of $O(\Delta t^{p+1})$.
Therefore, the embedding must be performed before saturating the accuracy associated to the current polynomial basis to avoid the consequent degradation of the order.
}
\end{remark}
Due to the previous remark, if the final polynomial degree of the bases in space and in time is fixed to $M$, it is convenient to perform $M$ iterations in the form \eqref{eq:ADER_iteration_bis} to get $\undu^{(M)}$, associated to the desired final discretization, plus a final iteration in the same space $X^{(M+1)}=X^{(M)}$ with the same structures as the ones used in the $M$-th iteration to saturate the accuracy related to such discretization getting thus $\undu^{(M+1)}$ yielding $(M+1)$-th order of accuracy.

The degrees of the bases of the spaces $X^{(p)}$, assuming a fixed final polynomial degree equal to $M$, are summarized in Table \ref{tab:space_degree} and the procedure is displayed in the following sketch

\resizebox{.9\linewidth}{!}{
  \begin{minipage}{\linewidth}
\begin{alignat*}{7}
&\undu^{(0)} \xrightarrow{\mathcal{E}^{(0)}} &&\undu^{*(0)} \xrightarrow[\lopdt^{1,(1)}]{\lopdt^{2,(1)}} &&\undu^{(1)} \xrightarrow{\mathcal{E}^{(1)}} && \undu^{*(1)} \xrightarrow[\lopdt^{1,(2)}]{\lopdt^{2,(2)}}&& \undu^{(2)} \xrightarrow{\mathcal{E}^{(2)}}
 \dots \xrightarrow[\lopdt^{1,(M)}]{\lopdt^{2,(M)}} &&\undu^{(M)} \xrightarrow[\lopdt^{1,(M)}]{\lopdt^{2,(M)}} &&\undu^{(M+1)}. \nonumber\\
&\!\!\!\!O(\Delta t) &&\!\!\!\!O(\Delta t)   &&\!\!\!\!O(\Delta t^2) && \!\!\!\!O(\Delta t^2)  &&\!\!\!\!O(\Delta t^3)  &&\!\!\!\!O(\Delta t^{M+1}) &&\!\!\!\!O(\Delta t^{M+2})
\end{alignat*}
  \end{minipage}
}

\begin{table}
	\caption{Increasing degrees of polynomial spaces $X^{(p)}$ varying the iteration}\label{tab:space_degree}\centering
	\begin{tabular}{|c|cccccccc|}\hline
		Space&$X^{(0)}$&$X^{(1)}$&$X^{(2)}$&$X^{(3)}$&\dots&$X^{(M-1)}$&$X^{(M)}$&$X^{(M+1)}$\\ \hline
		Polynomial degree &0&1&2&3&\dots&$M-1$ & $M$& $M$\\ \hline
	\end{tabular}
\end{table}

Finally, always assuming a final polynomial degree equal to $M$ in the predictor, the corrector step \eqref{eq:sys_weak_corrector_discrete} is normally performed with the $(M+1)$-th order accurate discretization given by the $M$-th degree local polynomial basis functions $\trialspace{i}^{(M)}$ used in the two last predictor iterations. This leads, in each element $K$, to the local approximation $\uvec{u}_{n+1}(\uvec{x})=\sum_{i=1}^I \uvec{c}_i^{n+1} \trialspace{i}^{(M)}(\uvec{x})$ which is $(M+1)$-th order accurate.
The computational advantage of the modified method with respect to the original formulation is clear: all the iterations but the last two are performed with matrix and vector structures which are smaller, implying the solution of smaller systems. Also the space-time discretization of $\E(\uvec{u}_h,\uvec{x})=\mathrm{div}_{\uvec{x}}\boldsymbol{F}(\uvec{u}_h(\uvec{x},t))-\boldsymbol{S}(\uvec{x},\uvec{u}_h(\uvec{x},t))$ and the orders of the quadrature formulas used in the low order iterations can be suitably chosen to decrease the related computational cost.
The only extra cost can come from the embedding between the spaces, which, for interpolations, can be easily recast as products by precomputable interpolation matrices, which can therefore be efficiently performed.

In this work, we assume modal bases in space and time. This further increases the computational advantage as, in such context, the higher order mode is easily introduced by adding zero components to $\undu^{(p)}$ thus getting
$\undu^{*(p)}=(\undu^{(p)},\uvec{0})^T$ without any other effort.

We denote this scheme by ADER-DG-u, referring to the $\alpha$DeCu schemes introduced in \cite{loredavide} with a similar technique, where u denotes the quantity that has been embedded.

\begin{remark}{(Galerkin projection)}
	In the specific context of these new modified ADER-DG methods, the embedding procedure between the spaces $X^{(p-1)}$ and $\Xp$ could be replaced by a Galerkin projection onto $\Xp$. 
	Namely, in \eqref{eq:ADER_iteration_bis} one  could skip the interpolation procedure and directly consider $\undspacetilde^{(p)}(\undu^{(p-1)})$ which is defined in each $j$-th component by the integral over $[t_n,t_{n+1}]$ of 
	\begin{equation}
		\spacestuff^{(p)}_j(\undu^{(p-1)},t):=\int_K \E(\uvec{u}^{(p-1)}_h(\uvec{x},t),\uvec{x})\trialspacetimeiter{j}{(p)}(\uvec{x},t) d\uvec{x},
	\end{equation}
with $\uvec{u}^{(p-1)}_h(\uvec{x},t) = \sum_{\ell=1}^{L^{(p-1)}} \uvec{u}^{\ell,(p-1)} \vartheta^{\ell,(p-1)}(\uvec{x},t)$.
	This mismatch between the spaces of the explicit term and of the test functions permits the evolution to the next space $\Xp$. This is particularly convenient as there would be no interpolation, whose cost is however negligible with respect to the rest of the scheme.
	%This is particularly convenient as the interpolation, whose cost is however negligible with respect to the rest of the scheme, would not be present at all. 
	For modal bases the two approaches are equivalent.
\end{remark}

\begin{remark}{(Space-time accuracy)} Since the ADER schemes are one-step fully discrete predictor-corrector methods on space-time control volumes $C_K$, the order of accuracy in space and time is simultaneously evolved, thus if the iterative solution $\undu^{(p)}$ is of order $O(\Delta t^{p+1})$ in time, it is also accurate $O(h^{p+1})$ in space, assuming a suitable CFL condition linking $\Delta t$ and $h$. This is omitted to lighten the notation.	
\end{remark}

As already remarked, since the predictor of the \APNPM is identical to the one of the standard ADER-DG, we can analogously introduce the \APNPM\!\!-u and, as a particular case, the ADER-FV-u methods in a straightforward way.

\RIIcolor{
\begin{remark}{(On the memory and computational differences of ADER-DG-u)}
	The implementation of ADER-DG and ADER-DG-u can be performed in various ways. 
	The evolution structures of the predictor might be either pre-computed at the beginning of the simulation or on-the-fly at each time iteration (this is necessary for Lagrangian codes on moving meshes). 
	We make use of polygonal meshes, but on triangular meshes all operators can be pre-computed with minimal storage on a reference element.
	
	For ADER-DG-u with respect to ADER-DG, there might be an increase in cost and storage for the different iterative structures, e.g. $B^{(p)}$, only if nodal basis functions are used. 
	In the considered case, where modal polynomial basis functions are adopted, the iterative structures are simply constituted of slices of the highest order structures, that would be anyway computed for ADER-DG. 
	Hence, there is no extra computation nor memory storage to be considered.
	In the modal case, these costs can be reduced by a pre-computation of all the operators and, in case of triangular meshes, with computations only on a reference element.
	
	%% LORENZO's
%Assuming a maximal polynomial degree $M$, the fixed structures associated to the different polynomial degrees, e.g., the matrices $B^{(p)}$, can be pre-computed and stored to speed up the algorithm.
%%
%However, it could be also useful at the implementation level not to use iteration-specific vectors and matrices, but rather to allocate unique structures used throughout the whole iterative procedure. The structures should be allocated with dimensions associated to the final polynomial degree $M$. The needed components could be then filled iteration by iteration. In this way, the memory requirements of the modified method would not exceed the ones associated to the original formulation.
%%
%Nevertheless, this would require to compute the fixed structures on the fly iteration by iteration.
%%
%This could be efficiently avoided in the context of modal bases, pre-computing the fixed structures for the final polynomial degree and suitably considering the needed components iteration by iteration.
\end{remark}
}

\subsection{DOOM limiter based on adaptivity }\label{chap:DOOM}
In the context of the novel schemes, it is very natural to introduce a limiter that guarantees some structural properties of the solution. The limiter will be denoted by Discrete Optimally increasing Order Method (DOOM), as it will stop the iteration process in the predictor at an optimal value.

We consider the ADER-FV-u scheme, to inherit the robustness of FV formulations and the far less restrictive CFL constraints which are suitable for large scale simulations, and we introduce an adaptive criterion.
We fix a final number of iterations $P=M+1$, corresponding to $(M+1)$-th order of accuracy, and we perform the local predictor iterations as prescribed in the context of the ADER-FV-u scheme but, in contrast with the standard method, we check for the non-violation of some physical constraints (for example the positivity of density and pressure in hydrodynamics) of the computed solution. If at iteration $p$, with $1\leq p\leq M+1$, the computed $\undu^{(p)}$ does not fulfill some of the mentioned constraints, the solution is rejected and $\undu^{(p-1)}$ is assumed to be the output of the iterative procedure for the correction step. 
Let us notice that, in the worst case, considering $\undu^{(0)}$ in the correction step in a given region of $\Omega$ leads to a standard first order Godunov scheme which is, indeed, reliable. A sketch of the limiter is displayed in Algorithm \ref{algo:DOOM}, in which $\overline{\uvec{u}}_n$ represents the local constant value of $\uvec{u}_n(\uvec{x})$ in a cell $K$ in the ADER-FV-u (and ADER-FV) context.
\begin{algorithm}
	\caption{DOOM limiter for ADER-FV-u on a cell $K$}\label{algo:DOOM}
	\begin{algorithmic}
		\Require $\overline{\uvec{u}}_n$
		%\State $\undu^{1,(0)} = \frac{1}{\vert K \vert} \int_K \uvec{u}_n(x) dx$
		\State $\undu^{(0)} = \overline{\uvec{u}}_n$
		\For{$p=1,\dots,M+1$}
		\State $\undu^{*(p-1)} = \mathcal{E}^{(p-1)}(\undu^{(p-1)})$
		\State $\undu^{(p)}=\left( B^{(p)} \right)^{-1}\left[\undr^{(p)}-\undspacetilde^{(p)}(\undu^{*(p-1)})\right]$
		\If{$\undu^{(p)}$ does not meet the criteria}
		\State \texttt{return} $\undu^{(p-1)}$
		\EndIf
		\EndFor
		\State \texttt{return} $\undu^{(M+1)}$
	\end{algorithmic}
\end{algorithm}

The strategy may remind the \textit{a posteriori} MOOD technique \cite{clain2011high,diot2012improved,boscheri2015direct,bacigaluppi2019posteriori} with some fundamental differences. The low order acceptable solution $\undu^{(p-1)}$ has been computed before $\undu^{(p)}$, as it was a necessary step towards the increase of an order of accuracy. Moreover, the order of accuracy is automatically pushed as much as possible without violating the physical constraints: in fact $\undu^{(p)}$, possibly rejected, is computed if and only if $\undu^{(p-1)}$ was reliable. This avoids the risk of an over-diffusion in having the safe low order scheme guaranteeing an accuracy lower than the one actually achievable. Therefore, it is then possible to preserve some physical properties through this procedure as explained in the following proposition.

\begin{proposition}[ADER-FV-u with DOOM property]
	Suppose that the FV scheme preserves a property $\mathscr{P}$. Suppose that the property $\mathscr{P}$ is checked in the DOOM admissibility criteria. Then the ADER-FV-u with the DOOM limiter preserves the property $\mathscr{P}$.
\end{proposition}
\begin{proof}
	In the ADER-FV-u case, $\uvec{u}_n$ is locally represented through a single constant basis function $\lambda_1\equiv 1$.
	By induction on the time step $n$, suppose that the value of $\uvec{u}_n$ in each cell $K$, i.e., $ \uvec{c}_{1}^n:=\overline{\uvec{u}}_n$ of \eqref{eq:reconstruction_at_tn_bis}, verifies the property $\mathscr{P}$.
	Then, independently of the (WENO) reconstruction of $\uvec{u}_n$ used to obtain the polynomial $\widetilde{\uvec{u}}_{n}(\uvec{x})$ for computing the related integrals in \eqref{eq:ADER_predictor_structures_bis} with the desired accuracy, $\undu^{(0)}=\uvec{c}_{1}^n$ still is the local value of the original $\uvec{u}_n$ and fulfills the property $\mathscr{P}$.
	Then, during the ADER-FV-u DOOM procedure, $\undu^{(p)}$ are kept in the iterations only if property $\mathscr{P}$ is fulfilled.
	Hence, property $\mathscr{P}$ holds for the final predictor $\uvec{u}_h$ and, in the corrector step \eqref{eq:sys_weak_corrector_discrete}, we perform the FV method using $\uvec{u}_h$. So, $\uvec{u}_{n+1}$ still fulfills property $\mathscr{P}$.
\end{proof}

At the moment, this procedure does not guarantee to preserve the property for \APNPM with $N>0$, due to the fact that the corrector in such case is not an explicit FV step. However, the authors are working on new structure preserving strategies for such schemes. Moreover, in the simulations of this work, only positivity of density and pressure is checked with the DOOM limiter, but other properties like discrete local maximum principle or entropy inequalities \cite{kuzmin2020entropy,kuzmin2020entropycg,hajduk2021monolithic,abgrall2021relaxation} can be ensured.

\begin{remark}[Other applications of the p-adaptivity]
	The adaptive nature of the novel methods can be exploited also for other applications.
	In particular, the approach can approximate the exact solution with arbitrary precision as $p\to +\infty$ and 
	it is not constrained to a maximum degree $M$ and the related approximation accuracy. 
	Within this framework, it is easy to design efficient arbitrary high order adaptive schemes, 
	as in \cite{loredavide} in a DeC context for ODEs, 
	choosing the stopping criterion for the iterations in accordance with the iteration error. 
	% hp-adaptivity
	Also hp-adaptivity can be introduced in this framework. As soon as the DOOM limiter requires low order steps, it is possible to locally use h-adaptivity to recover for the lost accuracy.
	These applications are already object of study of the authors, but they will not be treated in this work.
\end{remark}

%\lore{In particular, it is worth remarking that, due to the intrinsic adaptive nature of the introduced schemes, other applications concerning structure preserving and adaptivity are indeed possible and will be investigated in future works.}
%altri structure (kuzmin), adaptivity

%\lore{This last phrase was a bit dramatic. If DG preserves a property $\mathscr{P}$, the same proof holds (modulo some details). Here I'd spend some words on the potentiality of the new methods: we are just focusing on this simple particular application but, due to their intrinsic adaptive nature, other approaches can be easily embedded in this framework and the authors are working on that. But of course feel free to delete this comment.}
\section{Numerical results}\label{chap:Num}
%SCHEMES INTRODUCED
%ADER-DG-u and ADER-FV-u
%ADER-DG and ADER-FV
%with tolerance?
In this section, we will report the numerical results of several tests performed in order to validate the accuracy and the robustness of the novel ADER-DG-u and ADER-FV-u methods, i.e., \APNPM\!\!-u respectively with $N=M$ and $N=0$.
In order to quantify the obtained speed-up in terms of computational time, they will be compared with the state-of-the-art ADER-DG and ADER-FV methods \cite{boscheri2022continuous}, characterized by a fixed polynomial degree along the whole iterative procedure and a convergence criterion $\norm{\undu^{(p)}-\undu^{(p-1)}}_\infty<{\texttt{tol}}$ to stop the predictor iterations \eqref{eq:fixed_point}, where here we assume $\texttt{tol}=10^{-12}$.
%\lore{Plus, do we want to add here the notation ADER-DGnumero, ADER-FVnumero ADER-DG-unumero, ADER-FV-unumero, maybe it is better so that we do not constantly write P0P3 or whatever.}

We adopt the following notation: for each method we explicitly specify the formal order of accuracy. Therefore, ADER-DG($M+1$) and ADER-FV($M+1$) represent the original methods with predictor spatial and temporal basis functions of degree $M$ guaranteeing $(M+1)$-th order of accuracy. In the context of ADER-DG-u($M+1$) and ADER-FV-u($M+1$), instead, $M$ is the final degree of the predictor spatial and temporal basis functions at the end of the iteration process still leading to accuracy $M+1$. %When the DOOM limiter is active for ADER-FV-u($M+1$)

%EQUATIONS INTRODUCED 
We will focus on the Euler and compressible Navier--Stokes equations.
The Euler equations are a system of hyperbolic PDEs in the form \eqref{eq:sys} given by
\begin{equation}\label{eq:Euler}
	\uvec{u}=\begin{pmatrix}
		\rho\\
		\uvec{q}\\
		E
	\end{pmatrix}, \quad \uvec{F}(\uvec{u}) = \begin{pmatrix}
		\uvec{q}\\
		\rho \uvec{v}\otimes\uvec{v} +p \mathbb{I}\\
		\uvec{v} (E+p)
	\end{pmatrix}, \quad \uvec{S}(\uvec{x},\uvec{u})=\uvec{0},
\end{equation}
where $\rho$ is the density, $\uvec{q}\in \mathbb{R}^D$ the momentum, $E$ the energy, $p$ the pressure, $\uvec{v}=\frac{\uvec{q}}{\rho}\in \mathbb{R}^D$ the velocity of the flow and $\mathbb{I}\in \mathbb{R}^{D\times D}$ is the identity matrix. The system is completed by specifying the closure equation of state $E=\frac{p}{\gamma -1} +\rho\frac{ \norm{\uvec{v}}_2^2}{2}$, where $\gamma=\frac{c_p}{c_v}$ is the adiabatic coefficient defined as the ratio between the specific heats at constant pressure and volume and is here assumed to be $\gamma=1.4$.

The more general compressible Navier--Stokes equations are obtained by keeping the viscosity effects into account and, for ideal gases, are defined by
\begin{equation}\label{eq:NS}
	\uvec{u}=\begin{pmatrix}
		\rho\\
		\uvec{q}\\
		E
	\end{pmatrix}, \quad \uvec{F}(\uvec{u}) = \begin{pmatrix}
		\uvec{q}\\
		\rho \uvec{v}\otimes\uvec{v} +\uvec{\sigma}(\uvec{u},\nabla_{\uvec{x}} \uvec{u})\\
		\uvec{v} (E \mathbb I +\uvec{\sigma}(\uvec{u},\nabla_{\uvec{x}} \uvec{u}))- \kappa \nabla_{\uvec{x}} \Temp
	\end{pmatrix}, \quad \uvec{S}(\uvec{x},\uvec{u})=\uvec{0},
\end{equation}
where $\uvec{\sigma}(\uvec{u},\nabla_{\uvec{x}} \uvec{u})$ denotes the stress tensor, $\kappa$ is the heat conduction coefficient and $\Temp$ represents the temperature, while, the other terms have the same meaning as in the context of the Euler equations.
In particular, the stress tensor $\uvec{\sigma}(\uvec{u},\nabla_{\uvec{x}} \uvec{u})$ is given, under the Stokes hypothesis, by 
\begin{equation}
	\uvec{\sigma}(\uvec{u},\nabla_{\uvec{x}} \uvec{u}) = \left(p +\frac23 \mu ~ \mathrm{div}_{\uvec{x}} \uvec{v}\right)\mathbb I -\mu \left(\nabla_{\uvec{x}} \uvec{v} + \nabla_{\uvec{x}} \uvec{v}^T\right),
\end{equation}
with $p$ being the pressure of the fluid and $\mu$ the dynamic viscosity that we assume to be constant. The heat conduction coefficient $\kappa$ is linked to the viscosity coefficient through the Prandtl number Pr with the following law
\begin{equation}
	\kappa = \frac{\mu \gamma c_v}{\text{Pr}}.
\end{equation}
where again $\gamma = \frac{c_p}{c_v}$. A thermal and a caloric equation of state are needed for the closure of \eqref{eq:NS}. %The temperature can then be canceled out by these equations. 
For an ideal gas those are
\begin{equation}
	\frac{p}{\rho} = R \Temp, \quad \frac{e}{\rho} = c_v \Temp,
\end{equation}
with $R$ being the specific gas constant and $e=E-\rho \frac{\norm{\uvec{v}}^2_2}{2}$ the internal energy.

We will consider two-dimensional ($D=2$) problems, hence $\uvec{v}:=(u,v)^T$.

If not stated otherwise, the CFL number is set to $\CFL=0.5$, and the time step is computed according to an explicit stability condition which is given by
%\begin{equation}
%	\dt \leq \CFL \, \frac{\min \limits_{i \in K} h_i}{(2M+1) \max \limits_{i \in K} \left( \| \lambda^{\max,i} \| + 2 \| \lambda_v^{\max,i} \| \frac{2M+1}{h_i} \right)},
%	\label{eqn.timestep}
%\end{equation}
\begin{equation}
	\dt \leq \CFL \, \frac{\min \limits_{ K \in \tau_h} h_K}{(2N+1) \max \limits_{K \in \tau_h} \left(  \max\limits_{\substack{\uvec{x}\in K}} \norm{\uvec{\lambda} }_\infty + 2 \max\limits_{\substack{\uvec{x}\in K}} \norm{ \uvec{\lambda}_v }_\infty \frac{2N+1}{h_K} \right)},
	\label{eqn.timestep}
\end{equation}
%NB: I have put N so that there is no ambiguity. 
%For ADER-FV it is N, for ADER-DG it can only be M
where $N$ represents the degree of the chosen polynomial representation, while $\uvec{\lambda}=\left(\norm{\uvec{v}}_2-\sqrt{\gamma \frac{p}{\rho}},\norm{\uvec{v}}_2,\norm{\uvec{v}}_2+\sqrt{\gamma \frac{p}{\rho}}\right)$ are the convective eigenvalues of the Euler system, and the viscous eigenvalues $\uvec{\lambda}_v$ are given in~\cite{ADERNSE}. The characteristic mesh size of the cell $h_K$ is given by the square root of its surface in 2D. If not stated differently, the local Lax-Friedrichs numerical flux function \cite{Rusanov:1961a} is used in the corrector step \eqref{eq:sys_weak_corrector_discrete}.

For more challenging tests, in which the density is close to zero, we will activate the DOOM limiter checking for the positivity of the density and pressure in the quadrature points and that no NaN appears in the solution. This limiter will be used only with the ADER-FV-u technique, which provably guarantees the preservation of the positivity of these quantities.

% ----------
\subsection{Numerical convergence studies} \label{ssec.conv}
% ----------

%\lore{RMK: For PDEs theory the domain is an open set.}
%\lore{RMK: Column vectors by default.}
%\lore{Q: Did we use periodic BCs for the vortex? It may be useful to specify that we consider a big-enough domain to get rid of boundary effects}

To test the accuracy of the method, we perform a convergence test on a smooth isentropic vortex~\cite{shuosher1} for the compressible Euler equations. The computational domain is $\Omega = [0,10]^2$ with periodic boundary conditions, and it is tessellated by a polygonal mesh. The vortex is centered at the initial time in $\uvec{x}_c=(x_c,y_c)^T=(5,5)^T$ and moves with a background speed of $\uvec{v}_\infty=(u_\infty,v_\infty)^T=(1,1)^T$. The initial position of the vortex, in a generic point $\uvec{x}=(x,y)^T$, can be described using the radial coordinate $r:=\norm{\uvec{x}-\uvec{x}_c}_2$ as
\begin{equation}\label{eq:vortex}
\begin{cases}
	\rho(\uvec{x},0) = (1+\delta \Temp)^{\frac{1}{\gamma-1}},\\
	\uvec{v}(\uvec{x},0)=\uvec{v}_\infty+\frac{\epsilon}{2\pi}e^{\frac{1-r^2}{2}} \begin{pmatrix}
-(y-y_c)\\(x-x_c)
\end{pmatrix},	\\
	p(\uvec{x},0) = (1+\delta \Temp)^{\frac{\gamma}{\gamma-1}},
\end{cases}
	\quad \delta \Temp = -\frac{(\gamma-1)\epsilon^2}{8\gamma\pi}e^{1-r^2},
\end{equation}
with $\Temp$ denoting the fluid temperature. The exact solution is obtained as $\uvec{u}(\uvec{x},t)=\uvec{u}(\uvec{x}-\uvec{v}_\infty t,0)$. We run the simulation until final time $t_f=1$ using the Osher-type numerical flux function \cite{OsherNC}.

\begin{figure}[!htbp]
	\begin{center}
		\begin{tabular}{cc}  
			\includegraphics[width=0.47\textwidth]{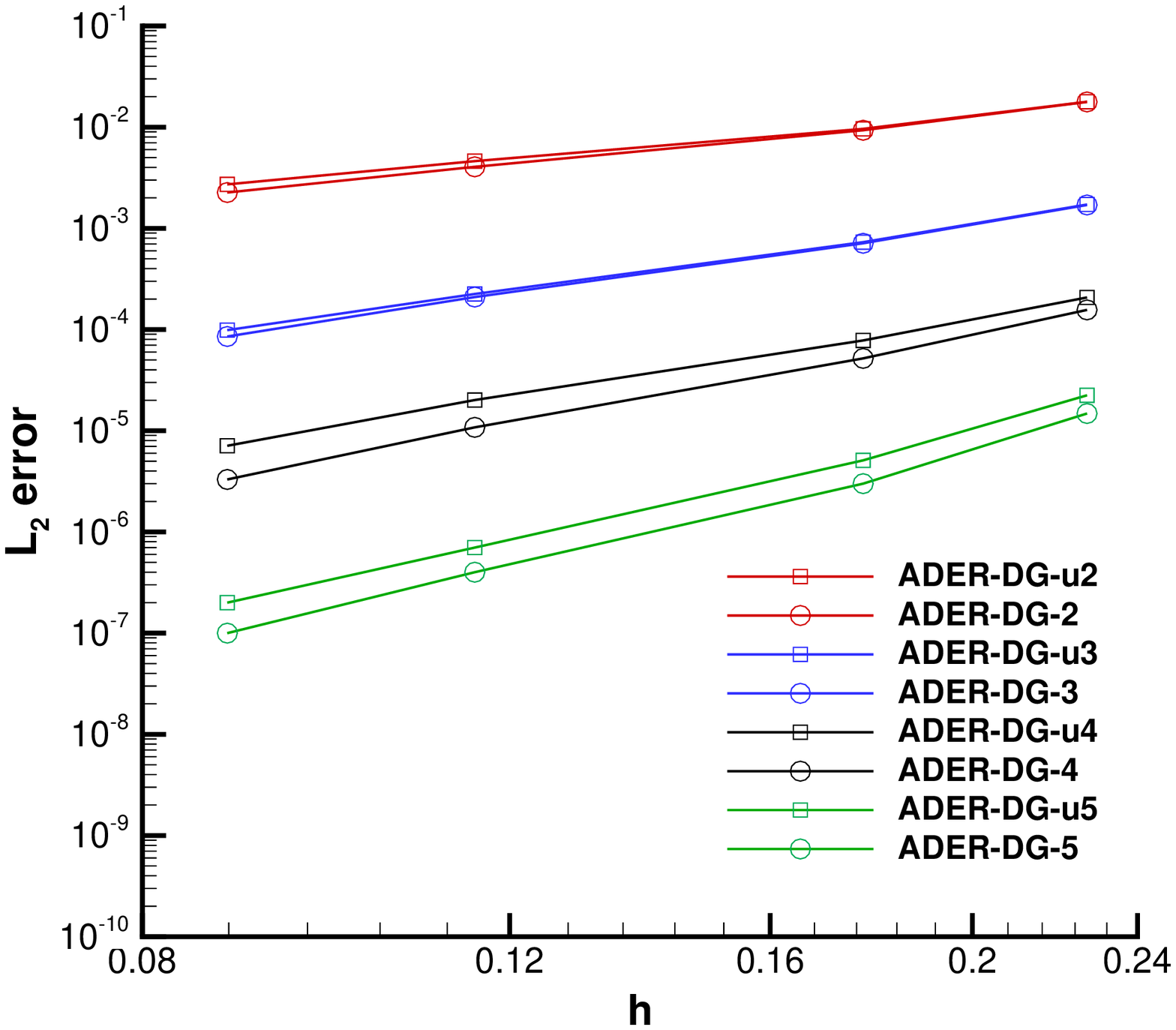} &
			\includegraphics[width=0.47\textwidth]{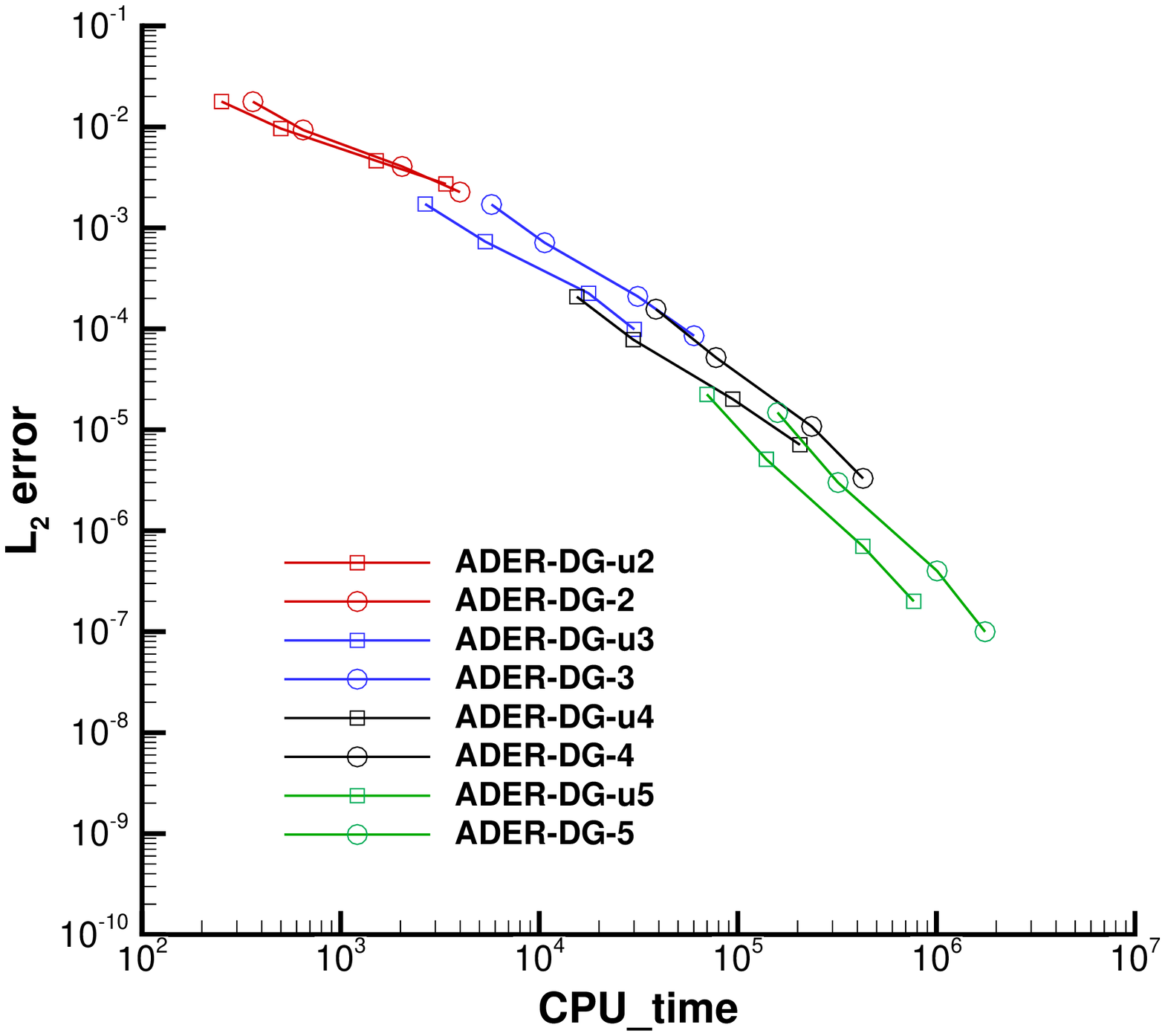} \\ 
		\end{tabular} 
		\caption{Comparison between ADER-DG-u and ADER-DG schemes from second up to fifth order of accuracy. Left: dependency of the error norm on the mesh size. Right: dependency of the error norm on the CPU time.}
		\label{fig.L2-h-time}
	\end{center}
\end{figure}

%Smooth isentropic vortex. Convergence analysis on the left and t and error with respect to computational time on the right

In Figure \ref{fig.L2-h-time}, we can observe on the left the errors of the ADER-DG and ADER-DG-u methods for different mesh sizes. All the methods achieve the formal order of accuracy. As expected, the ADER-DG-u has slightly larger errors with respect to the original ADER-DG method, as the first iterations of the predictors are done with lower order accurate operators. Nevertheless, the final error is quite comparable with the ADER-DG one and, looking at the right figure, we observe that the computational time required by ADER-DG-u for such simulations is much less (for high order methods it is around half) than the one required by the competitor. The slight increase in error is hugely beaten by the computational advantage of the new ADER-DG-u schemes. Indeed, the Pareto front on the right figure is only composed by ADER-DG-u points. The results are quantitatively reported in Table \ref{tab.convRates}. 
\RIcolor{We observe that the computed orders of accuracy are very close to the expected ones. Convergence analyses with vortex-type solutions are often subjected to some loss of order of accuracy as explained in \cite{spiegel2015survey,vort}. In our case, this may also be due to an imprecise choice of the mesh parameter for our polygonal meshes: we consider the maximum internal diameter of the polygons, but for some meshes this choice might not well represent the characteristic size of the cells. The same convergence trends have been observed also in in \cite{boscheri2022continuous}.} 

%----------------------------------------
% convergence: from 2nd up to 4th order for Euler eqns
%----------------------------------------
\begin{table}[!t]  
	\caption{Numerical convergence results for the compressible Euler equations using both ADER-DG-u and ADER-DG schemes from second up to fifth order of accuracy in space and time. The errors are measured in the $L_2$ norm and refer to the variable $\rho$ (density) at time $t_{f}=1$. The absolute CPU time of each simulation is reported in seconds $[s]$.}  
	\begin{center} 
		\begin{small}
			\renewcommand{\arraystretch}{1.2}
			\begin{tabular}{c|ccc|ccc}
				\multicolumn{1}{c}{} & \multicolumn{3}{|c}{ADER-DG-u} & \multicolumn{3}{|c}{ADER-DG} \\
				\hline
				$h(\Omega)$ & $\rho_{L_2}$ & $O(\rho_{L_2})$ & CPU time & $\rho_{L_2}$ & $O(\rho_{L_2})$ & CPU time \\ 
				\hline
				\hline
				& \multicolumn{6}{|c}{Order of accuracy: $O(2)$} \\
				2.270E-01 & 1.781E-02 & -    & 2.511E+02 & 1.775E-02 & -    & 3.616E+02 \\
				1.773E-01 & 9.625E-03 & 2.49 & 4.997E+02 & 9.322E-03 & 2.61 & 6.472E+02 \\
				1.155E-01 & 4.614E-03 & 1.71 & 1.509E+03 & 4.055E-03 & 1.94 & 2.039E+03 \\
				8.786E-02 & 2.723E-03 & 1.93 & 3.387E+03 & 2.262E-03 & 2.14 & 3.989E+03 \\
				& \multicolumn{6}{c}{Order of accuracy: $O(3)$} \\
				2.270E-01 & 1.719E-03 & -    & 2.664E+03 & 1.704E-03 & -   & 5.750E+03 \\
				1.773E-01 & 7.301E-04 & 3.46 & 5.346E+03 & 7.121E-04 & 3.53 & 1.065E+04 \\
				1.155E-01 & 2.247E-04 & 2.75 & 1.773E+04 & 2.095E-04 & 2.85 & 3.133E+04 \\
				8.786E-02 & 9.871E-05 & 3.01 & 3.010E+04 & 8.542E-05 & 3.29 & 6.020E+04 \\
				& \multicolumn{6}{c}{Order of accuracy: $O(4)$} \\
				2.270E-01 & 2.076E-04 & -    & 1.547E+04 & 1.563E-04 & -    & 3.868E+04 \\
				1.773E-01 & 7.803E-05 & 3.96 & 2.975E+04 & 5.195E-05 & 4.46 & 7.766E+04 \\
				1.155E-01 & 2.013E-05 & 3.16 & 9.427E+04 & 1.085E-05 & 3.65 & 2.354E+05 \\
				8.786E-02 & 7.139E-06 & 3.80 & 2.054E+05 & 3.332E-06 & 4.32 & 4.270E+05 \\
				& \multicolumn{6}{c}{Order of accuracy: $O(5)$} \\
				2.270E-01 & 2.238E-05 & -    & 6.993E+04 & 1.475E-05 & -   & 3.171E+05 \\
				1.773E-01 & 5.080E-06 & 6.00 & 1.393E+05 & 3.002E-06 & 6.44 & 6.390E+05 \\
				1.155E-01 & 7.405E-07 & 4.49 & 4.261E+05 & 4.180E-07 & 4.60 & 2.015E+06 \\
				8.786E-02 & 2.154E-07 & 4.52 & 7.691E+05 & 1.228E-07 & 4.48 & 3.516E+06 \\
			\end{tabular}
		\end{small}
	\end{center}
	\label{tab.convRates}
\end{table}

Finally, Figure \ref{fig.speedup} depicts the speedup achieved by the novel adaptive schemes compared against the classical formulation of iterative methods, namely ADER-DG-u versus ADER-DG. As the order of accuracy increases, the speedup becomes higher obtaining efficient schemes which are up to $\approx4.5$ times faster than the classical methods. Let us notice that the formal order of accuracy is still maintained, while getting a remarkable gain in the computational efficiency.

\begin{figure}[!htbp]
	\begin{center}
		\begin{tabular}{c}  
			\includegraphics[width=0.7\textwidth]{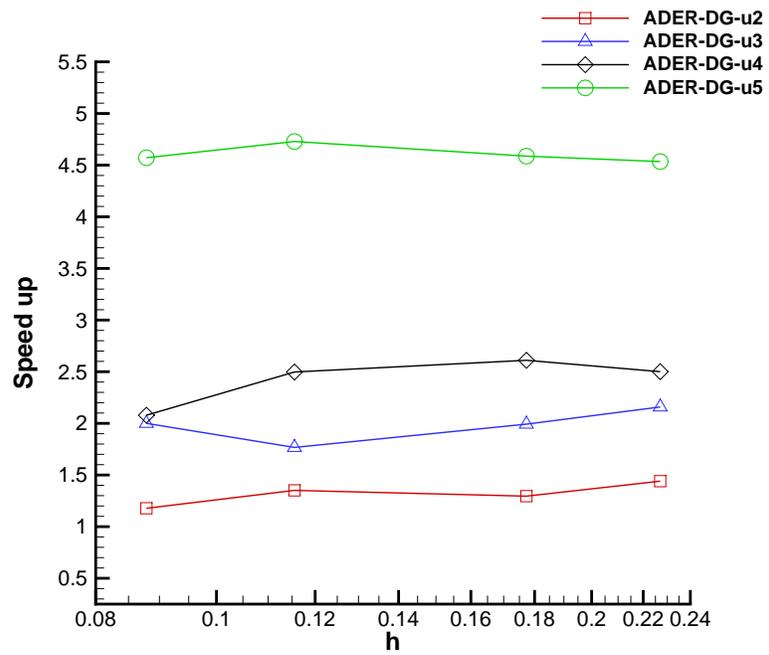}  
		\end{tabular} 
		\caption{Speedup of the ADER-DG-u schemes compared to the ADER-DG methods depending on the mesh size for different orders.}
		\label{fig.speedup}
	\end{center}
\end{figure}

% ----------
\subsection{Riemann problems} \label{ssec.RP}
% ----------
In this section, we will show the results of the ADER-FV-u$4$ scheme, i.e., with $M=3$, for some Riemann problems. The computational domain is the box $\Omega = [-0.5,0.5]\times[-0.05,0.05]$ with periodic boundary conditions in $y$ direction and Dirichlet boundaries imposed at $x=\pm 0.5$. We use an unstructured polygonal mesh made of $N_h=2226$ control volumes of characteristic mesh size of $h\approx 1/100$. Despite the one-dimensional setting of the test case, we underline that the preservation of symmetry of the solution is not trivial on unstructured meshes, where no cell boundaries are in principle aligned with the main flow velocity. We solve again the Euler equations \eqref{eq:Euler} with initial conditions given, as a function of the $x$ coordinate only, by 
%with initial conditions given by 
\begin{equation}
	\uvec{u}(x,0)=\begin{cases}
		\uvec{u}_L, & \text{if }x<0,\\
		\uvec{u}_R, & \text{else},
	\end{cases}
\end{equation}
where the values of $\uvec{u}_L$ and $\uvec{u}_R$ and the final times for the different tests are taken from \cite{toro2009riemann} and they can be found in Table \ref{tab:IC_RP}.
The velocity along the $y$-direction is set to be $v=0$ for all the tests.

The DOOM limiter is here active checking for the positivity of density and pressure and avoiding NaN.
These tests are very challenging and not all the numerical methods can stably perform on them. In particular, shocks are often not well captured or numerical oscillations appear around them and it is common that negative density or pressure values appear in the simulations, making the code crash.
\begin{table}
	\centering
	\caption{Initial conditions for Riemann problems}\label{tab:IC_RP}
	\begin{tabular}{|c||c|c|c||c|c|c||c|}
		\hline
		Test& $\rho_L$& $u_L$ & $p_L$& $\rho_R$& $u_R$ & $p_R$& $t_f$\\\hline\hline
		1  &0.445&0.698&3.528 &0.5&0&0.571   &0.14\\\hline
		2  &1&2&0.1 &1&-2&0.1     &0.8 \\\hline
		3  &1&-2&0.4 &1&2&0.4  &0.15\\\hline
		4  &1&0&1000 &1&0&100  & 0.012\\ \hline
	\end{tabular}
\end{table}

\begin{figure}[!htbp]
	\begin{center}
		\begin{tabular}{ccc} 
			\includegraphics[width=0.32\textwidth]{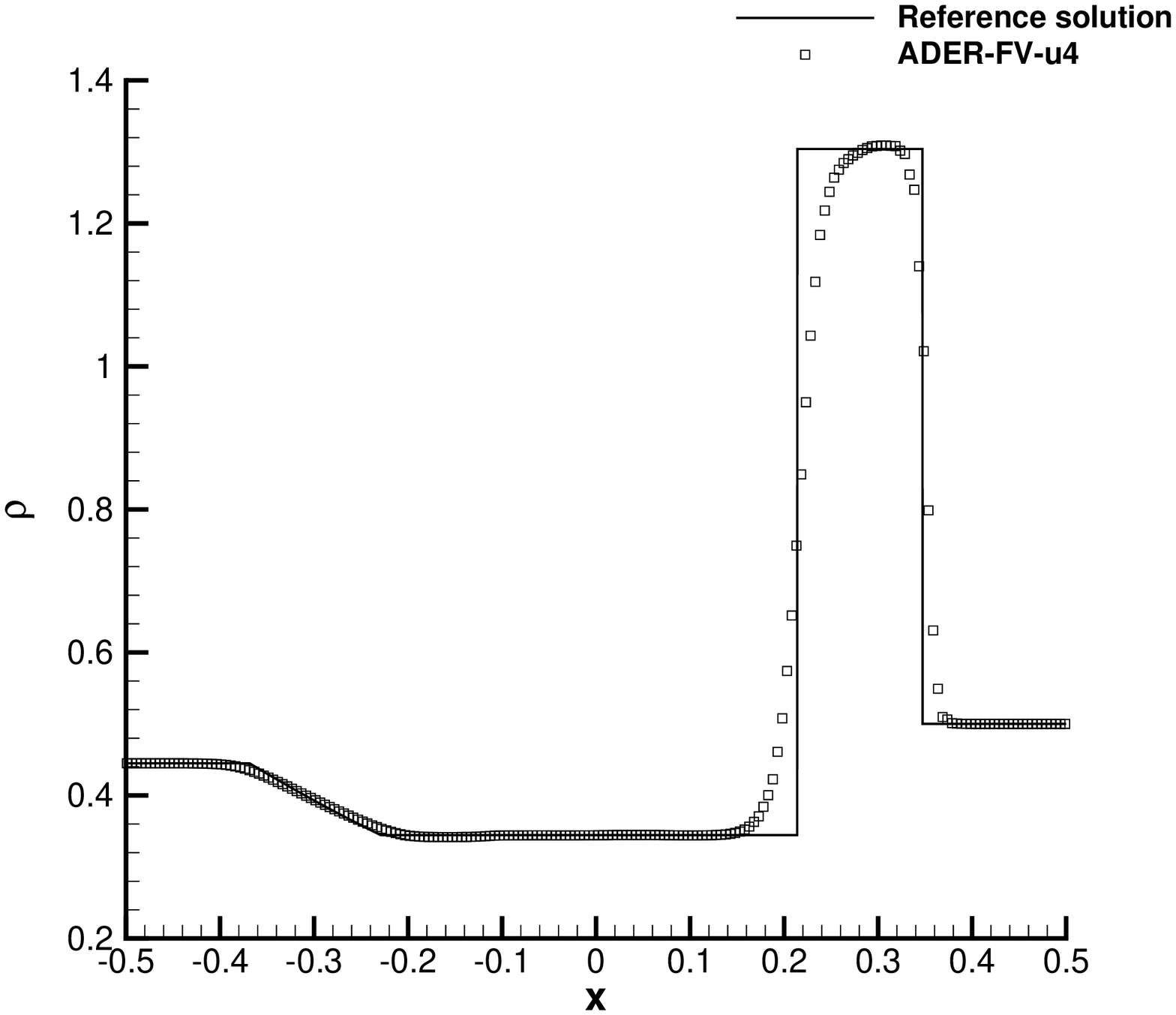} &  
			\includegraphics[width=0.32\textwidth]{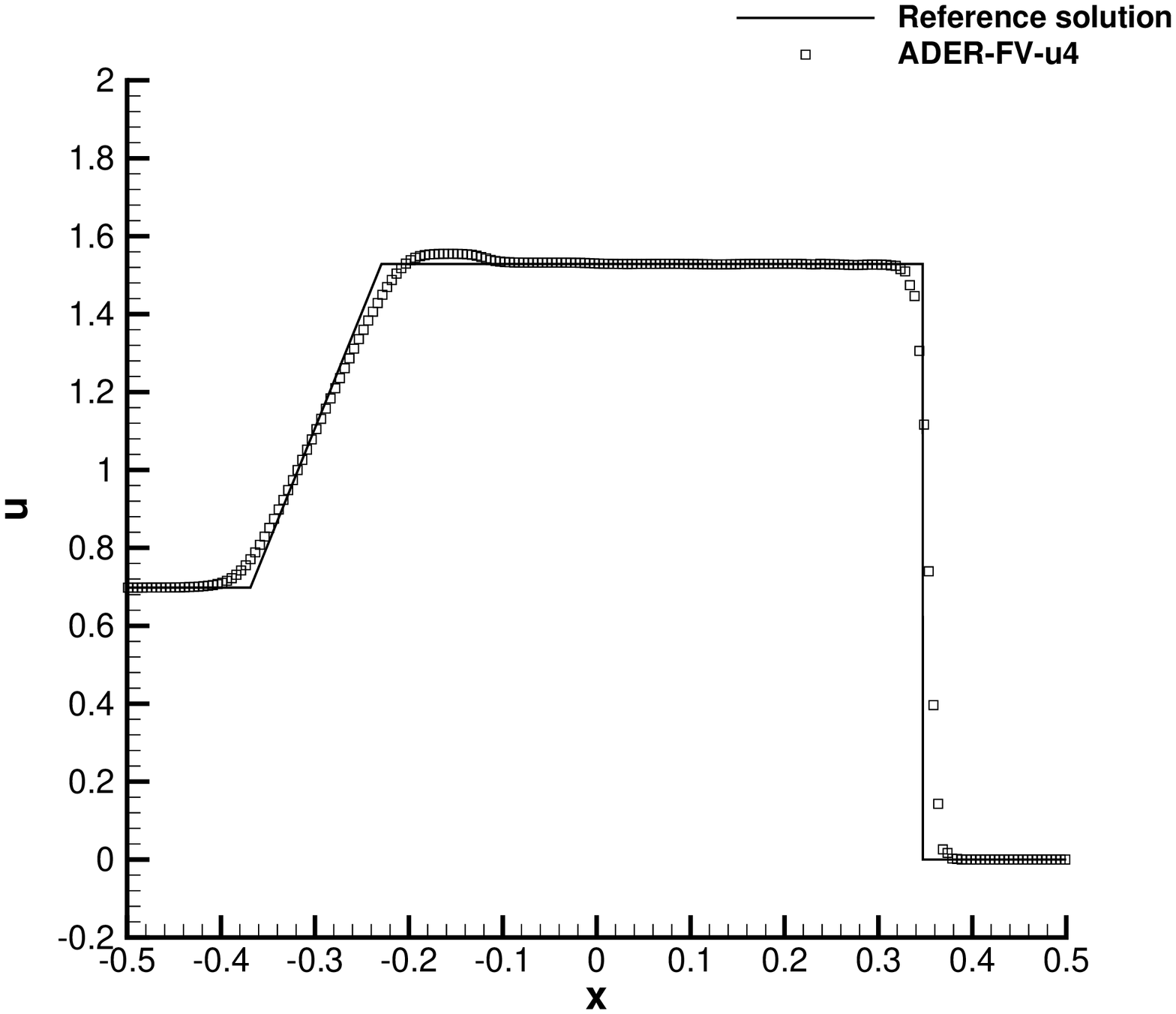}  &  
			\includegraphics[width=0.32\textwidth]{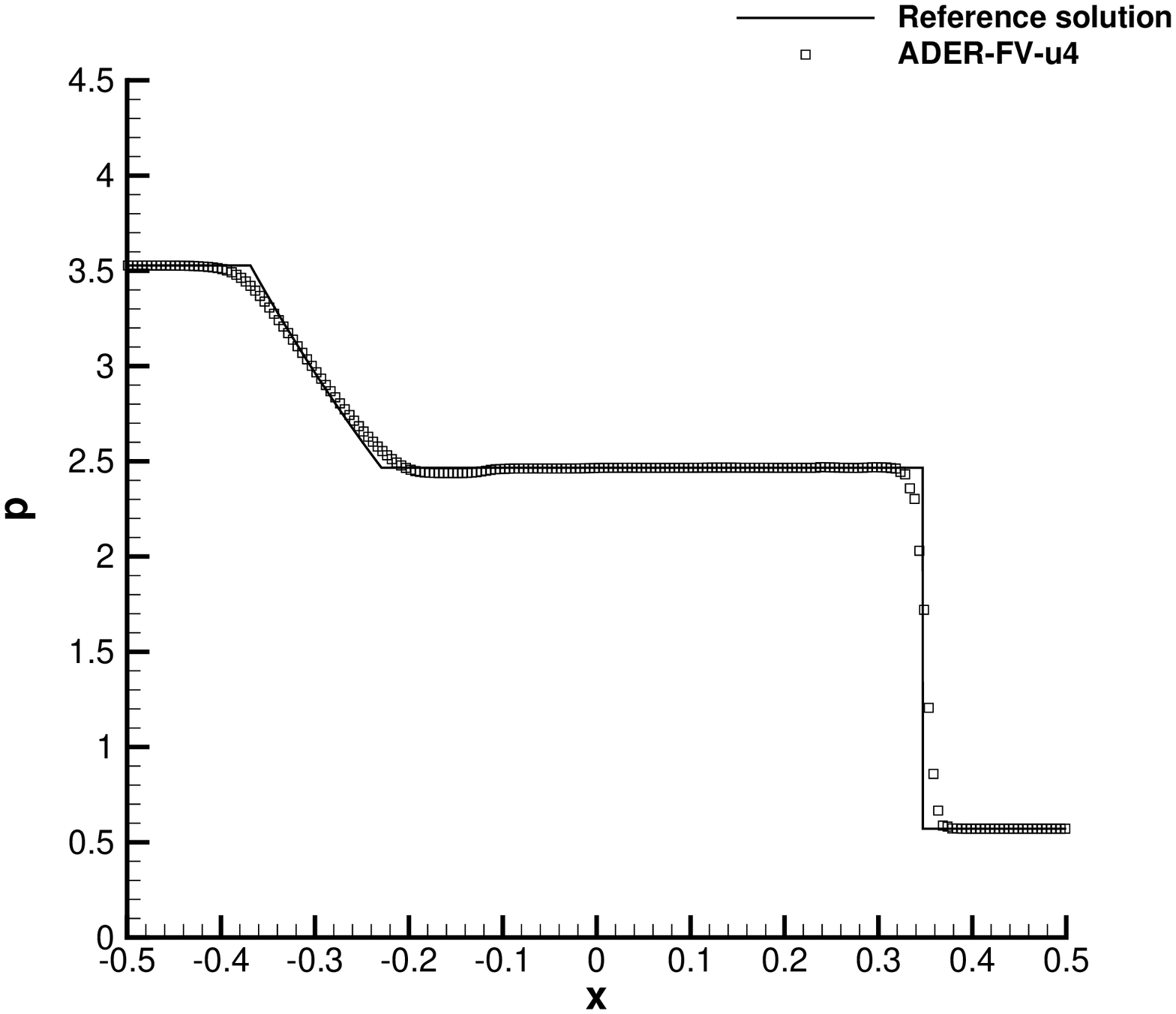} \\     
		\end{tabular} 
		\caption{Lax shock tube problem (RP1) at final time $t_f=0.14$. Comparison of density, velocity and pressure versus the reference solution for ADER-FV-u$4$ scheme.}
		\label{fig.Lax}
	\end{center}
\end{figure}
The first test (RP1) is the classical Lax shock tube problem. The initial discontinuity develops into a rarefaction wave, a contact discontinuity and a shock. In Figure \ref{fig.Lax}, we observe that the ADER-FV-u does not exhibit any oscillations around the shock and that exactly catches the speed of the discontinuities. % $\mathbb P_0\mathbb P_3$

\begin{figure}[!htbp]
	\begin{center}
		\begin{tabular}{ccc} 
			\includegraphics[width=0.32\textwidth]{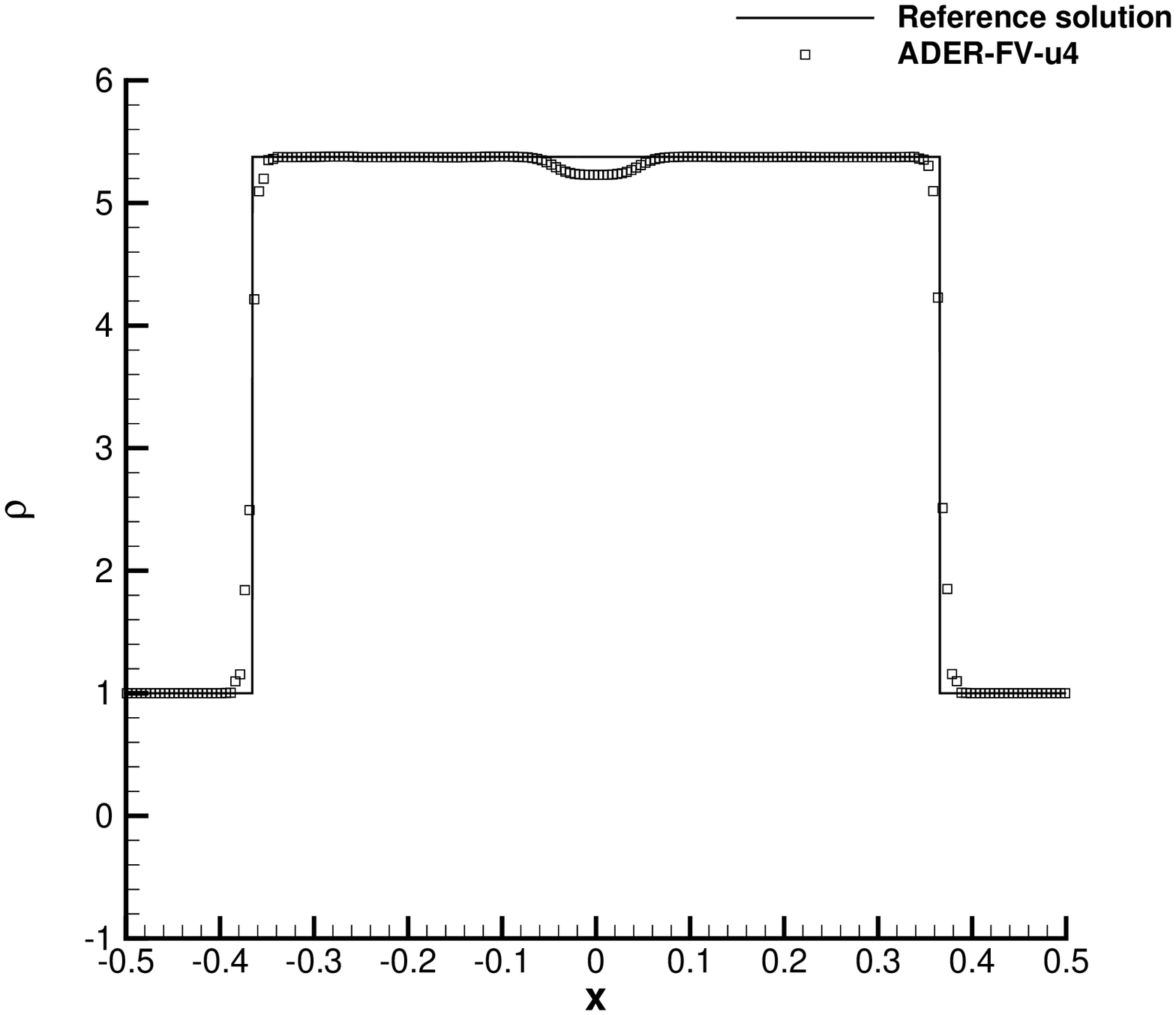}  & 
			\includegraphics[width=0.32\textwidth]{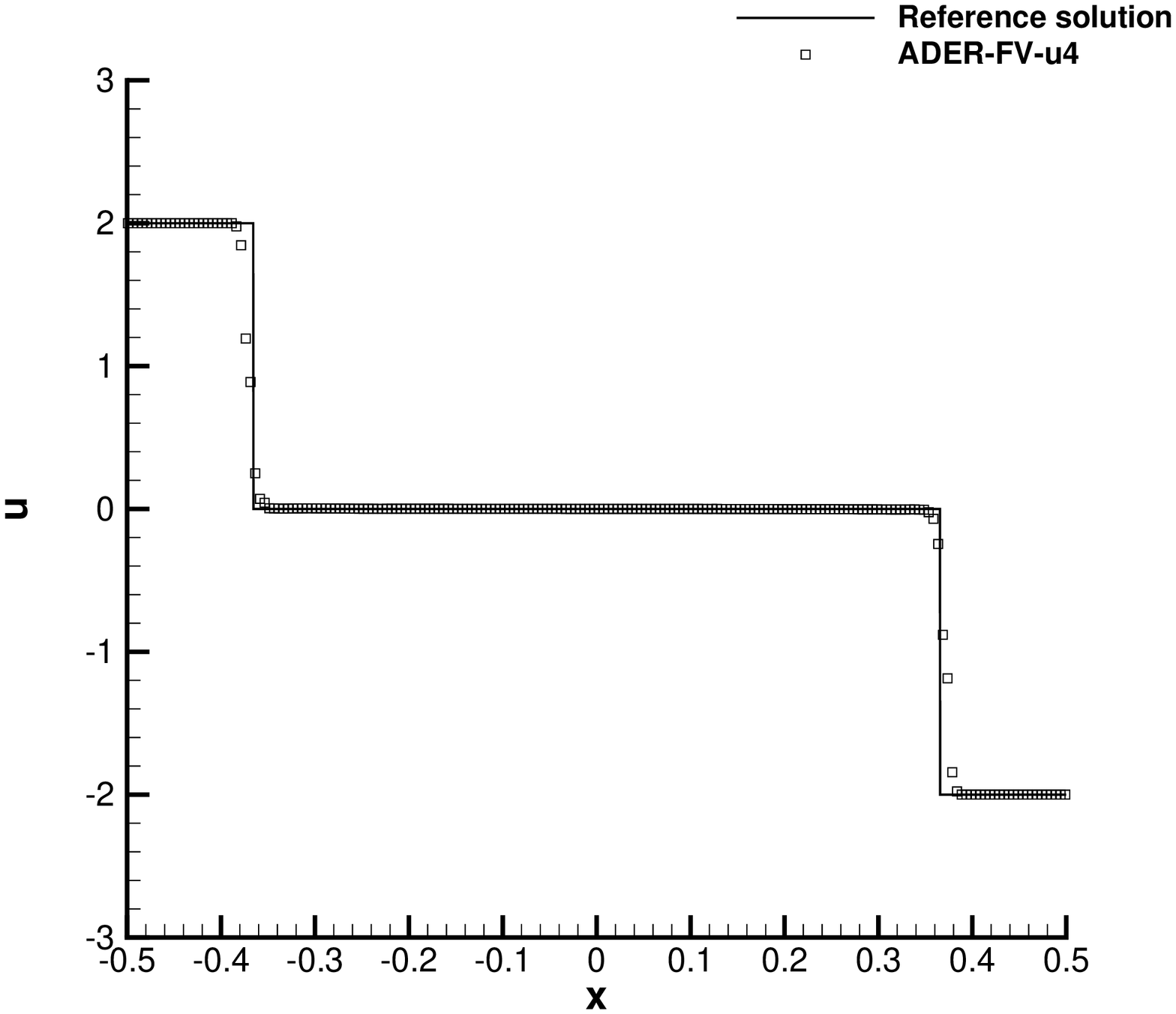}  &
			\includegraphics[width=0.32\textwidth]{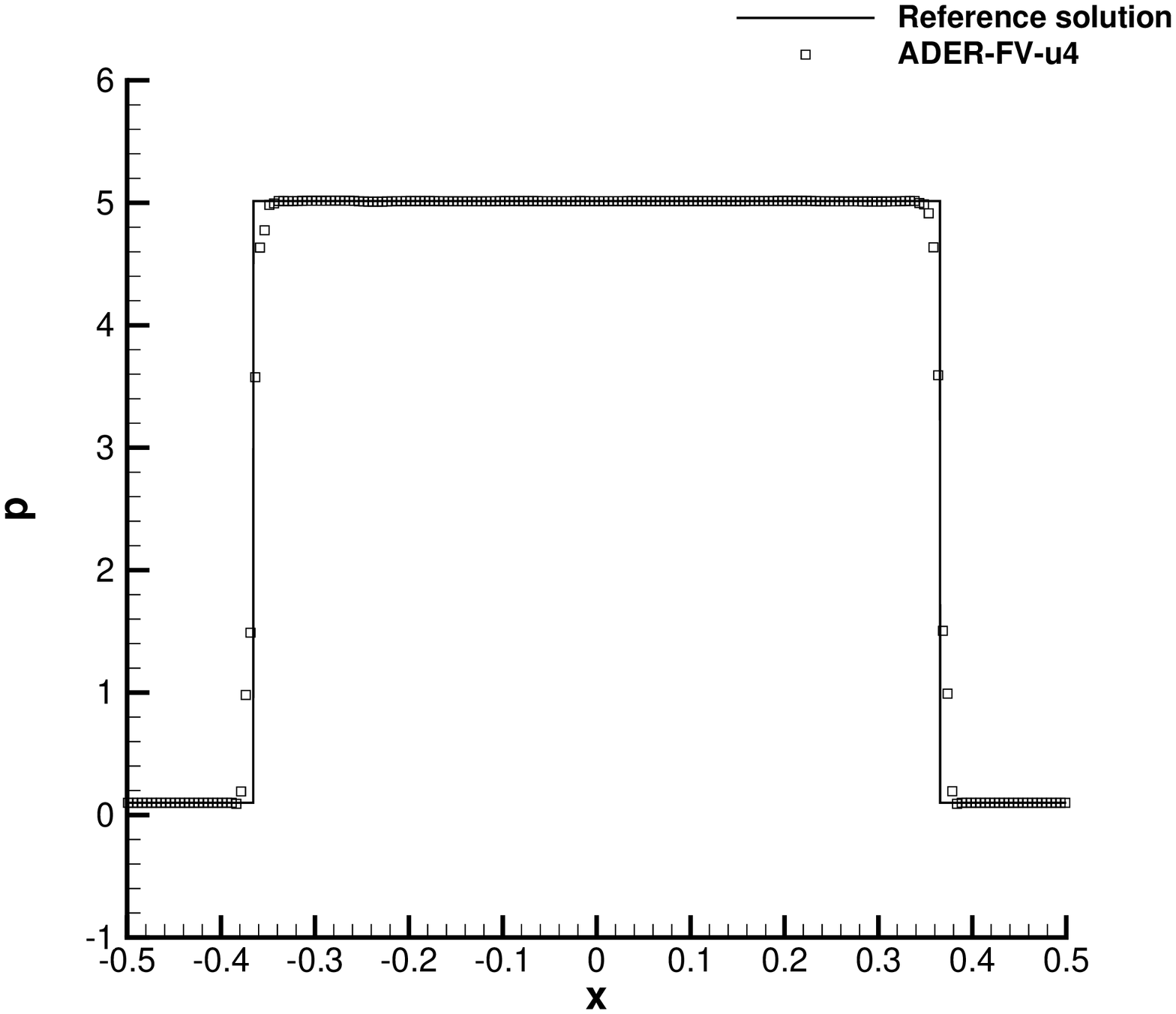} \\   
		\end{tabular} 
		\caption{Colliding shock test (RP2) at final time $t_f=0.8$. Comparison of density, velocity and pressure versus the reference solution for ADER-FV-u$4$ scheme.}
		\label{fig.2shock}
	\end{center}
\end{figure}
The second test (RP2) consists of a colliding shock test. 
The initial discontinuity in the velocity gives rise to two shocks traveling outside the domain. 
This test creates a very high density and pressure region in the middle of the domain. As it can be seen in Figure \ref{fig.2shock}, the new ADER-FV-u with DOOM limiter is able to perfectly capture the shock behavior within few cells without over/under-shootings at the sides of the shocks.

\begin{figure}[!htbp]
	\begin{center}
		\begin{tabular}{ccc} 
			\includegraphics[width=0.32\textwidth]{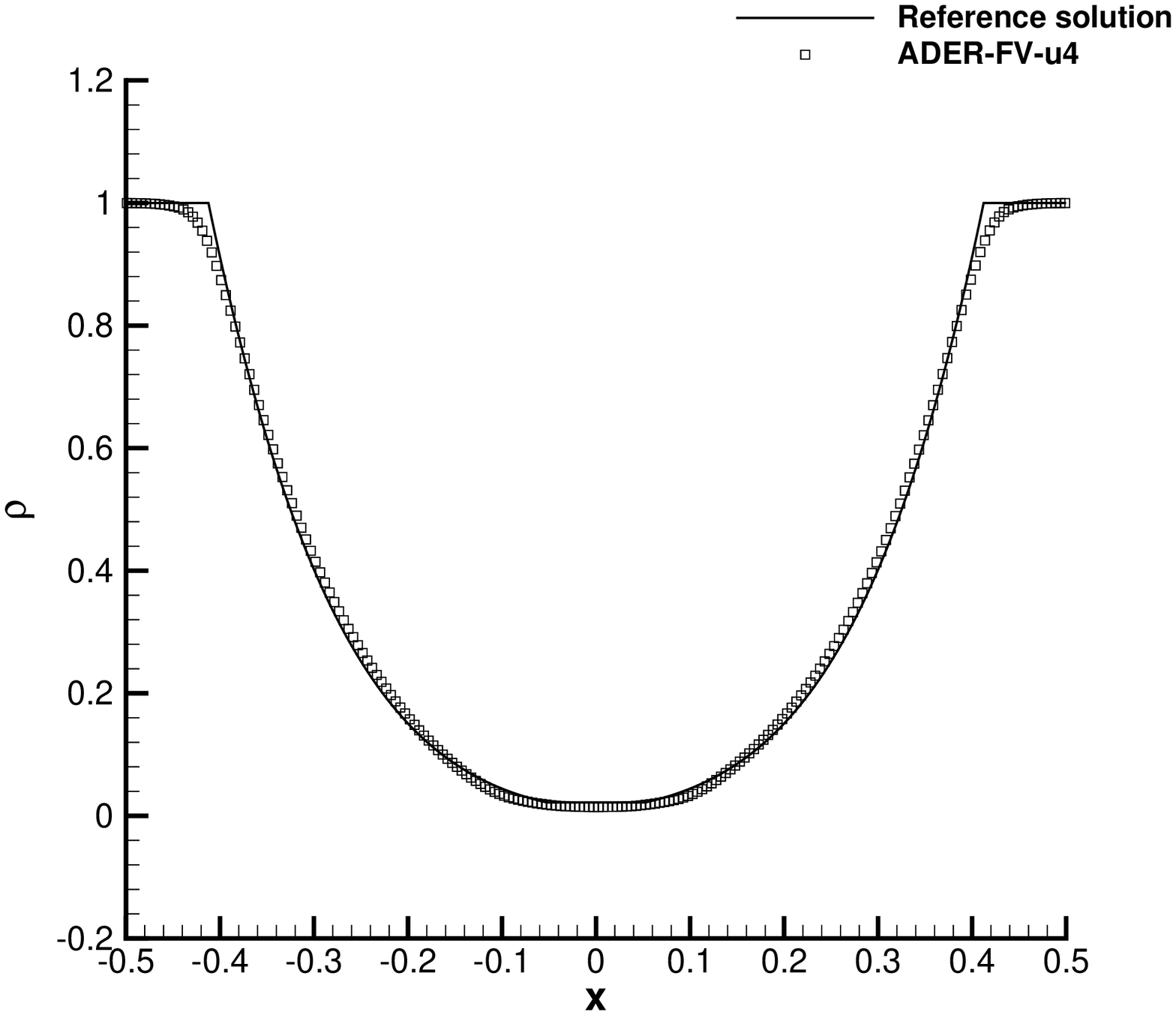}  & 
			\includegraphics[width=0.32\textwidth]{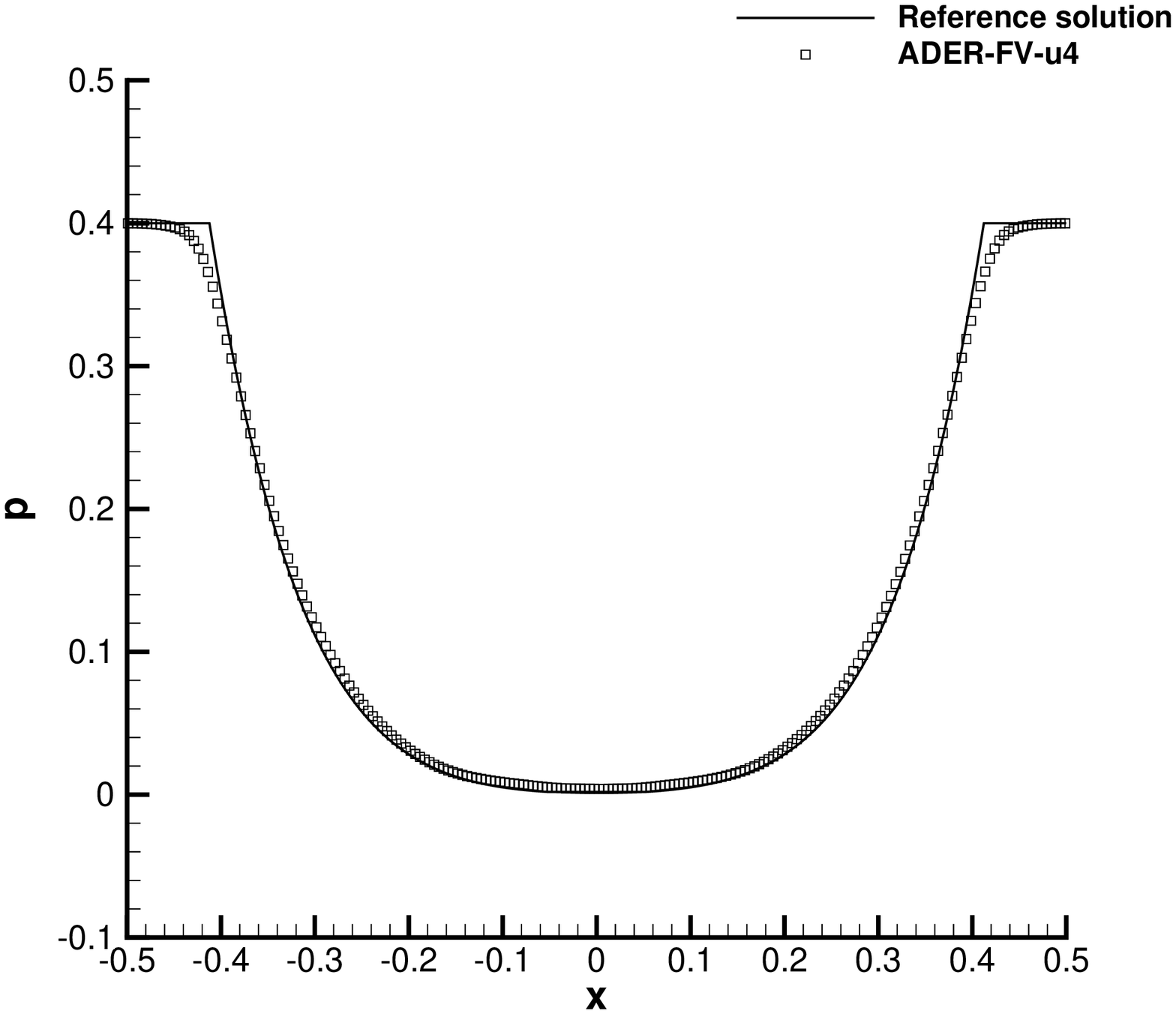}  &
			\includegraphics[width=0.32\textwidth]{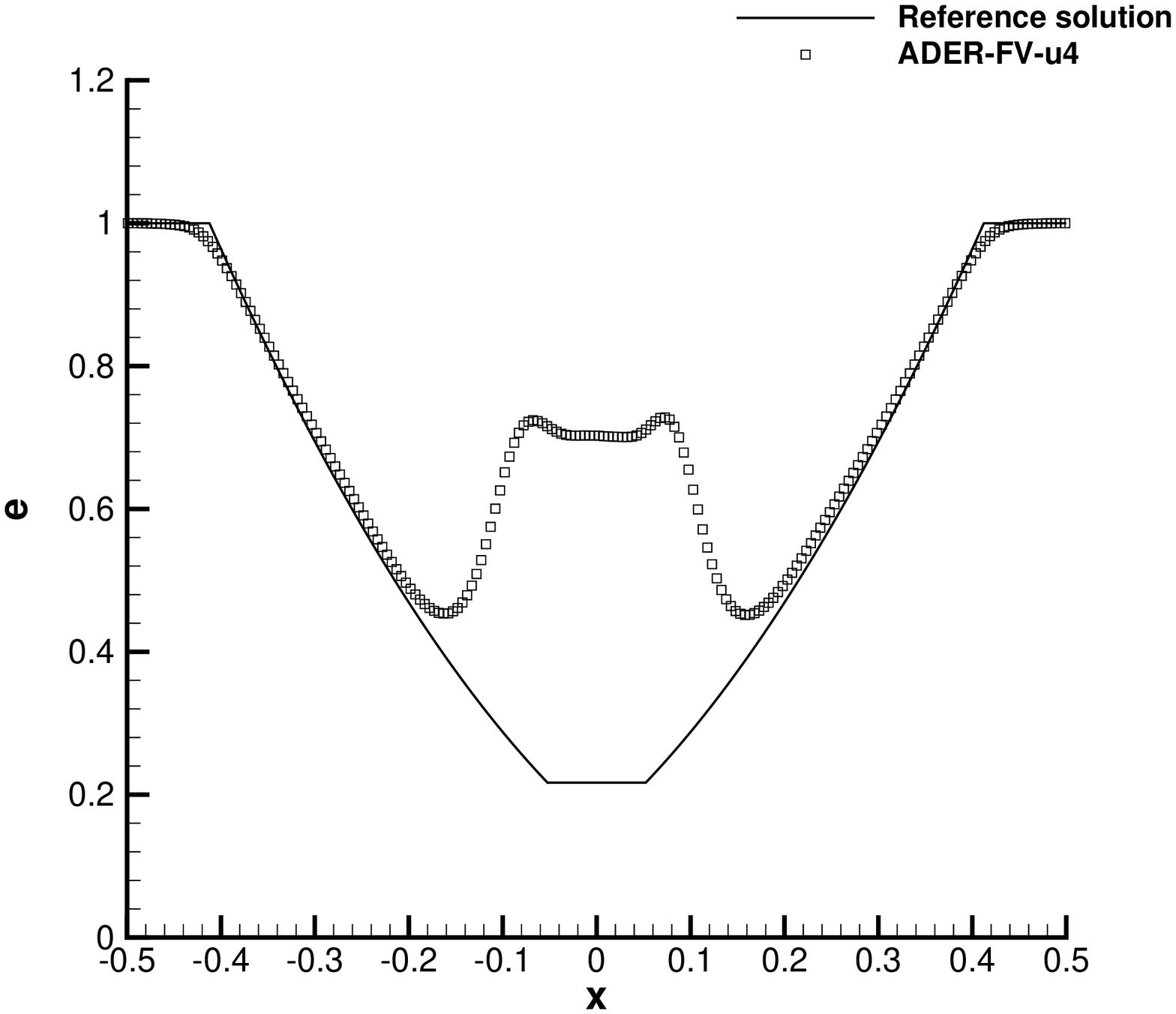} \\   
		\end{tabular} 
		\caption{Double rarefaction test (RP3) at final time $t_f=0.15$. Comparison of density, pressure and internal energy versus the reference solution for ADER-FV-u$4$ scheme.}
		\label{fig.2raref}
	\end{center}
\end{figure}
The next problem (RP3) is the one presented as Test 2 in \cite[Section 4.3.3]{toro2009riemann}. It is a double rarefaction waves which leads to very low pressure and density areas at the center of the domain. In Figure \ref{fig.2raref}, we can appreciate the capability of the scheme of maintaining positive quantities for these variables, thanks to the DOOM limiter which, at the beginning of the simulation, ensures positivity preservation in the predictor. The mismatching of the internal energy distribution is essentially due to the excessive numerical dissipation of the scheme, which could be reduced by introducing entropy preserving techniques \cite{kuzmin2020entropycg,gaburro2022high,chen2017entropy}.

\begin{figure}[!htbp]
	\begin{center}
		\begin{tabular}{ccc} 
			\includegraphics[width=0.32\textwidth]{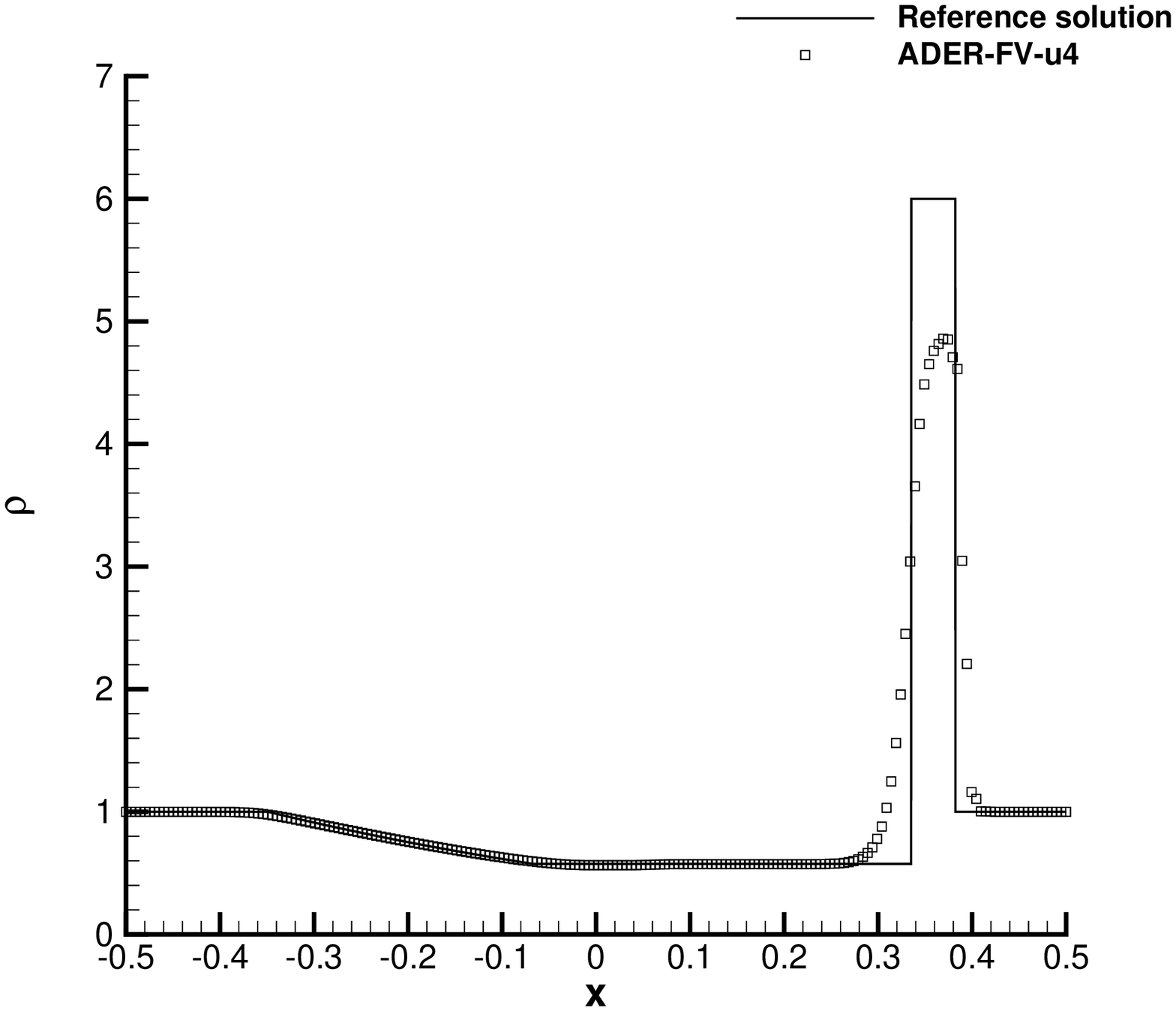}  & 
			\includegraphics[width=0.32\textwidth]{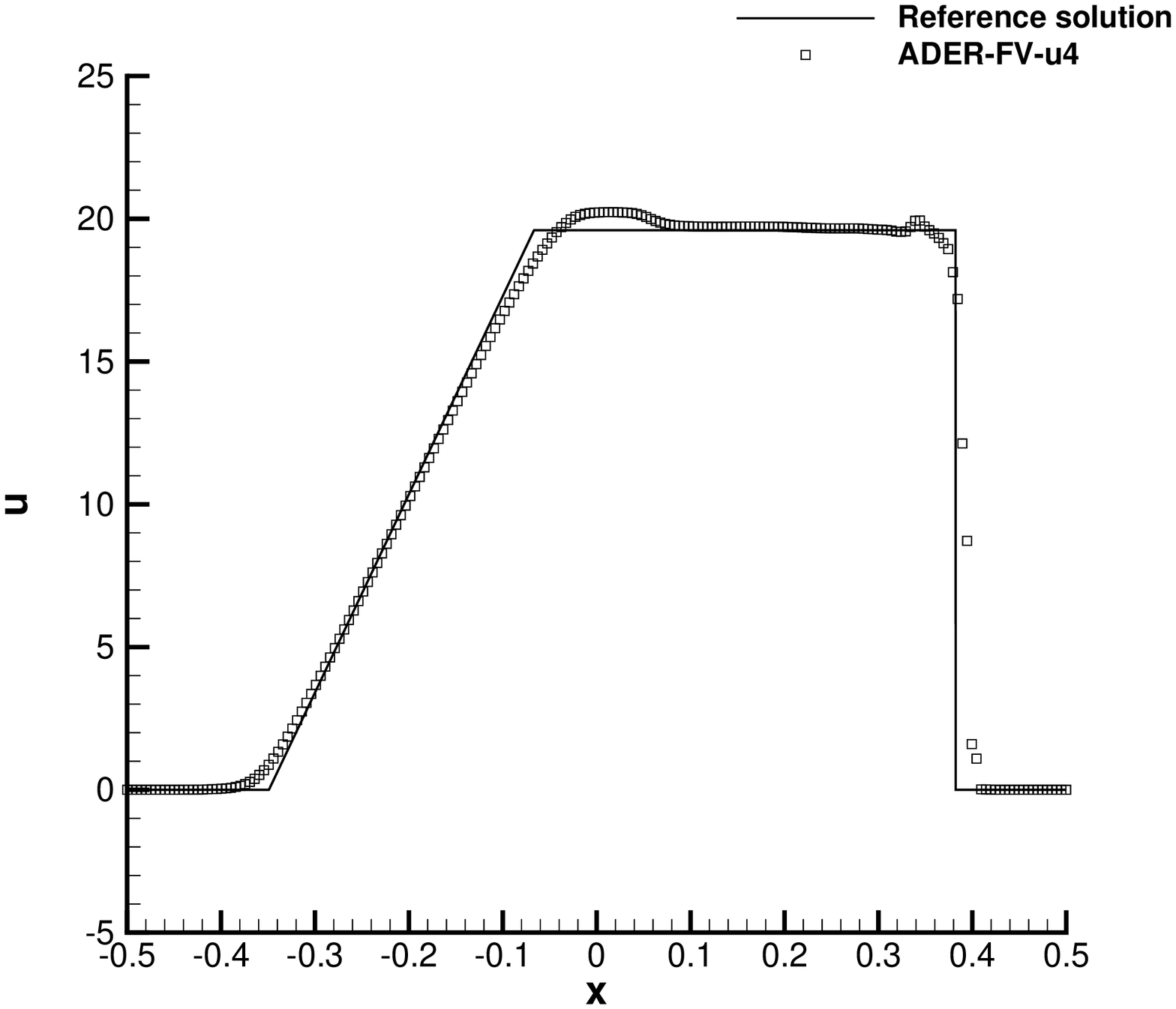}  &
			\includegraphics[width=0.32\textwidth]{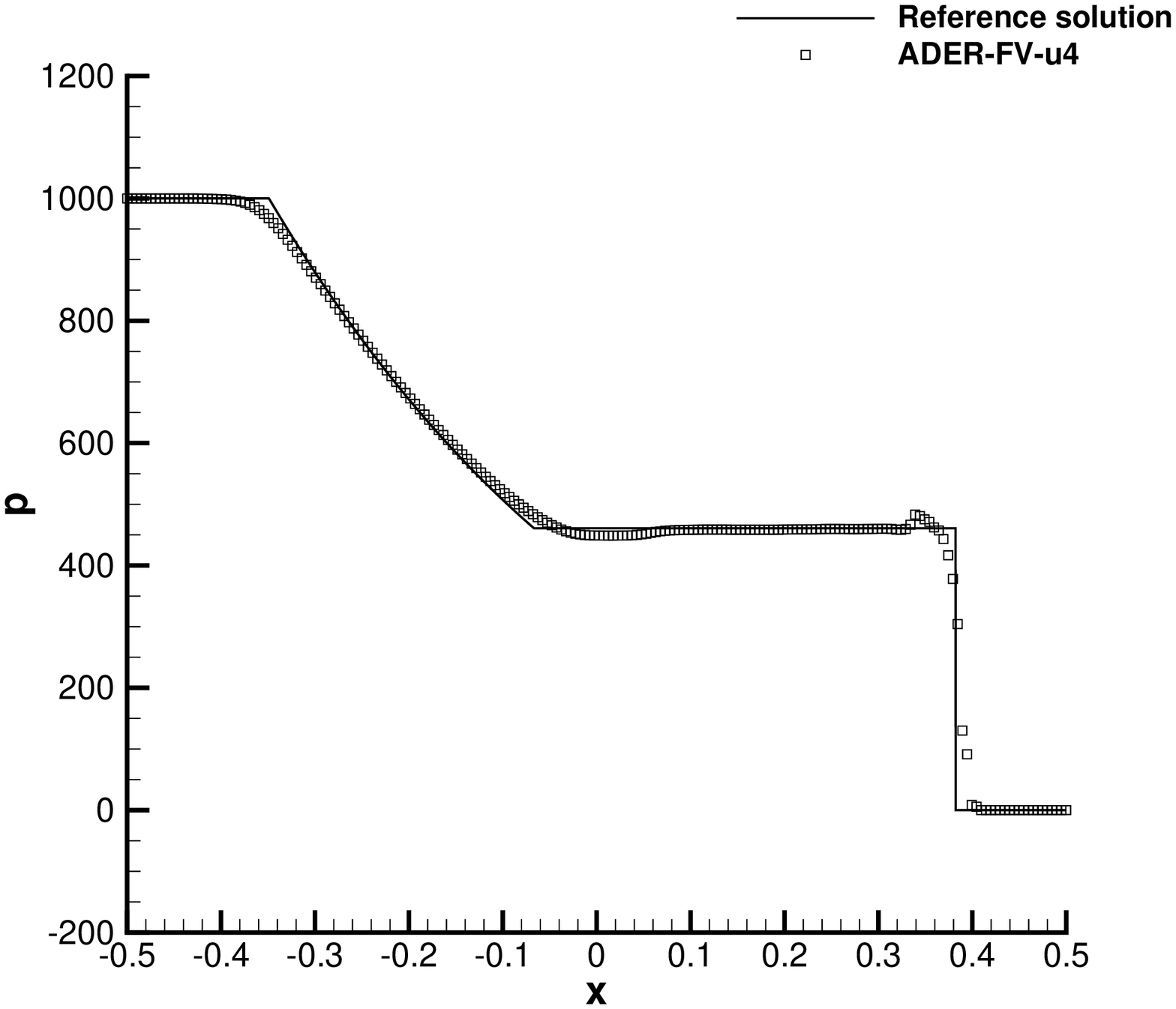} \\   
		\end{tabular} 
		\caption{Test RP4 at final time $t_f=0.012$. Comparison of velocity and pressure versus the reference solution for ADER-FV-u$4$ scheme.}
		\label{fig.ToroRP3}
	\end{center}
\end{figure}
The last Riemann problem (RP4) is the one presented as Test 3 in \cite[Section 4.3.3]{toro2009riemann}. It is a very severe test problem and it consists of a
rarefaction, a contact discontinuity and a shock. The results obtained in Figure \ref{fig.ToroRP3} are in agreement with the reference solution and the smearing around the contact discontinuity is comparable to other high order FV schemes with similar resolution.

% ----------
\subsection{Viscous shock profile} \label{ssec.ShockNS}
% ----------
Now, we consider an isolated viscous shock that is traveling through a medium at rest with a shock Mach number $M_s>1$ \cite{becker1922stosswelle,boscheri2022continuous,dumbser2010arbitrary,busto2021staggered,boscheri2017arbitrary}, thus we solve the compressible Navier-Stokes equations \eqref{eq:NS}. The analytical solution and the details to compute it can be found in \cite{becker1922stosswelle}, where the stationary shock wave at Prandtl number $\textnormal{Pr}=0.75 $ is resolved with constant viscosity. The computational domain is $\Omega = [0,1]\times [0,0.2]$, which is discretized by $N_h=1120$ Voronoi elements. On the left side of the domain a constant inflow velocity is prescribed, while outflow boundary conditions are assumed at the right of the domain. Periodic boundary conditions are, instead, assigned to the top/bottom boundaries.
%The top/bottom boundaries are assigned periodic boundary conditions.
The initial condition consists of a shock wave centered at $x=0.25$ propagating at Mach
$M_s = 2$ from left to right with a Reynolds number $\textnormal{Re} = 100$, thus the viscosity coefficient is set to $\mu  = 2 \cdot 10^{-2}$. The upstream shock state is defined such that the adiabatic sound speed is $c_0=1$. The final time of the simulation is $t_f = 0.2$ with the shock front located at $x = 0.65$.
\begin{figure}[!htbp]
	\begin{center}
		\begin{tabular}{cc} 
			\multicolumn{2}{c}{\includegraphics[width=0.7\textwidth]{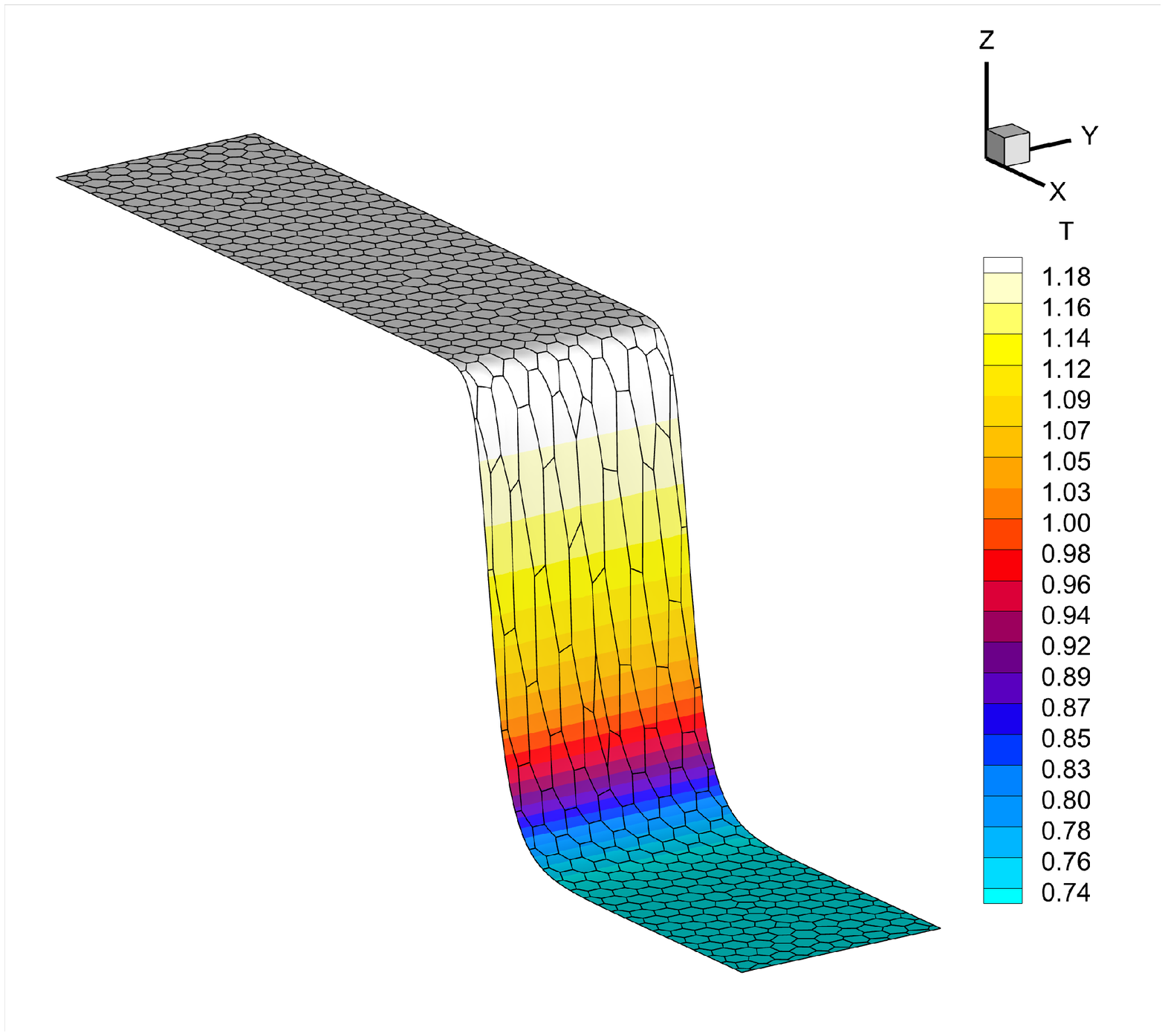}} \\
			\includegraphics[width=0.47\textwidth]{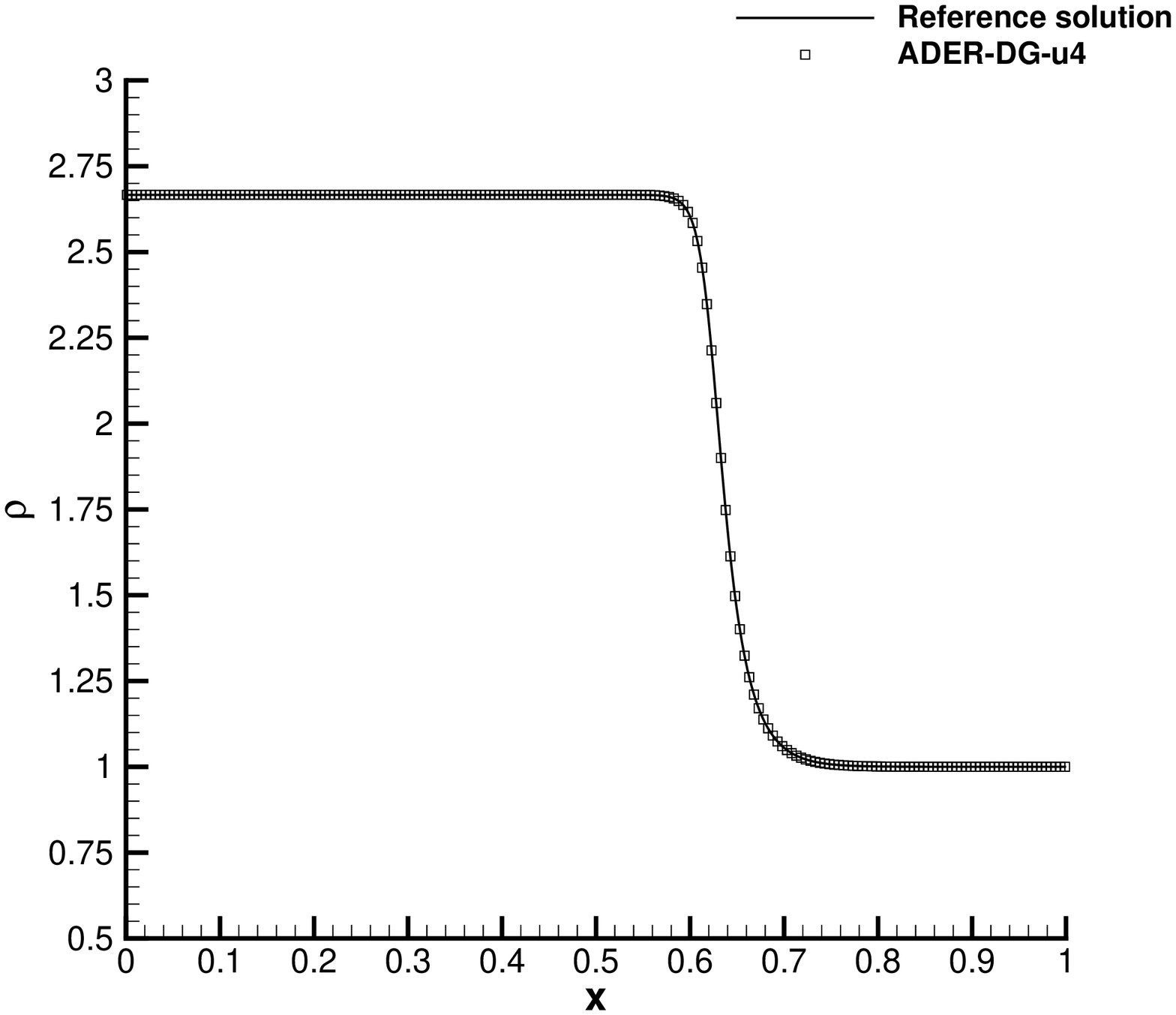} & 
			\includegraphics[width=0.47\textwidth]{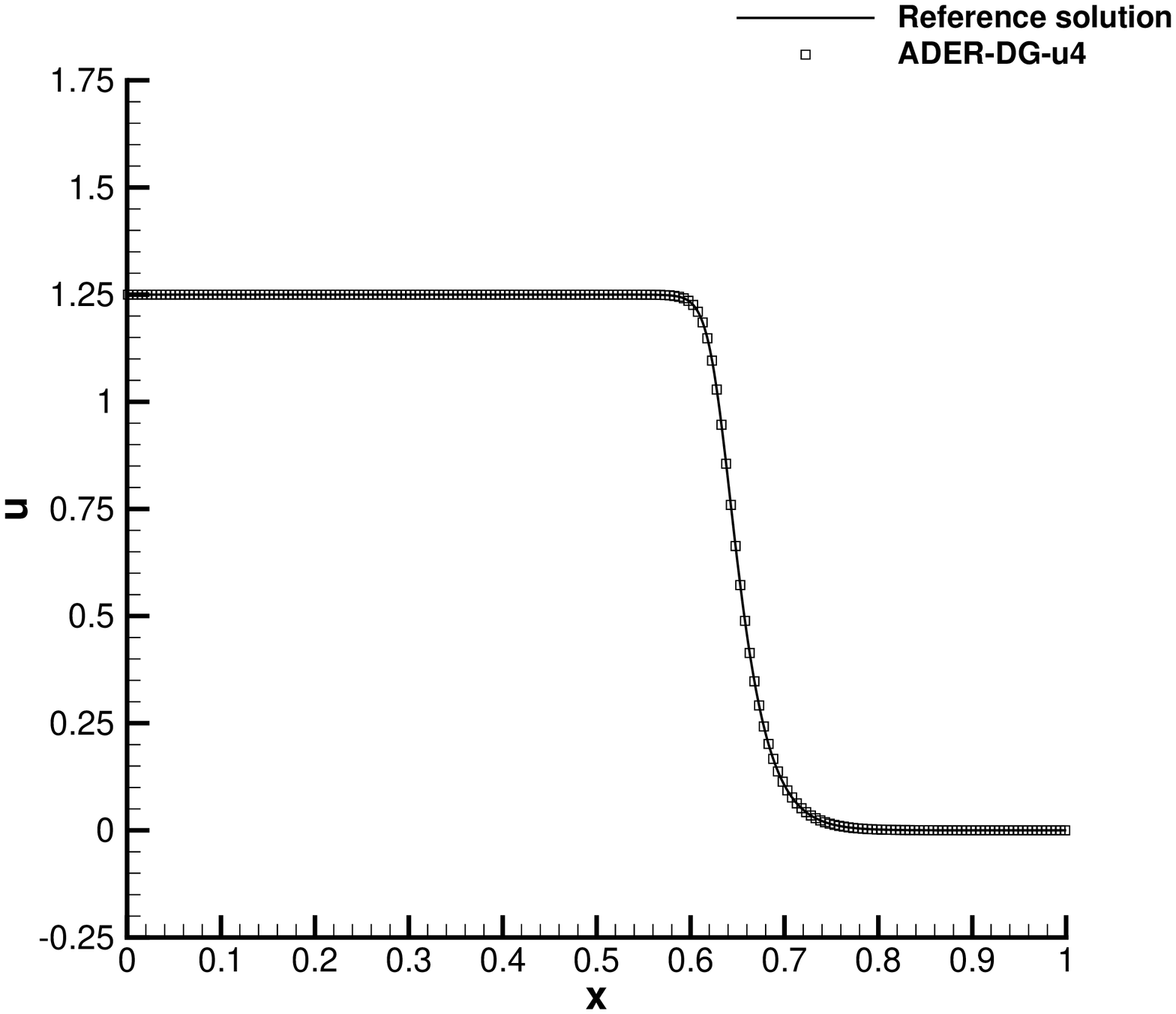} \\
			\includegraphics[width=0.47\textwidth]{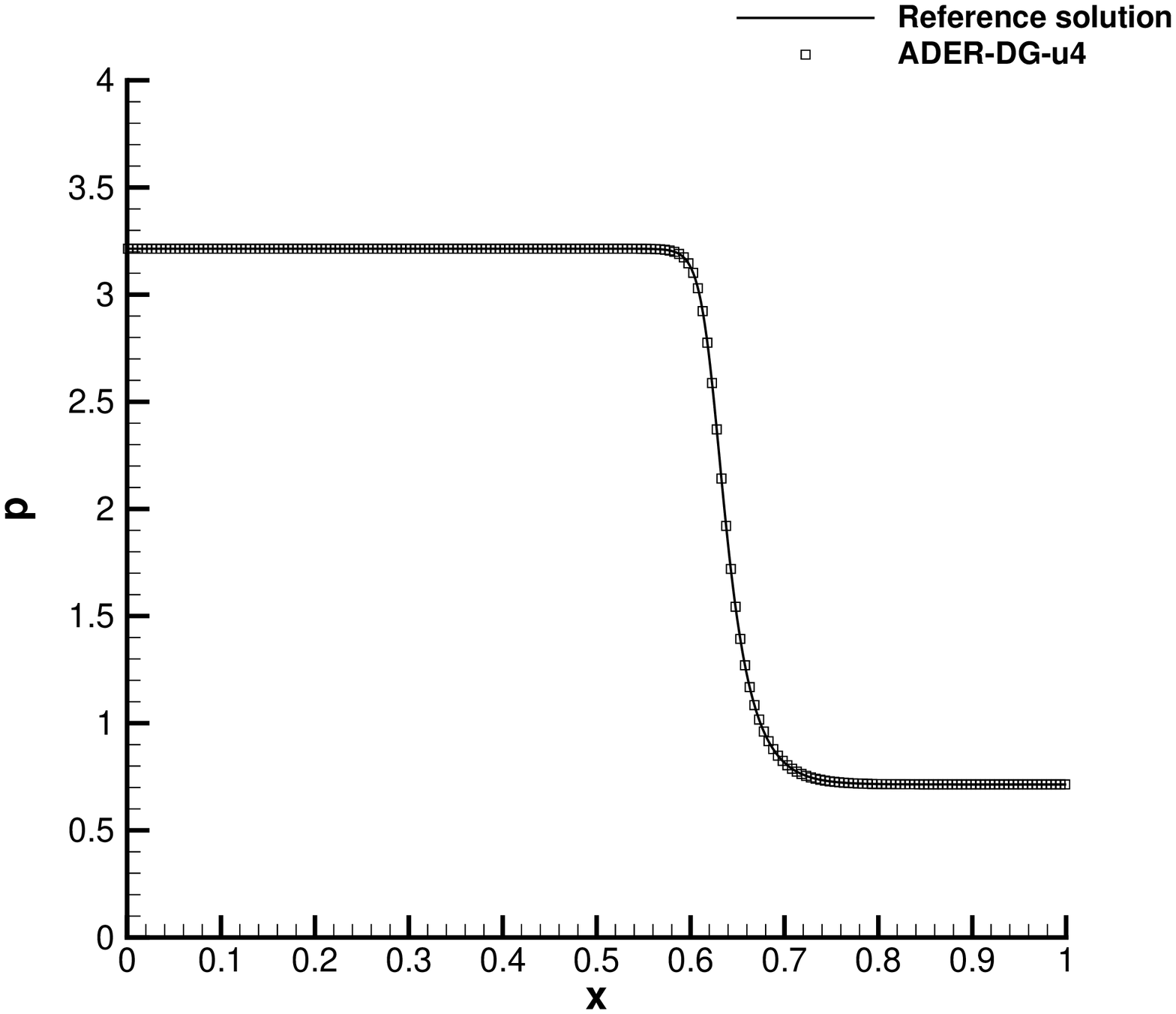} & 
			\includegraphics[width=0.47\textwidth]{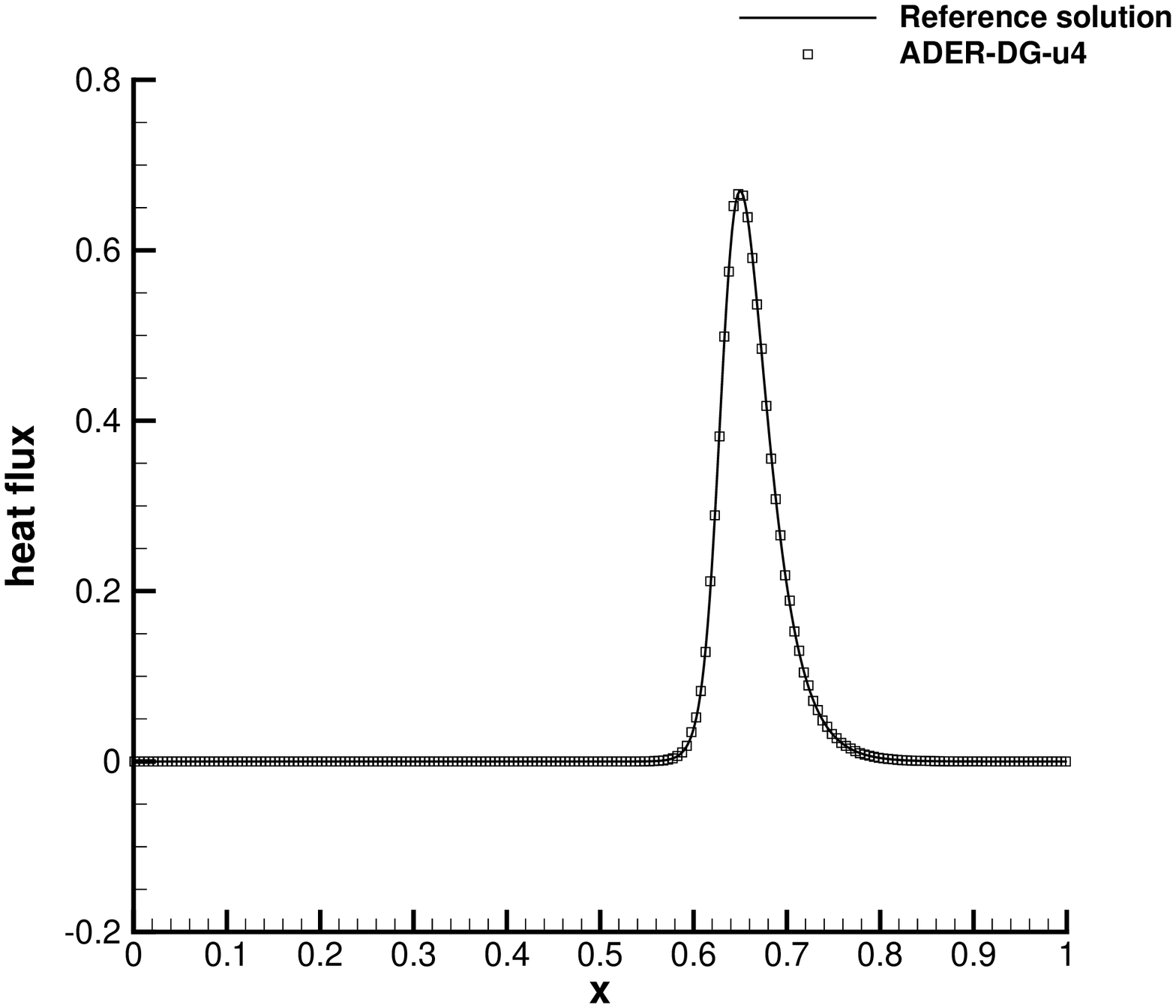} \\
		\end{tabular} 
		\caption{Viscous shock profile with shock Mach number $M_s=2$ and Prandtl number $Pr=0.75$ at time $t_f=0.2$. Top panel: Voronoi tessellation and temperature distribution along the $z-$axis. Fourth order numerical solution with ADER-DG-u scheme compared against the reference solution for density, horizontal velocity, pressure and heat flux (from middle left to bottom right panel): in particular, we show a one-dimensional cut of 200 equidistant points along the $x-$direction at $y=0.1$.}
		\label{fig.ViscousShock}
	\end{center}
\end{figure}
We run the simulations with ADER-DG-u($4$). Since the solution is smooth, nothing is checked along the DOOM procedure. Qualitatively, we see in Figure \ref{fig.ViscousShock} that there is an excellent agreement between the numerical solution and the analytical one.
We underline that this
test case allows all terms contained in the Navier-Stokes system to be properly checked,
since advection, thermal conduction and viscous stresses are present.

% ----------
\subsection{2D Taylor-Green vortex} \label{ssec.TGV}
% ----------
A classical test case for the incompressible Navier--Stokes equations is the Taylor--Green vortex problem. In two dimensions, the exact solution is known and it is given on the domain $\Omega = [0,2\pi]^2$ with periodic boundary conditions by 
\begin{equation}
	\begin{split}
		u(\uvec{x},t) = &\sin(x)\cos(y) e^{-2\nu t},\\
		v(\uvec{x},t) = -&\cos(x)\sin(y) e^{-2\nu t},\\
		p(\uvec{x},t) = C&+\frac14 (\cos(2x)\cos(2y))e^{-4\nu t},
	\end{split}	
\end{equation}
with $\nu = \frac{\mu}{\rho}$ the kinematic viscosity and $\mu=10^{-2}$. In this test, we also validate the quality of the scheme in a low Mach regime. Hence, the additive constant for the pressure is chosen as $C=100/\gamma$ and the density is set at the beginning as $\rho(\uvec{x},0)\equiv 1$. For this test, heat conduction is neglected, i.e., $\kappa=0$.
The mesh is discretized by $N_h=2916$ cells and the final time is set at $t_f=1$. We use the ADER-DG-u method with order 4 for this simulation without checks in the DOOM procedure. The results are depicted in Figure \ref{fig.TGV2D}, which are compared against the analytical solutions, obtaining an excellent matching.
\begin{figure}[!htbp]
	\begin{center}
		\begin{tabular}{cc} 
			\includegraphics[width=0.47\textwidth]{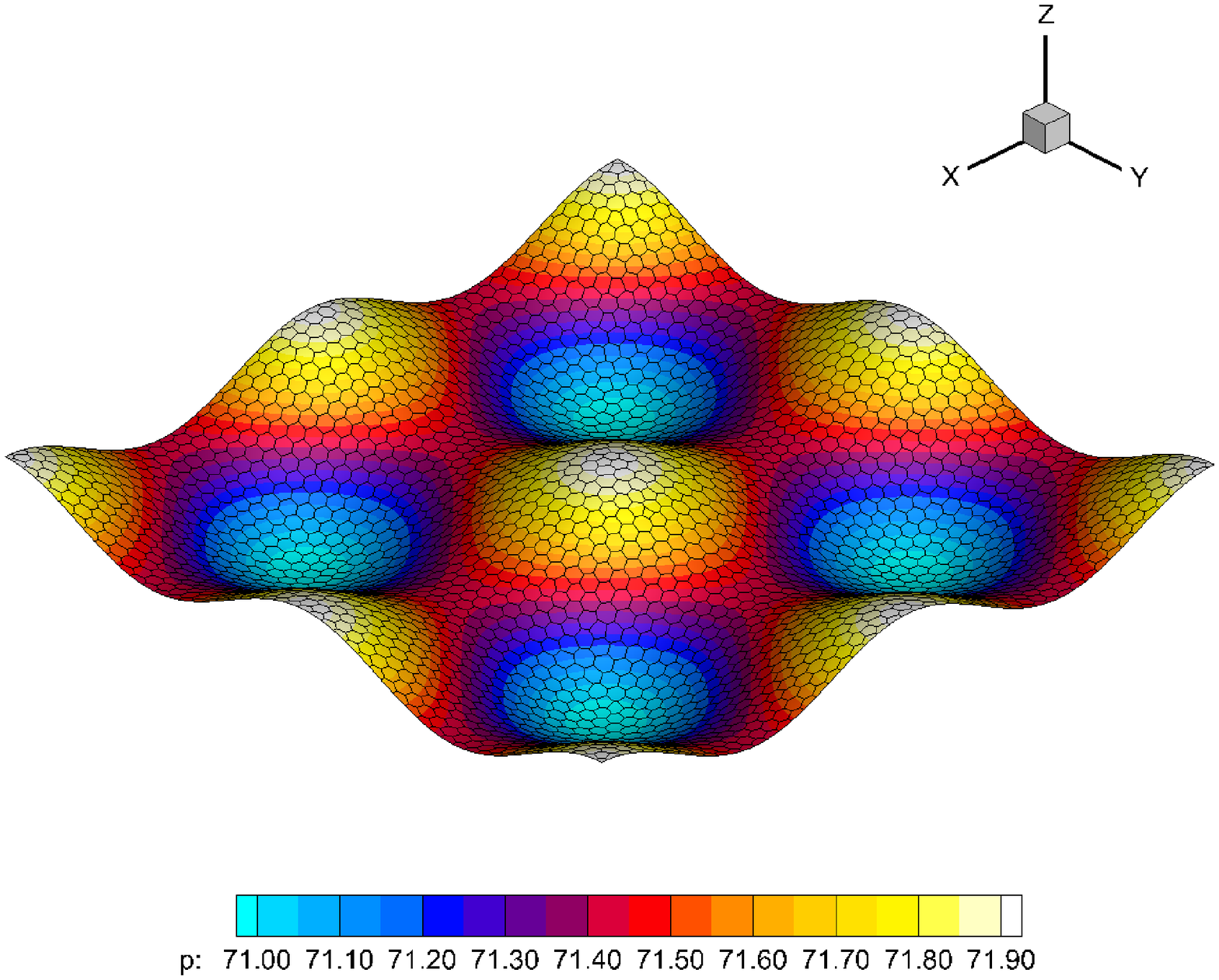} & 
			\includegraphics[width=0.47\textwidth]{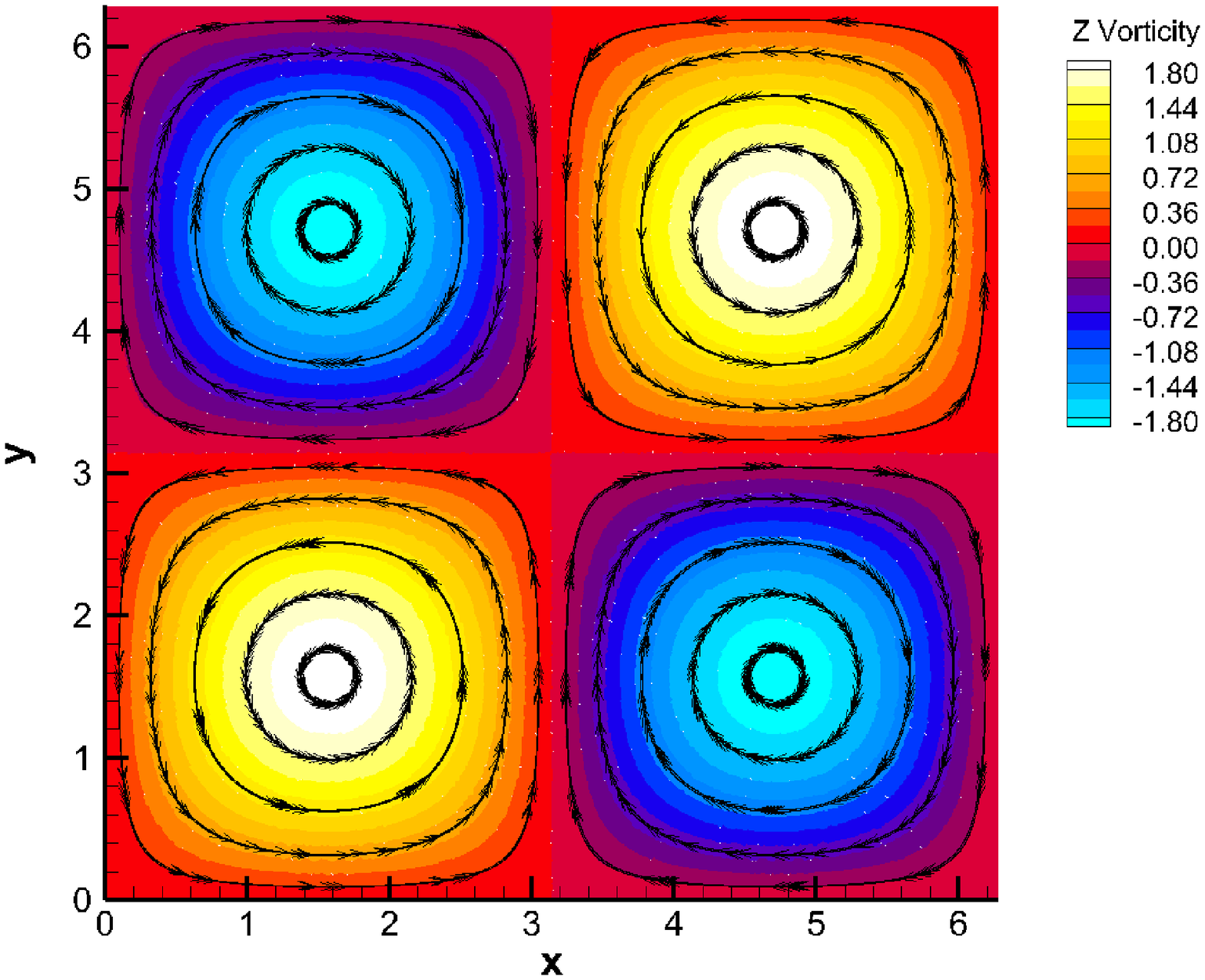} \\
			\includegraphics[width=0.47\textwidth]{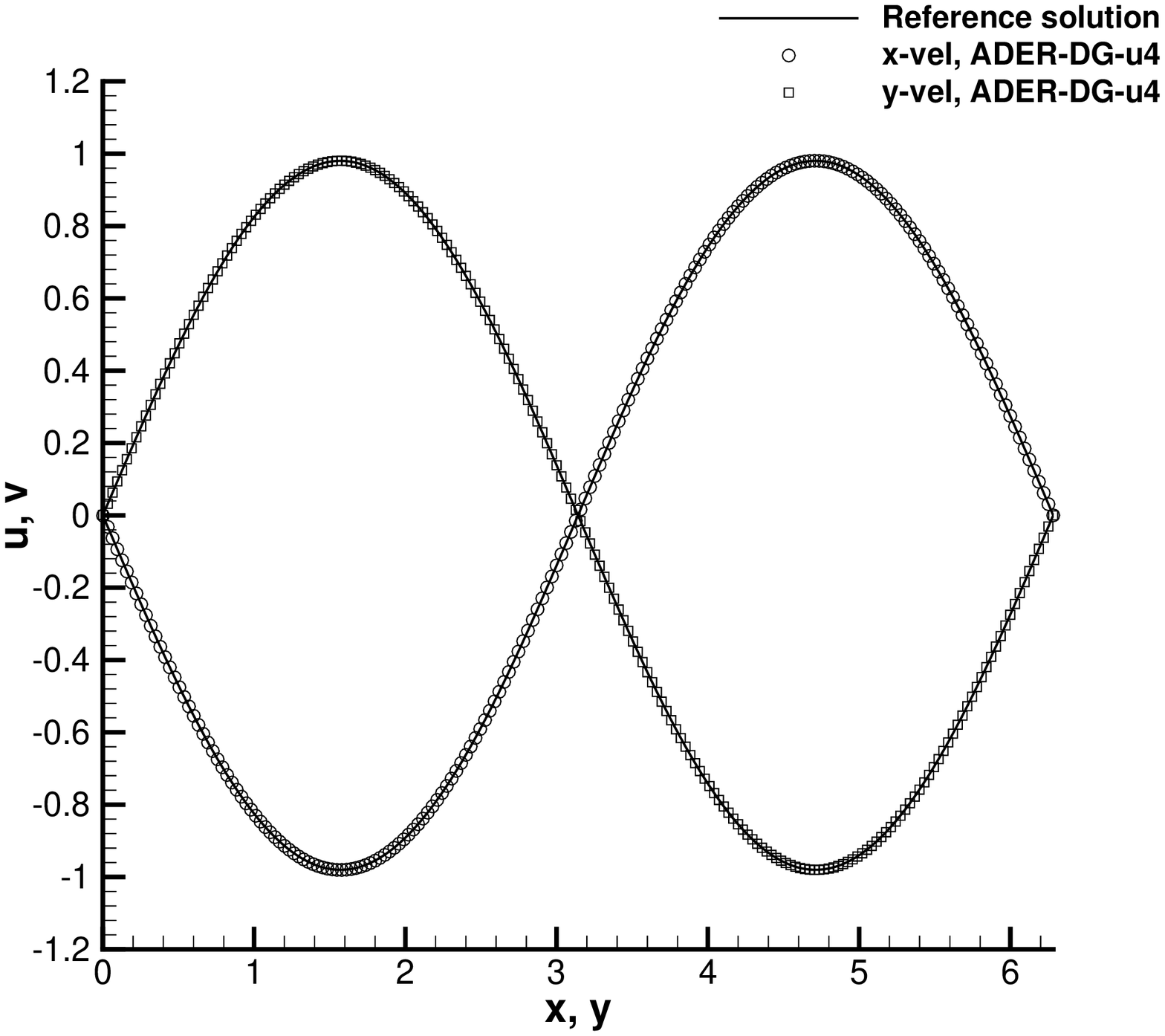} & 
			\includegraphics[width=0.47\textwidth]{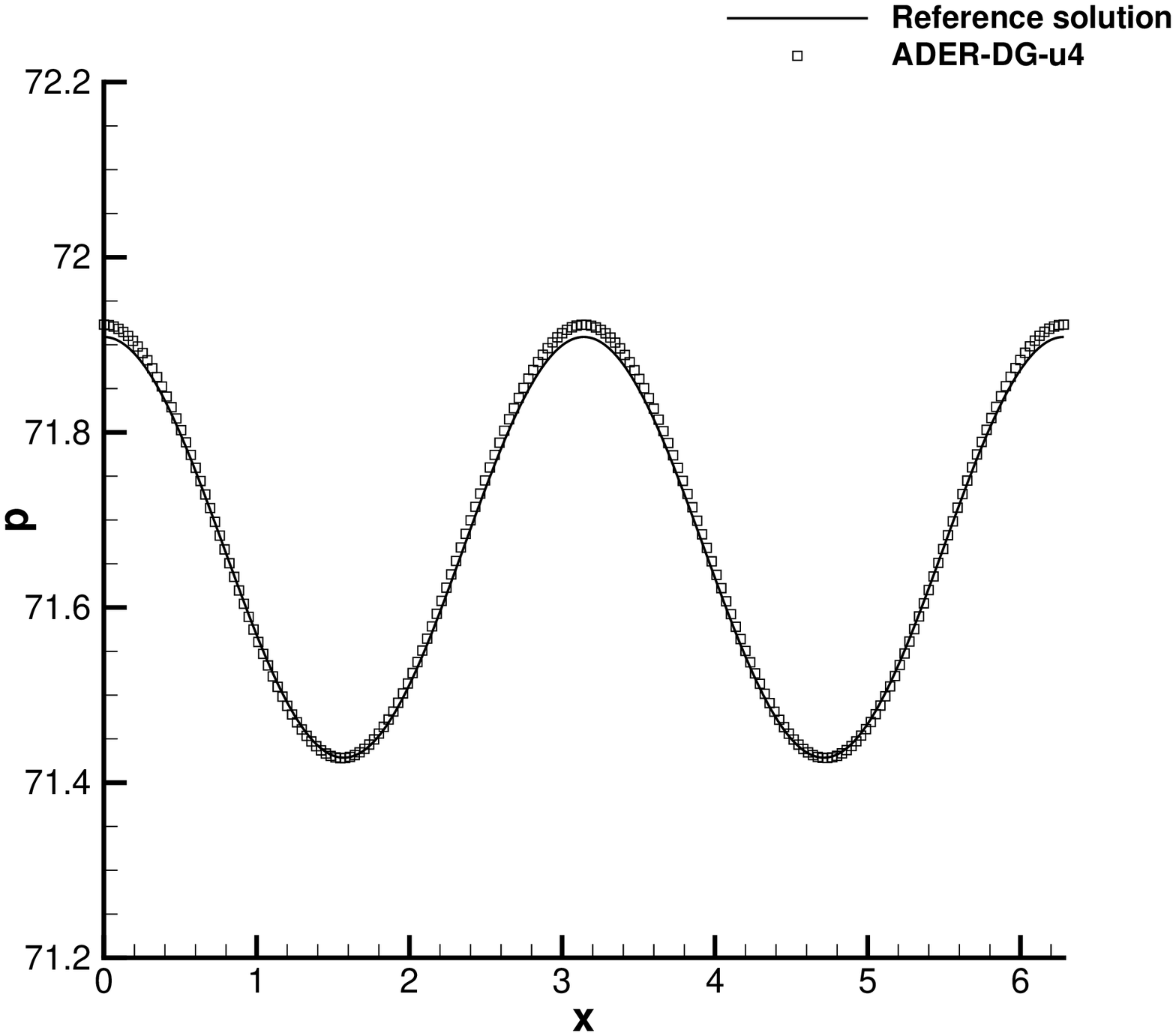} \\
		\end{tabular} 
		\caption{2D Taylor-Green vortex at time $t_f=1$ with viscosity $\mu=10^{-2}$. Exact solution of the Navier--Stokes equations and fourth order numerical solution with ADER-DG-u$4$ scheme. Top: mesh configuration with pressure distribution (left) and $z-$vorticity with stream-traces (right). Bottom: one-dimensional cut of 200 equidistant points along the $x$-axis and the $y-$axis for the velocity components $u$ and $v$ (left) and for the pressure $p$ (right).}
		\label{fig.TGV2D}
	\end{center}
\end{figure}

We also compare the numerical results with the ADER-DG scheme, and the errors are reported in Table \ref{tab.tgv} as well as the computational time. We observe that the errors are almost the same for both method, while the novel ADER-DG-u scheme is 2.5 times faster than the classical ADER-DG.

%----------------------------------------
% Error analysis for Taylor-Green vortex
%----------------------------------------
\begin{table}[!htbp]  
	\caption{Error analysis for the Taylor-Green vortex using both ADER-DG-u and ADER-DG schemes with fourth order of accuracy in space and time. The errors are measured in $L_2$ and $L_{\infty}$ norms and refer to the variables $\rho$ (density) and horizontal velocity $u$ at the final time $t_{f}=1$. The computational time measured in seconds is also reported.}  
	\begin{center} 
		\begin{small}
			\renewcommand{\arraystretch}{1.1}
			\begin{tabular}{c|cc|cc|c}
				\multirow{2}{*}{Scheme} & \multicolumn{2}{c|}{Density ($\rho$)} & \multicolumn{2}{c|}{Velocity ($u$)} & CPU time \\
				& $L_2$ & $L_{\infty}$ & $L_2$ & $L_{\infty}$ & [s] \\ 
				\hline
				ADER-DG-u4 & 8.950E-03 & 3.305E-03 & 1.604E-03 & 6.112E-04 & 1.154E+04 \\
				ADER-DG4 \phantom{-u} & 8.950E-03 & 3.305E-03 & 1.604E-03 & 6.112E-04 & 2.706E+04 \\
			\end{tabular}
		\end{small}
	\end{center}
	\label{tab.tgv}
\end{table}

% ----------
\subsection{Compressible mixing layer} \label{ssec.MixLayer}
% ----------
Finally, we test the novel ADER-DG-u4 on the unsteady compressible mixing layer studied in~\cite{Colonius}. The two-dimensional computational domain is the rectangular box $\Omega=[-200,200] \times [-50,50]$, and a total number of $N_h=15723$ polygonal Voronoi cells compose the computational mesh. The initial condition of the flow is given by two fluid layers moving with different velocities along the $x-$direction, that is
\begin{eqnarray}
\rho(\uvec{x},0) &=& \rho_0= 1, \nonumber \\
%(u(\uvec{x},0),v(\uvec{x},0))^T&=&\uvec{v}_0 = \left( \begin{array}{c}
\uvec{v}(\uvec{x},0)&=&\uvec{v}_0 = \left( \begin{array}{c}
\frac{1}{8} \tanh(2y) + \frac{3}{8} \\ 0
\end{array} \right), \\
 p(\uvec{x},0) &=& p_0 = \frac{1}{\gamma}. \nonumber
\end{eqnarray} 
The free stream velocities are imposed as boundary conditions in the $y-$direction, thus we set $u_{+\infty}=0.5$ and $u_{-\infty}=0.25$ for $y \to + \infty$ and $y \to -\infty$, respectively. Along the $x-$direction, at the right side is simply assign an outflow boundary, whereas the left side is given a time-dependent inflow boundary condition with a perturbation $\delta(y,t)$:
\begin{eqnarray}
\rho(0,y,t) &=& \rho_0 + 0.05 \, \delta(y,t), \nonumber \\
\uvec{v}(0,y,t) &=& \uvec{v}_0 + \left( \begin{array}{c}
1.0 \\ 0.6
\end{array} \right)\, \delta(y,t), \\
\quad p(0,y,t) &=& p_0 + 0.2 \, \delta(y,t). \nonumber 
\end{eqnarray}
The function $\delta(y,t)$ is given by
{\small
\begin{eqnarray}
\delta(y,t) &=& -10^{-3} \exp(-0.25 y^2) \cdot \nonumber \\
            &\phantom{=}&\left[ \cos(\omega t) + \cos\left(\frac{1}{2}\omega t -0.028\right) + \cos\left(\frac{1}{4}\omega t +0.141\right) + \cos\left(\frac{1}{8}\omega t +0.391\right) \right], \nonumber \\
            &\phantom{=}&
\end{eqnarray}}
with the fundamental frequency of the mixing layer $\omega = 0.3147876$. The compressible Navier-Stokes equations are considered with viscosity coefficient $\mu=10^{-3}$ and no heat conduction ($\kappa=0$). The final time is $t_f=1596.8$ and the DG solution is depicted in Figure \ref{fig.MixLayer} at three different output times. The vorticity of the flow field is shown, demonstrating the capability of the novel methods to capture the complex vortical structures generated by the perturbation assigned at the inflow of the channel.

\begin{figure}[!htbp]
	\begin{center}
		\begin{tabular}{cc} 
			\includegraphics[trim= 5 5 5 5, clip, width=0.9\textwidth]{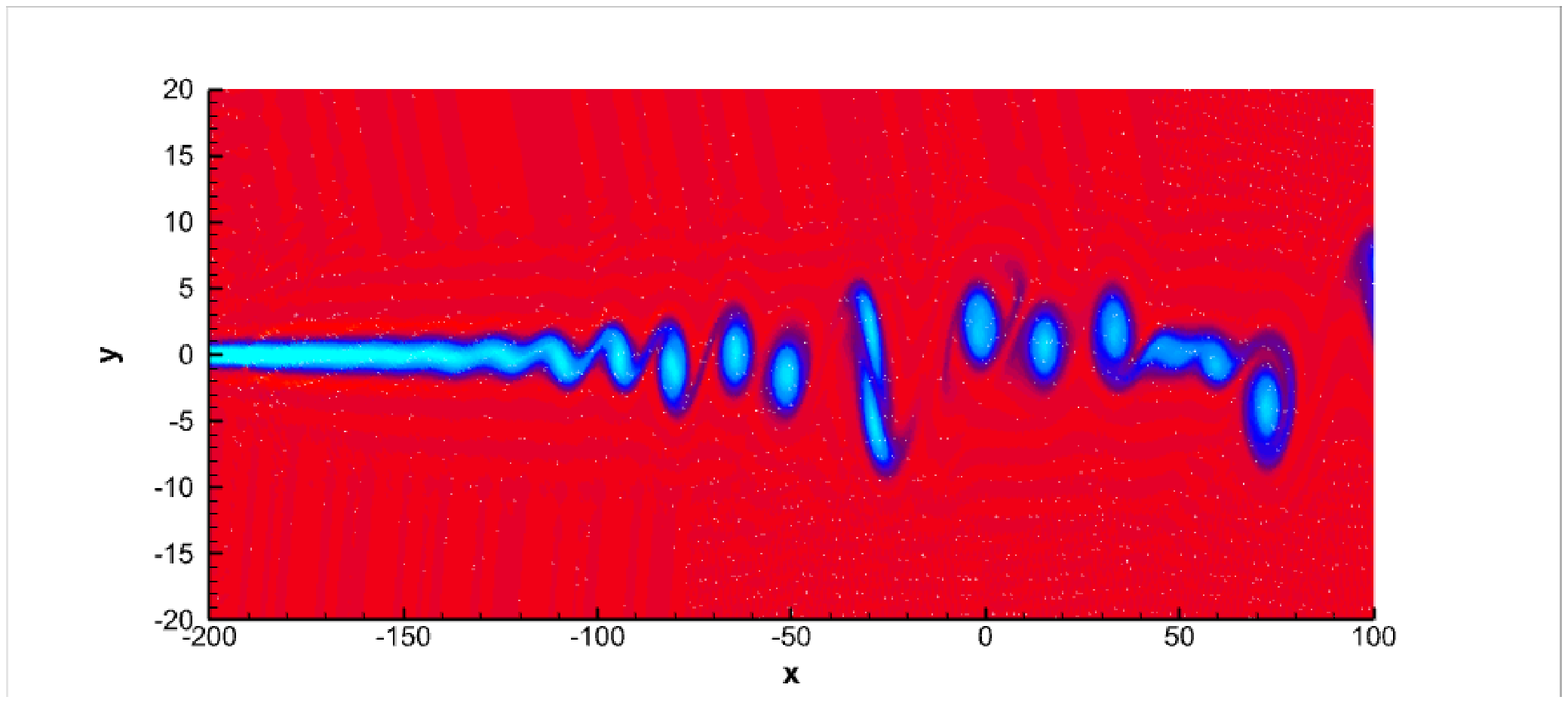}\\				
			\includegraphics[trim= 5 5 5 5, clip, width=0.9\textwidth]{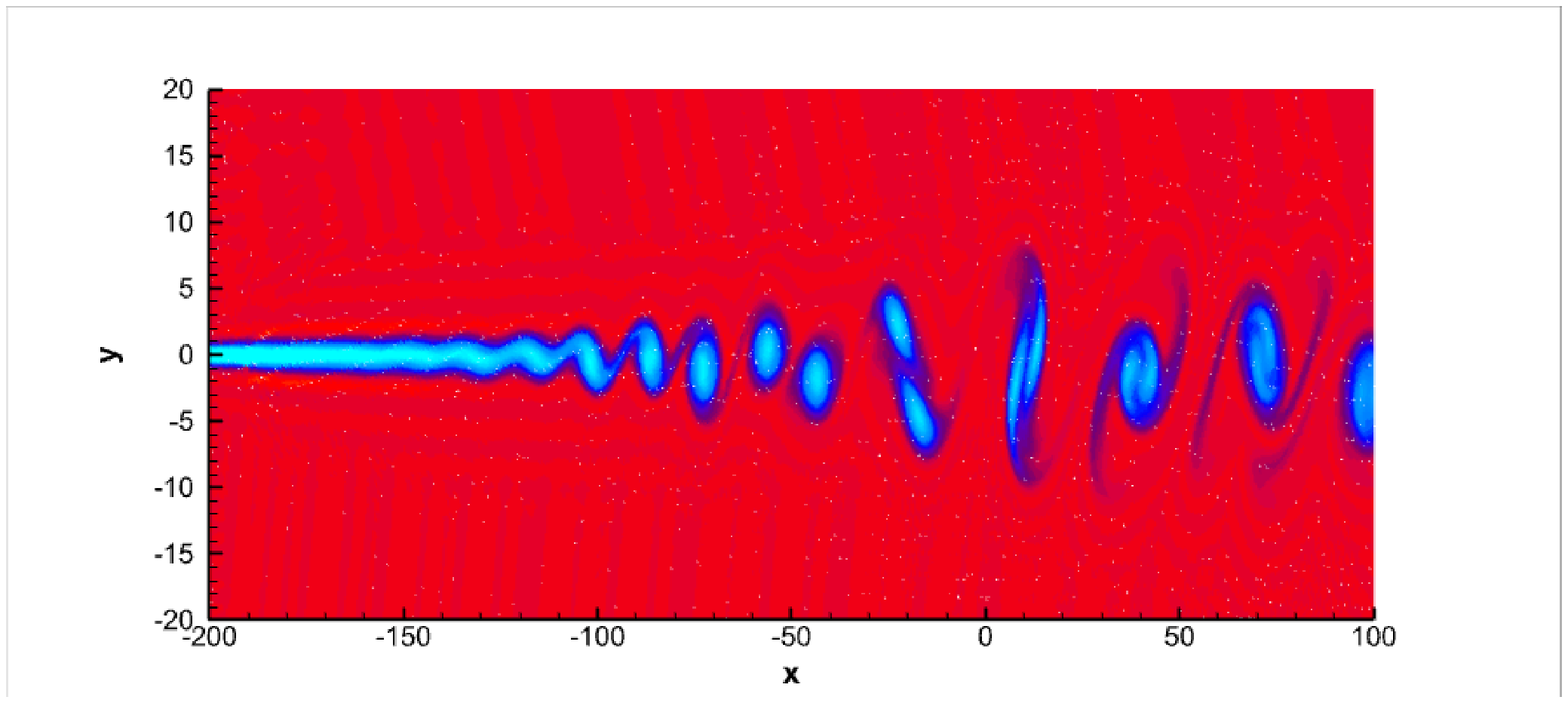}\\
			\includegraphics[trim= 5 5 5 5, clip, width=0.9\textwidth]{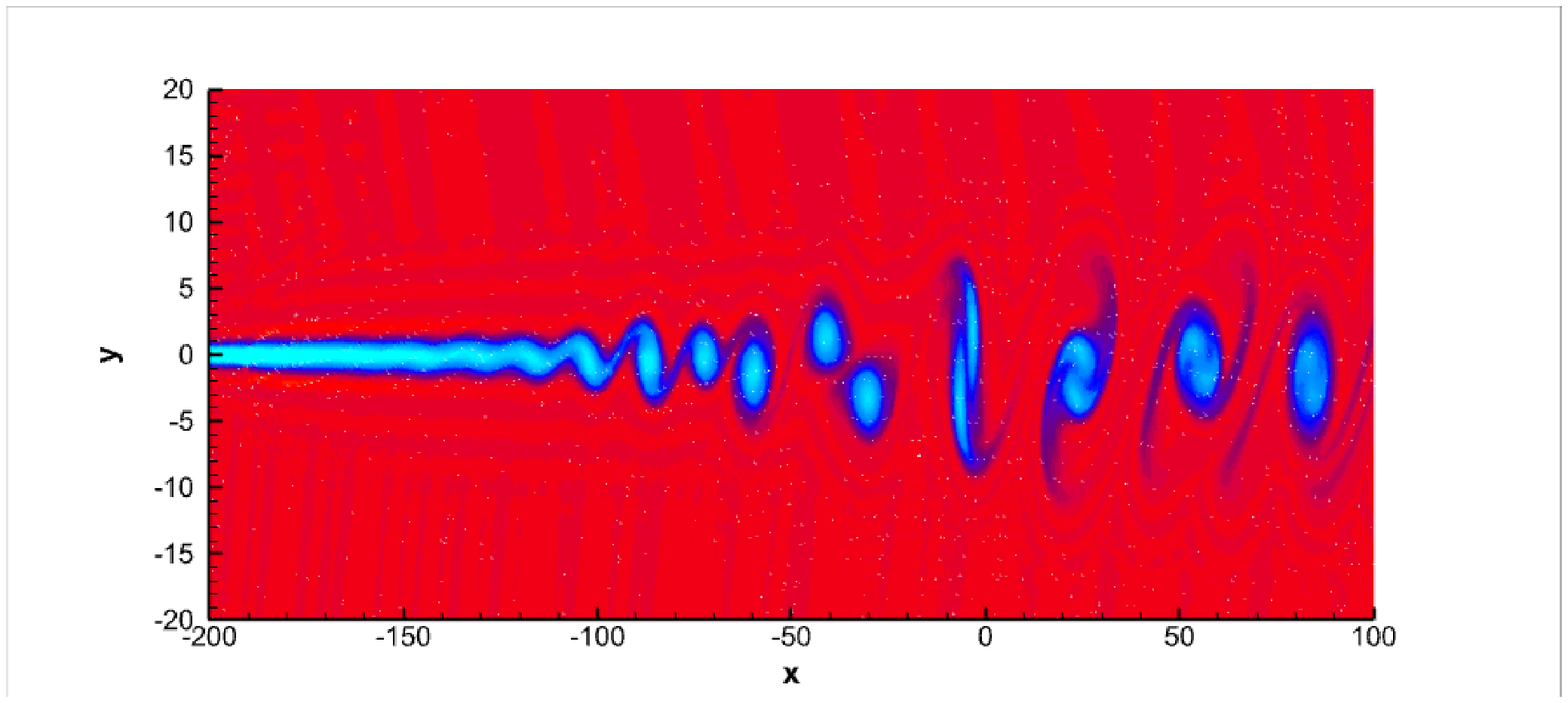}\\				
		\end{tabular} 
		\caption{Compressible mixing layer at time $t=500$, $t=1000$ and $t=1596.8$ (from top to bottom row). Fourth order numerical results with ADER-DG-u for $z-$vorticity. 51 contour levels in the range $[-0.12,0.12]$ have been used for plotting the vorticity distribution on the sub-domain $[-200,100]\times[-20,20]$.}
		\label{fig.MixLayer}
	\end{center}
\end{figure}

\section{Conclusions and further developments}\label{chap:Conc}
%Summarizing, we have presented an improved version of arbitrary high order iterative methods, in which at each iteration the discretization space accuracy matches the one of that specific iteration. 
To sum up, generalizing the idea proposed in \cite{loredavide}, we have introduced a new framework for the construction of efficient p-adaptive arbitrary high order methods, based on the modification of underlying arbitrary high order iterative schemes.
Specifically, the accuracy of the discretization is progressively increased with the number of iterations, gaining one order of accuracy at each iteration. 
%The modification consists in redesigning the generic iteration in such a way that its discretization accuracy matches the order of accuracy achieved in the same iteration.
Given an implementation of an iterative arbitrary high order method, the novel technique is easy to include and it gives a remarkable advantage in terms of computational costs.
Moreover, in this context, p-adaptivity can be achieved very naturally inserting some criteria to stop the iterations.
We showed an application to ADER-DG, designing the new efficient \APNPM\!\!-u methods.
In particular, in the ADER-FV-u context ($N=0$), we have proposed DOOM, an \textit{a posteriori} limiter, that is able to preserve the physical properties of the solution (i.e. positivity of density and pressure) obtaining the maximum admissible order of accuracy that guarantees these physical constraints to be respected.
In this framework, there is a huge advantage with respect to similar \textit{a posteriori} limiters, e.g. MOOD \cite{CDL1}, as DOOM is waste--free, i.e., all the computations are useful either for increasing the order of accuracy or for detecting a troubled state. %and they are used in case a troubled state is detected.
In the numerical tests, we have solved Euler and compressible Navier-Stokes equations very robustly, provably keeping the positivity of density and pressure, and with computational costs up to 4 times smaller than the original method.

We believe that the proposed framework is very versatile and can improve many arbitrary high order methods on different sides: reducing the computational costs, easily providing p-adaptivity in a very efficient and natural way without wasting computed solutions, and helping obtaining structure preserving solutions.
\RIIcolor{
The authors are currently working on the application of the novel framework to obtain: adaptive methods that converge to the analytical solution up to a given tolerance and hp-adaptive methods introducing local mesh refinements in non-smooth regions. 
We also aim at further investigations on the implicit version of the schemes obtained by choosing a low order implicit operator $\lopdt^1$ in the DeC formulation and on structure-preserving.
}

% summary: improved arbitrary high order iterative methods as less iterations and p-adaptivity with some criteria (DOOM)
% huge advantage with respect to the classical limiters as it is almost waste-free, all computations are useful, local because it works in the predictor, faster because smaller size of the operators
% tests: computational costs, positivity
% outlook: extend the structure preserving to ADER-PNPM with N>0
% other structure preserving
% adaptivity to reach absolute errors
%hp adaptivity

\subsubsection*{Acknowledgments}
LM has been funded by the SNF grant 200020\_204917 ``Structure preserving and fast methods for hyperbolic systems of conservation laws''. DT has been funded by a SISSA Mathematical Fellowship. WB received financial support by Fondazione Cariplo and Fondazione CDP
(Italy) under the grant No. 2022-1895.

\subsubsection*{Conflict of interest}
On behalf of all authors, the corresponding author states that there is no conflict of interest. 

\subsubsection*{Compliance with Ethical Standards}
On behalf of all authors, the corresponding author is available to collect documentation of compliance with ethical standards and send upon request.

\subsubsection*{Funding}
This work received financial support by the Swiss National Foundation (Switzerland), Fondazione Cariplo and CDP (Italy), and Scuola Internazionale Superiore di Studi Avanzati (Italy).

\appendix
\section{DG modal Taylor basis functions}\label{app:basis_function}
The basis functions $\left\lbrace \trialspacetime{\ell}(\uvec{x},t) \right\rbrace_{\ell=1,\dots,L} $ used to span the predictor polynomial spaces in this work are modal Taylor basis functions. As already said, they are the tensor products of space basis functions $\left\lbrace \trialspace{i}(\uvec{x}) \right\rbrace_{i=1,\dots,I} $ and time basis functions $\left\lbrace \trialtime{m}(t) \right\rbrace_{m=0,\dots,M}.$

The spatial basis functions $\left\lbrace \trialspace{i}(\uvec{x}) \right\rbrace_{i=1,\dots,I} $ of degree at most $M$ in $D$ dimensions are defined locally for each element $K$. Denoting by $\uvec{x}_K$ the barycenter of the element, they can be defined as
\begin{equation}
\trialspace{\alpha} (\uvec{x}):= \prod_{d=1}^{D}  \frac{(x_d-x_{K,d})^{\alpha_d}}{\alpha_d! h_K^{\alpha_d}},\quad 0\leq \abs{\alpha} \leq M,
\end{equation}
where $h_K=\sqrt[D]{\abs{K}}$ is the characteristic mesh size of the element $K$, used to rescale the functions to agree with the Taylor expansion terms, while $\alpha$ is a $D$-dimensional multi-index with $\abs{\alpha}=\sum_{d=1}^D\alpha_d$. In practice, we identify the multi-index $\alpha$ as a single index $i=1,\dots,I$ with $I={{M+D}\choose{D}}$ via a bijection giving $i = i(\alpha)$. %For full reproducibility of the solution, we mention that for all basis functions with $\abs{\alpha}>0$ we subtract to the basis function the constant of the average on the cell, so that $\int_K \varphi_\alpha (\uvec{x})d\uvec{x} =0.$

The time basis functions $\left\lbrace \trialtime{m}(t) \right\rbrace_{m=0,\dots,M}$ of degree at most $M$ are defined in a similar fashion, but with respect to a scalar argument only, over $[t_{n},t_{n+1}]$
\begin{equation}
\trialtime{m} (t):=  \frac{(t-t_{n})^{m}}{m! \Delta t^{m}},\quad 0\leq m \leq M.
\end{equation}

Finally, the tensor product between the two functional spaces gives the space-time basis functions $\lbrace\trialspacetime{\ell}(\uvec{x},t)\rbrace_{\ell=1,\dots,L}$. For full reproducibility, we specify that accuracy $M+1$ has been achieved selecting space-time basis functions  up to degree $M$ only.
Also for the spaces $\lbrace\trialspacetimeiter{\ell}{(p)}(\uvec{x},t)\rbrace_{\ell=1,\dots,L^{(p)}}$ we have used basis functions up to degree $p$ only.

\begin{remark}[On the ordering of the space-time basis functions]
The novel approach is based on the adoption of iteration-specific bases $\lbrace\trialspacetimeiter{\ell}{(p)}(\uvec{x},t)\rbrace_{\ell=1,\dots,L^{(p)}}$ of increasing degree. 
For modal bases, the introduction of higher order modes is simply performed by considering higher order terms in the space-time polynomial expansion. Therefore, in the context of an efficient implementation of the new methods, it is particularly useful to directly define all the basis functions $\lbrace \trialspacetimeiter{\ell}{(M)}(\uvec{x},t)\rbrace_{\ell=1,\dots,L^{(M)}}$, up to an accuracy order $M+1$, ordering them in increasing polynomial order. By doing so, it is enough to change the final index from $L^{(p-1)}$ to $L^{(p)}$ to pass from $X^{(p-1)}$ to $X^{(p)}$ in all the iterations but the last one, which is performed without changing polynomial space to saturate the related accuracy.
\end{remark}

\bibliography{sn-bibliography}

\begin{thebibliography}{10}

\bibitem{tsoutsanis2016addressing}
Addressing the challenges of implementation of high-order finite-volume schemes
  for atmospheric dynamics on unstructured meshes.
\newblock In M.~Papadrakakis, V.~Papadopoulos, G.~Stefanou, and V.~Plevris,
  editors, {\em ECCOMAS Congress 2016 - Proceedings of the 7th European
  Congress on Computational Methods in Applied Sciences and Engineering},
  volume~1, pages 684--708, GRC, June 2016. National Technical University of
  Athens.

\bibitem{abgrall2006residual}
R{\'e}mi Abgrall.
\newblock Residual distribution schemes: current status and future trends.
\newblock {\em Computers \& Fluids}, 35(7):641--669, 2006.

\bibitem{abgrall2019high}
R{\'e}mi Abgrall, Paola Bacigaluppi, and Svetlana Tokareva.
\newblock High-order residual distribution scheme for the time-dependent
  {Euler} equations of fluid dynamics.
\newblock {\em Computers \& Mathematics with Applications}, 78(2):274--297,
  2019.

\bibitem{abgrall2021relaxation}
R\'emi Abgrall, \'Elise Le~M\'el\'edo, Philipp \"Offner, and Davide Torlo.
\newblock Relaxation {Deferred} {Correction} {Methods} and their {Applications}
  to {Residual} {Distribution} {Schemes}.
\newblock {\em The SMAI Journal of computational mathematics}, 8:125--160,
  2022.

\bibitem{abgrall2022convergence}
R{\'e}mi Abgrall, M{\'a}ria Luk{\'a}cova-Medvid'ov{\'a}, and Philipp
  {\"O}ffner.
\newblock On the convergence of residual distribution schemes for the
  compressible euler equations via dissipative weak solutions.
\newblock {\em arXiv preprint arXiv:2207.11969}, 2022.

\bibitem{abgrall2019analysis}
R{\'e}mi Abgrall, Jan Nordstr{\"o}m, Philipp {\"O}ffner, and Svetlana Tokareva.
\newblock Analysis of the {SBP-SAT} stabilization for finite element methods
  part {II}: entropy stability.
\newblock {\em Communications on Applied Mathematics and Computation}, pages
  1--23, 2021.

\bibitem{abgrall2022reinterpretation}
R{\'e}mi Abgrall, Philipp {\"O}ffner, and Hendrik Ranocha.
\newblock Reinterpretation and extension of entropy correction terms for
  residual distribution and discontinuous galerkin schemes: Application to
  structure preserving discretization.
\newblock {\em Journal of Computational Physics}, 453:110955, 2022.

\bibitem{abgrall2020high}
R{\'e}mi Abgrall and Davide Torlo.
\newblock High order asymptotic preserving deferred correction
  implicit-explicit schemes for kinetic models.
\newblock {\em SIAM Journal on Scientific Computing}, 42(3):B816--B845, 2020.

\bibitem{Decremi}
Rémi Abgrall.
\newblock High order schemes for hyperbolic problems using globally continuous
  approximation and avoiding mass matrices.
\newblock {\em Journal of Scientific Computing}, 73(2-3):461--494, 2017.

\bibitem{bacigaluppi2019posteriori}
Paola Bacigaluppi, R{\'e}mi Abgrall, and Svetlana Tokareva.
\newblock ``{A} posteriori'' limited high order and robust residual
  distribution schemes for transient simulations of fluid flows in gas
  dynamics.
\newblock {\em arXiv preprint arXiv:1902.07773}, 2019.

\bibitem{balsara2000monotonicity}
Dinshaw~S Balsara and Chi-Wang Shu.
\newblock Monotonicity preserving weighted essentially non-oscillatory schemes
  with increasingly high order of accuracy.
\newblock {\em Journal of Computational Physics}, 160(2):405--452, 2000.

\bibitem{becker1922stosswelle}
R~Becker.
\newblock Stosswelle und detonation.
\newblock {\em Zeitschrift f{\"u}r Physik}, 8(1):321--362, 1922.

\bibitem{berberich2021high}
Jonas~P Berberich, Praveen Chandrashekar, and Christian Klingenberg.
\newblock High order well-balanced finite volume methods for multi-dimensional
  systems of hyperbolic balance laws.
\newblock {\em Computers \& Fluids}, 219:104858, 2021.

\bibitem{Lagrange2D}
W.~Boscheri and M.~Dumbser.
\newblock {Arbitrary--Lagrangian--Eulerian One--Step WENO Finite Volume Schemes
  on Unstructured Triangular Meshes}.
\newblock {\em Communications in Computational Physics}, 14:1174--1206, 2013.

\bibitem{ALEMOOD2}
W.~Boscheri and R.~Loub\`ere.
\newblock {High order accurate direct Arbitrary-Lagrangian-Eulerian {ADER}-MOOD
  finite volume schemes for non-conservative hyperbolic systems with stiff
  source terms}.
\newblock {\em Communications in Computational Physics}, 21:271--312, 2017.

\bibitem{ALEMOOD1}
W.~Boscheri, R.~Loub\`ere, and M.~Dumbser.
\newblock {Direct Arbitrary-Lagrangian-Eulerian {ADER}-MOOD finite volume
  schemes for multidimensional hyperbolic conservation laws}.
\newblock {\em Journal of Computational Physics}, 292:56--87, 2015.

\bibitem{boscheri2019high}
Walter Boscheri and Dinshaw~S Balsara.
\newblock {High order direct Arbitrary-Lagrangian-Eulerian (ALE) PNPM schemes
  with WENO Adaptive-Order reconstruction on unstructured meshes}.
\newblock {\em Journal of Computational Physics}, 398:108899, 2019.

\bibitem{boscheri2017arbitrary}
Walter Boscheri and Michael Dumbser.
\newblock {Arbitrary-Lagrangian--Eulerian discontinuous Galerkin schemes with a
  posteriori subcell finite volume limiting on moving unstructured meshes}.
\newblock {\em Journal of Computational Physics}, 346:449--479, 2017.

\bibitem{boscheri2022continuous}
Walter Boscheri, Michael Dumbser, and Elena Gaburro.
\newblock {Continuous finite element subgrid basis functions for Discontinuous
  Galerkin schemes on unstructured polygonal Voronoi meshes}.
\newblock {\em Communications in Computational Physics}, 32:259--298, 2022.

\bibitem{boscheri2015direct}
Walter Boscheri, Rapha{\"e}l Loubere, and Michael Dumbser.
\newblock {Direct Arbitrary-Lagrangian--Eulerian ADER-MOOD finite volume
  schemes for multidimensional hyperbolic conservation laws}.
\newblock {\em Journal of Computational Physics}, 292:56--87, 2015.

\bibitem{BTVC16}
S.~Busto, E.~F. Toro, and M.~E. V\'azquez-Cend\'on.
\newblock Design and analysis of {ADER}-type schemes for model
  advection--diffusion--reaction equations.
\newblock {\em Journal of Computational Physics}, 327:553--575, 2016.

\bibitem{Busto2020}
Saray Busto, Simone Chiocchetti, Michael Dumbser, Elena Gaburro, and Ilya
  Peshkov.
\newblock High order ader schemes for continuum mechanics.
\newblock {\em Frontiers in Physics}, 8:32, 2020.

\bibitem{busto2021staggered}
Saray Busto, Michael Dumbser, and Laura R{\'\i}o-Mart{\'\i}n.
\newblock {Staggered semi-implicit hybrid finite volume/finite element schemes
  for turbulent and non-Newtonian flows}.
\newblock {\em Mathematics}, 9(22):2972, 2021.

\bibitem{CaPa}
Manuel~J. Castro and Carlos Par\'{e}s.
\newblock Well-balanced high-order finite volume methods for systems of balance
  laws.
\newblock {\em J. Sci. Comput.}, 82(2), 2020.

\bibitem{chen2017entropy}
Tianheng Chen and Chi-Wang Shu.
\newblock Entropy stable high order discontinuous {G}alerkin methods with
  suitable quadrature rules for hyperbolic conservation laws.
\newblock {\em Journal of Computational Physics}, 345:427--461, 2017.

\bibitem{chenreview}
Tianheng Chen and Chi-Wang Shu.
\newblock Review of entropy stable discontinuous {G}alerkin methods for systems
  of conservation laws on unstructured simplex meshes.
\newblock {\em CSIAM Transactions on Applied Mathematics}, 1:1--52, 2020.

\bibitem{cheng2019new}
Yuanzhen Cheng, Alina Chertock, Michael Herty, Alexander Kurganov, and Tong Wu.
\newblock A new approach for designing moving-water equilibria preserving
  schemes for the shallow water equations.
\newblock {\em Journal of Scientific Computing}, 80(1):538--554, 2019.

\bibitem{Kur}
Alina Chertock, Shumo Cui, Alexander Kurganov, \c{S}eyma~Nur \"{O}zcan, and
  Eitan Tadmor.
\newblock Well-balanced schemes for the {E}uler equations with gravitation:
  conservative formulation using global fluxes.
\newblock {\em J. Comput. Phys.}, 358:36--52, 2018.

\bibitem{ciallella2022arbitrary}
Mirco Ciallella, Lorenzo Micalizzi, Philipp {\"O}ffner, and Davide Torlo.
\newblock An arbitrary high order and positivity preserving method for the
  shallow water equations.
\newblock {\em Computers \& Fluids}, 247:105630, 2022.

\bibitem{ciallella2022global}
Mirco Ciallella, Davide Torlo, and Mario Ricchiuto.
\newblock {Arbitrary High Order WENO Finite Volume Scheme with Flux
  Globalization for Moving Equilibria Preservation}.
\newblock {\em arXiv preprint arXiv:2205.13315}, 2022.

\bibitem{CDL1}
St{\'e}phane Clain, Steven Diot, and Rapha{\"e}l Loub{\`e}re.
\newblock {A high-order finite volume method for systems of conservation
  laws--Multi-dimensional Optimal Order Detection (MOOD)}.
\newblock {\em Journal of computational Physics}, 230(10):4028--4050, 2011.

\bibitem{clain2011high}
St{\'e}phane Clain, Steven Diot, and Rapha{\"e}l Loub{\`e}re.
\newblock {A high-order finite volume method for systems of conservation
  laws—Multi-dimensional Optimal Order Detection (MOOD)}.
\newblock {\em Journal of computational Physics}, 230(10):4028--4050, 2011.

\bibitem{Colonius}
T.~Colonius, S.~Lele, and P.~Moin.
\newblock {Sound generation in a mixing layer}.
\newblock {\em J. Fluid Mech.}, 330:375--409, 1997.

\bibitem{CDL3}
S.~Diot, R.~Loub{\`e}re, and S.~Clain.
\newblock The {MOOD} method in the three-dimensional case: Very-high-order
  finite volume method for hyperbolic systems.
\newblock {\em International Journal of Numerical Methods in Fluids},
  73:362--392, 2013.

\bibitem{CDL2}
Steven Diot, St{\'e}phane Clain, and Rapha{\"e}l Loub{\`e}re.
\newblock Improved detection criteria for the multi-dimensional optimal order
  detection ({MOOD}) on unstructured meshes with very high-order polynomials.
\newblock {\em Computers \& Fluids}, 64:43--63, 2012.

\bibitem{diot2012improved}
Steven Diot, St{\'e}phane Clain, and Rapha{\"e}l Loub{\`e}re.
\newblock {Improved detection criteria for the multi-dimensional optimal order
  detection (MOOD) on unstructured meshes with very high-order polynomials}.
\newblock {\em Computers \& Fluids}, 64:43--63, 2012.

\bibitem{ADERNSE}
M.~Dumbser.
\newblock Arbitrary high order {PNPM} schemes on unstructured meshes for the
  compressible {{Navier}--{Stokes}} equations.
\newblock {\em Computers \& Fluids}, 39:60--76, 2010.

\bibitem{OsherNC}
M.~Dumbser and E.~F. Toro.
\newblock A simple extension of the {Osher} {{Riemann}} solver to
  non-conservative hyperbolic systems.
\newblock {\em Journal of Scientific Computing}, 48:70--88, 2011.

\bibitem{dumbser2010arbitrary}
Michael Dumbser.
\newblock {Arbitrary high order PNPM schemes on unstructured meshes for the
  compressible Navier--Stokes equations}.
\newblock {\em Computers \& Fluids}, 39(1):60--76, 2010.

\bibitem{dumbser2008unified}
Michael Dumbser, Dinshaw~S Balsara, Eleuterio~F Toro, and Claus-Dieter Munz.
\newblock A unified framework for the construction of one-step finite volume
  and discontinuous {G}alerkin schemes on unstructured meshes.
\newblock {\em Journal of Computational Physics}, 227(18):8209--8253, 2008.

\bibitem{dumbser2005ader}
Michael Dumbser and Claus-Dieter Munz.
\newblock {ADER} discontinuous {G}alerkin schemes for aeroacoustics.
\newblock {\em Comptes Rendus M{\'e}canique}, 333(9):683--687, 2005.

\bibitem{dumbser2009very}
Michael Dumbser and Olindo Zanotti.
\newblock {Very high order PNPM schemes on unstructured meshes for the
  resistive relativistic MHD equations}.
\newblock {\em Journal of Computational Physics}, 228(18):6991--7006, 2009.

\bibitem{Decoriginal}
Alok Dutt, Leslie Greengard, and Vladimir Rokhlin.
\newblock Spectral deferred correction methods for ordinary differential
  equations.
\newblock {\em BIT}, 40(2):241--266, 2000.

\bibitem{FARHAT199361}
Charbel Farhat, Loula Fezoui, and Stéphane Lanteri.
\newblock {Two-dimensional viscous flow computations on the Connecti on
  Machine: Unstructured meshes, upwind schemes and massively parallel
  computations}.
\newblock {\em Computer Methods in Applied Mechanics and Engineering},
  102(1):61--88, 1993.

\bibitem{friedrich2018entropyHP}
Lucas Friedrich, Andrew~R Winters, David C Del~Rey Fern{\'a}ndez, Gregor~J
  Gassner, Matteo Parsani, and Mark~H Carpenter.
\newblock An entropy stable h/p non-conforming discontinuous {G}alerkin method
  with the summation-by-parts property.
\newblock {\em J. Sci. Comput.}, 77(2):689--725, 2018.

\bibitem{ArepoTN}
E.~Gaburro, W.~Boscheri, S.~Chiocchetti, C.~Klingenberg, V.~Springel, and
  M.~Dumbser.
\newblock {High order direct Arbitrary-Lagrangian-Eulerian schemes on moving
  Voronoi meshes with topology changes}.
\newblock {\em Journal of Computational Physics}, 407:109167, 2020.

\bibitem{gaburro2021unified}
Elena Gaburro.
\newblock {A unified framework for the solution of hyperbolic PDE systems using
  high order direct Arbitrary-Lagrangian--Eulerian schemes on moving
  unstructured meshes with topology change}.
\newblock {\em Archives of Computational Methods in Engineering},
  28(3):1249--1321, 2021.

\bibitem{gaburro2021posteriori}
Elena Gaburro and Michael Dumbser.
\newblock {A Posteriori Subcell Finite Volume Limiter for General PNPM Schemes:
  Applications from Gasdynamics to Relativistic Magnetohydrodynamics}.
\newblock {\em Journal of Scientific Computing}, 86(3):1--41, 2021.

\bibitem{gaburro2022high}
Elena Gaburro, Philipp {\"O}ffner, Mario Ricchiuto, and Davide Torlo.
\newblock High order entropy preserving {ADER-DG} scheme.
\newblock {\em Applied Mathematics and Computation}, 440:127644, 2023.

\bibitem{gassner2016well}
Gregor~J Gassner, Andrew~R Winters, and David~A Kopriva.
\newblock A well balanced and entropy conservative discontinuous {G}alerkin
  spectral element method for the shallow water equations.
\newblock {\em Applied Mathematics and Computation}, 272:291--308, 2016.

\bibitem{glaubitz2020stable}
Jan Glaubitz and Philipp {\"O}ffner.
\newblock Stable discretisations of high-order discontinuous {Galerkin} methods
  on equidistant and scattered points.
\newblock {\em Applied Numerical Mathematics}, 151:98--118, 2020.

\bibitem{gomez2023implicit}
Irene G{\'o}mez-Bueno, Sebastiano Boscarino, Manuel~Jes{\'u}s Castro, Carlos
  Par{\'e}s, and Giovanni Russo.
\newblock Implicit and semi-implicit well-balanced finite-volume methods for
  systems of balance laws.
\newblock {\em Applied Numerical Mathematics}, 184:18--48, 2023.

\bibitem{gottlieb1998total}
Sigal Gottlieb and Chi-Wang Shu.
\newblock Total variation diminishing {R}unge-{K}utta schemes.
\newblock {\em Mathematics of computation of the American Mathematical
  Society}, 67(221):73--85, 1998.

\bibitem{hajduk2021monolithic}
Hennes Hajduk.
\newblock Monolithic convex limiting in discontinuous {G}alerkin
  discretizations of hyperbolic conservation laws.
\newblock {\em Computers \& Mathematics with Applications}, 87:120--138, 2021.

\bibitem{han2021dec}
Maria Han~Veiga, Philipp {\"O}ffner, and Davide Torlo.
\newblock {Dec} and {Ader}: similarities, differences and a unified framework.
\newblock {\em Journal of Scientific Computing}, 87(1):1--35, 2021.

\bibitem{Maria}
Maria Han~Veiga, David~A. Velasco-Romero, R\'{e}mi Abgrall, and Romain
  Teyssier.
\newblock Capturing near-equilibrium solutions: a comparison between high-order
  discontinuous {G}alerkin methods and well-balanced schemes.
\newblock {\em Commun. Comput. Phys.}, 26(1):1--34, 2019.

\bibitem{huang2020modeling}
Daniel~Z Huang, Philip Avery, Charbel Farhat, Jason Rabinovitch, Armen
  Derkevorkian, and Lee~D Peterson.
\newblock Modeling, simulation and validation of supersonic parachute inflation
  dynamics during mars landing.
\newblock In {\em AIAA Scitech 2020 Forum}, page 0313, 2020.

\bibitem{huang2019positivity}
Juntao Huang and Chi-Wang Shu.
\newblock Positivity-preserving time discretizations for
  production--destruction equations with applications to non-equilibrium flows.
\newblock {\em Journal of Scientific Computing}, 78(3):1811--1839, 2019.

\bibitem{jund2007arbitrary}
S{\'e}bastien Jund and St{\'e}phanie Salmon.
\newblock {Arbitrary High-Order Finite Element Schemes and High-Order Mass
  Lumping}.
\newblock {\em International Journal of Applied Mathematics \& Computer
  Science}, 17(3):375--393, 2007.

\bibitem{krais2021flexi}
Nico Krais, Andrea Beck, Thomas Bolemann, Hannes Frank, David Flad, Gregor
  Gassner, Florian Hindenlang, Malte Hoffmann, Thomas Kuhn, Matthias Sonntag,
  et~al.
\newblock {FLEXI: A high order discontinuous Galerkin framework for
  hyperbolic--parabolic conservation laws}.
\newblock {\em Computers \& Mathematics with Applications}, 81:186--219, 2021.

\bibitem{kuzmin2020entropy}
Dmitri Kuzmin.
\newblock Entropy stabilization and property-preserving limiters for {$\mathbb
  P_1$} discontinuous {G}alerkin discretizations of scalar hyperbolic problems.
\newblock {\em Journal of Numerical Mathematics}, 29(4):307--322, 2021.

\bibitem{kuzmin2020entropycg}
Dmitri Kuzmin and Manuel~Quezada de~Luna.
\newblock Entropy conservation property and entropy stabilization of high-order
  continuous {G}alerkin approximations to scalar conservation laws.
\newblock {\em Computers \& Fluids}, 213:104742, 2020.

\bibitem{lukavcova2023convergence}
M{\'a}ria Luk{\'a}{\v{c}}ov{\'a}-Medvid’ov{\'a} and Philipp {\"O}ffner.
\newblock {Convergence of discontinuous Galerkin schemes for the Euler
  equations via dissipative weak solutions}.
\newblock {\em Applied Mathematics and Computation}, 436:127508, 2023.

\bibitem{mantri2021well}
Yogiraj Mantri and Sebastian Noelle.
\newblock Well-balanced discontinuous {G}alerkin scheme for 2$\times$2
  hyperbolic balance law.
\newblock {\em Journal of Computational Physics}, 429:110011, 2021.

\bibitem{meister2014unconditionally}
Andreas Meister and Sigrun Ortleb.
\newblock On unconditionally positive implicit time integration for the {DG}
  scheme applied to shallow water flows.
\newblock {\em International Journal for Numerical Methods in Fluids},
  76(2):69--94, 2014.

\bibitem{loredavide}
Lorenzo Micalizzi and Davide Torlo.
\newblock {A new efficient explicit Deferred Correction framework: analysis and
  applications to hyperbolic PDEs and adaptivity}, 2022.

\bibitem{mill}
{R.C.} Millington, {E.F.} Toro, and {L.A.M.} Nejad.
\newblock {\em Arbitrary High Order Methods for Conservation Laws {I}: The One
  Dimensional Scalar Case}.
\newblock PhD thesis, Manchester Metropolitan University, Department of
  Computing and Mathematics, June 1999.

\bibitem{minion2003semi}
Michael~L Minion.
\newblock Semi-implicit spectral deferred correction methods for ordinary
  differential equations.
\newblock {\em Communications in Mathematical Sciences}, 1(3):471--500, 2003.

\bibitem{noelle2007high}
Sebastian Noelle, Yulong Xing, and Chi-Wang Shu.
\newblock High-order well-balanced finite volume {WENO} schemes for shallow
  water equation with moving water.
\newblock {\em Journal of Computational Physics}, 226(1):29--58, 2007.

\bibitem{offner2018stability}
Philipp {\"O}ffner, Jan Glaubitz, and Hendrik Ranocha.
\newblock {Stability of correction procedure via reconstruction with
  summation-by-parts operators for Burgers' equation using a polynomial chaos
  approach}.
\newblock {\em ESAIM: Mathematical Modelling and Numerical Analysis},
  52(6):2215--2245, 2018.

\bibitem{offner2020arbitrary}
Philipp {\"O}ffner and Davide Torlo.
\newblock Arbitrary high-order, conservative and positivity preserving
  {Patankar}-type deferred correction schemes.
\newblock {\em Applied Numerical Mathematics}, 153:15--34, 2020.

\bibitem{perthame1996positivity}
Benoit Perthame and Chi-Wang Shu.
\newblock On positivity preserving finite volume schemes for {E}uler equations.
\newblock {\em Numerische Mathematik}, 73(1):119--130, 1996.

\bibitem{rabenseifner2009hybrid}
Rolf Rabenseifner, Georg Hager, and Gabriele Jost.
\newblock {Hybrid MPI/OpenMP parallel programming on clusters of multi-core SMP
  nodes}.
\newblock In {\em 2009 17th Euromicro international conference on parallel,
  distributed and network-based processing}, pages 427--436, 2009.

\bibitem{ranocha2020relaxation}
Hendrik Ranocha, Mohammed Sayyari, Lisandro Dalcin, Matteo Parsani, and David~I
  Ketcheson.
\newblock Relaxation {R}unge--{K}utta methods: Fully discrete explicit
  entropy-stable schemes for the compressible {E}uler and {N}avier--{S}tokes
  equations.
\newblock {\em SIAM Journal on Scientific Computing}, 42(2):A612--A638, 2020.

\bibitem{ricchiuto2015explicit}
Mario Ricchiuto.
\newblock An explicit residual based approach for shallow water flows.
\newblock {\em Journal of Computational Physics}, 280:306--344, 2015.

\bibitem{ricchiuto2010explicit}
Mario Ricchiuto and Remi Abgrall.
\newblock Explicit {Runge--Kutta} residual distribution schemes for time
  dependent problems: second order case.
\newblock {\em Journal of Computational Physics}, 229(16):5653--5691, 2010.

\bibitem{vort}
Mario Ricchiuto and Davide Torlo.
\newblock Analytical travelling vortex solutions of hyperbolic equations for
  validating very high order schemes.
\newblock 2021.

\bibitem{rio2021massively}
Laura R{\'\i}o-Mart{\'\i}n, Saray Busto, and Michael Dumbser.
\newblock A massively parallel hybrid finite volume/finite element scheme for
  computational fluid dynamics.
\newblock {\em Mathematics}, 9(18):2316, 2021.

\bibitem{Rusanov:1961a}
V.~V. Rusanov.
\newblock {Calculation of Interaction of Non--Steady Shock Waves with
  Obstacles}.
\newblock {\em J. Comput. Math. Phys. USSR}, 1:267--279, 1961.

\bibitem{shu1998essentially}
Chi-Wang Shu.
\newblock Essentially non-oscillatory and weighted essentially non-oscillatory
  schemes for hyperbolic conservation laws.
\newblock pages 325--432. Springer, 1998.

\bibitem{shuosher1}
Chi-Wang Shu and Stanley Osher.
\newblock Efficient implementation of essentially non-oscillatory shock
  capturing schemes.
\newblock {\em Journal of Computational Physics}, 77:439--471, 1988.

\bibitem{spiegel2015survey}
Seth~C Spiegel, HT~Huynh, and James~R DeBonis.
\newblock A survey of the isentropic {Euler} vortex problem using high-order
  methods.
\newblock In {\em 22nd AIAA computational fluid dynamics conference}, page
  2444, 2015.

\bibitem{toro3}
{V.A.} Titarev and {E.F.} Toro.
\newblock {{ADER}}: Arbitrary high order {{Godunov}} approach.
\newblock {\em Journal of Scientific Computing}, 17(1-4):609--618, December
  2002.

\bibitem{titarevtoro}
V.A. Titarev and E.F. Toro.
\newblock {{ADER}} schemes for three-dimensional nonlinear hyperbolic systems.
\newblock {\em Journal of Computational Physics}, 204:715--736, 2005.

\bibitem{titarev2002ader}
Vladimir~A Titarev and Eleuterio~F Toro.
\newblock Ader: Arbitrary high order {G}odunov approach.
\newblock {\em Journal of Scientific Computing}, 17(1-4):609--618, 2002.

\bibitem{toro2009riemann}
E.F. Toro.
\newblock {\em Riemann Solvers and Numerical Methods for Fluid Dynamics: A
  Practical Introduction}.
\newblock Springer, Berlin Heidelberg, 2009.

\bibitem{tsoutsanis2018improvement}
Panagiotis Tsoutsanis, Antonis~F Antoniadis, and Karl~W Jenkins.
\newblock {Improvement of the computational performance of a parallel
  unstructured WENO finite volume CFD code for Implicit Large Eddy Simulation}.
\newblock {\em Computers \& Fluids}, 173:157--170, 2018.

\bibitem{wintermeyer2015entropy}
Niklas Wintermeyer, Andrew~R Winters, Gregor~J Gassner, and David~A Kopriva.
\newblock An entropy stable nodal discontinuous galerkin method for the two
  dimensional shallow water equations on unstructured curvilinear meshes with
  discontinuous bathymetry, 2017.

\bibitem{winters2015comparison}
Andrew~R Winters and Gregor~J Gassner.
\newblock A comparison of two entropy stable discontinuous {G}alerkin spectral
  element approximations for the shallow water equations with non-constant
  topography.
\newblock {\em Journal of Computational Physics}, 301:357--376, 2015.

\bibitem{xing2006high}
Yulong Xing and Chi-Wang Shu.
\newblock High-order well-balanced finite difference {WENO} schemes for a class
  of hyperbolic systems with source terms.
\newblock {\em Journal of Scientific Computing}, 27(1):477--494, 2006.

\bibitem{zhang2011maximum}
Xiangxiong Zhang and Chi-Wang Shu.
\newblock Maximum-principle-satisfying and positivity-preserving high-order
  schemes for conservation laws: survey and new developments.
\newblock {\em Proceedings of the Royal Society A: Mathematical, Physical and
  Engineering Sciences}, 467(2134):2752--2776, 2011.

\end{thebibliography}
\bibliographystyle{plain}

\end{document}